\documentclass[12pt,a4paper,twoside]{article}
\usepackage{amsmath,amsfonts,amssymb,amsthm,graphics}
\usepackage[mathscr]{euscript}
\newtheorem{Theorem}[equation]{Th\'eor\`eme}
\newtheorem{Corollary}[equation]{Corollaire}
\newtheorem{Proposition}[equation]{Proposition}
\newtheorem{Lemma}[equation]{Lemme}

\newtheorem{Remark}[equation]{Remarque}

\newtheorem{Definition}[equation]{D\'efinition}
\newtheorem{Example}[equation]{Exemple}

\def\Section#1{\section{#1}\setcounter{equation}{0}}

\newenvironment{resume}{%
\begin{abstract}
}{\end{abstract}}

\newenvironment{biblio}{%
\begin{thebibliography}
}{\end{thebibliography}}

\font\smc=cmcsc10

\def\bdots{\mathinner{\mkern1mu\raise1pt\hbox{.}\mkern2mu\raise4pt\hbox{.}
           \mkern2mu\raise7pt\vbox{\kern7pt\hbox{.}}\mkern1mu}}

\def\Ex{E^\times}

\def\oF{{\mathfrak o}_F}
\def\oFx{{\mathfrak o}_F^\times}
\def\pF{{\mathfrak p}_F}

\def\oE{{{\mathfrak o}_E}}

\def\oEi{{\mathfrak o}_{E_i}}
\def\oEun{{\mathfrak o}_{E_1}}

\def\fa{{\mathfrak a}}
\def\fb{{\mathfrak b}}

\def\End{{\hbox{\rm End}}}

\def\Hom{{\hbox{\rm Hom}\,}}

\def\dim{\hbox{\rm dim}\,}

\def\det{\hbox{\rm det}\,}

\def\Ind{{\hbox{\rm Ind}}}
\def\ind{{\hbox{\rm ind}}}
\def\Res{{\hbox{\rm Res}}}
\def\cInd{c{\hbox{\rm-Ind}}}

\DeclareMathOperator{\Sp}{Sp}

\DeclareMathOperator{\GL}{GL}
\DeclareMathOperator{\SL}{SL}
\DeclareMathOperator{\SO}{SO}

\def\b{{\beta}}

\def\l{{\lambda}}

\def\L{{\Lambda}}

\def\ext{{\text{ext}\,}}
\def\interieur{{\text{int}}}
\def\calW{{\cal W}}
\def\PM{{}}
\def\PMopp{{-}}
\def\Interieur{{\hbox{\rm Int}}}

\def\today{\number\day\space
\ifcase\month\or
        Janvier\or F\'evrier\or Mars\or Avril\or Mai\or Juin\or
        Juillet\or Ao\^ut\or Septembre\or Octobre\or Novembre\or D\'ecembre\fi
 \space\number\year}

\begin{document}

\title{Repr\'esentation de Weil et $\b$-extensions}

\author{Corinne Blondel
}
\date{8 juillet 2010 }
\maketitle
%
%
%
\begin{resume}
Nous \'etudions les $\b$-extensions dans un groupe classique $p$-adique et   obtenons
une relation entre certaines $\b$-extensions   \`a l'aide d'une repr\'esentation
de Weil. Nous en donnons une application   \`a l'\'etude des points de  r\'eductibilit\'e de certaines induites paraboliques.

\noindent
{\textit{Mathematics Subject Classification (2000):} 22E50}
\end{resume}

\markboth{\smc Corinne Blondel}{Repr\'esentation de Weil et $\b$-extensions}

\section*{Introduction}

Soit $F$ un corps local non archim\'edien de caract\'eristique r\'esiduelle impaire.
 Soit $G_0$ un groupe  symplectique, sp\'ecial orthogonal ou unitaire
 sur $F$ (sur le corps des points fixes d'une involution sur $F$ dans le cas unitaire) et, pour $n \in \mathbb N$, soit $G_n$   le groupe classique de m\^eme nature que $G_0$ ayant
 un sous-groupe parabolique maximal $P_n$  de facteur de Levi $M_n$ isomorphe \`a
 $\GL(n,F) \times G_0$.
Soit $\sigma$ une repr\'esentation supercuspidale irr\'eductible de $G_0$ et  $ \pi $    une repr\'esentation supercuspidale   irr\'eductible
de $\GL(n, F)$.

L'\'etude des points de r\'eductibilit\'e de la repr\'esentation induite
$\ind_{P_n}^{G_n} \pi |\det|^s \otimes \sigma$, o\`u traditionnellement le param\`etre    $s$
est r\'eel, joue un r\^ole important en th\'eorie des repr\'esentations, aussi bien, depuis plus de trente ans, dans la classification des repr\'esentations lisses irr\'educti\-bles de $G_n$ et la d\'etermination de son dual unitaire, que plus r\'ecemment
dans l'\'etude de la correspondance de Langlands. La repr\'esentation induite ci-dessus,   normalis\'ee, est toujours irr\'eductible si $\pi$ n'est pas   autoduale
(au sens appropri\'e  du paragraphe
\ref{rqpr}) ; si $\pi$ est     auto\-duale,
dans presque tous les cas ses points de r\'eductibilit\'e  sont soit $s=0$, soit $s= \pm \frac 1 2 $.
 D'apr\`es les travaux de M\oe glin,   l'ensemble  $\text{Red}(\sigma)$
 des paires $(\pi, s)$, form\'ees d'une repr\'esentation supercuspidale autoduale irr\'eductible
$\pi$ d'un groupe $\GL(n, F)$ et d'un nombre r\'eel $s \ge 1$, telles que la repr\'esentation induite
$\ind_{P_n}^{G_n} \pi |\det|^s \otimes \sigma$ soit r\'eductible d\'etermine le $L$-paquet auquel appartient $\sigma$ et son image  dans la correspondance de Langlands ;
cet ensemble est fini et sa taille est connue.

    Or la th\'eorie des types et paires couvrantes  est un outil puissant d'\'etude de telles  induites, qui permet de traiter le cas, \'echappant aux m\'ethodes traditionnelles,  des repr\'esentations non g\'en\'eriques, et qui surtout devrait aboutir
    \`a une description explicite de $\text{Red}(\sigma)$, donc du $L$-paquet de $\sigma$ et de son image dans la correspondance de Langlands,   en fonction du type (ou d'un type) attach\'e \`a $\sigma$. Ce programme, initi\'e dans l'article fondateur   \cite{BK1}, est techniquement assez difficile puisque pour obtenir $\text{Red}(\sigma)$  il faut partir d'un type $(J_\sigma, \l_\sigma)$ pour $\sigma$, d'un type $(\tilde J_\pi, \tilde \l_\pi)$ pour une repr\'esentation supercuspidale autoduale irr\'eductible
$\pi$ d'un groupe $\GL(n, F)$ et d\'eterminer :
\begin{enumerate}
    \item une paire couvrante $(J, \l)$ de $(\tilde J_\pi\times J_\sigma, \tilde \l_\pi\otimes \l_\sigma)$ dans $G_n$ ;
    \item la structure de l'alg\`ebre de Hecke associ\'ee  $\mathcal H(G_n, \lambda )$ ;
    \item les points de r\'eductibilit\'e {\it complexes} de l'induite ci-dessus via des \'equivalences de cat\'egories transformant l'induction parabolique en induction
    des $\mathcal H(M_n, \tilde \l_\pi\otimes \l_\sigma) )$-modules en $\mathcal H(G_n, \lambda )$-modules (voir le paragraphe
    \ref{rqpr}, essentiellement ind\'ependant du reste de  l'article). Un type pour $\pi$ est insensible \`a la torsion par un caract\`ere non ramifi\'e, c'est pourquoi le param\`etre $s$ varie ici dans $\mathbb C$ ; cela revient \`a \'etudier ensemble deux repr\'esentations induites \`a param\`etre r\'eel (\S \ref{rqpr}).
\end{enumerate}
La faisabilit\'e de ce programme est attest\'ee par \cite{BB} qui le m\`ene \`a son terme
pour les sous-groupes de Levi $\GL(1,F) \times \Sp(2,F)$ et $\GL(2,F)$ de $\Sp(4,F)$.
Dans un travail en cours avec Guy Henniart et Shaun Stevens, nous le mettons en \oe uvre pour d\'eterminer explicitement les $L$-paquets de $\Sp(4,F)$  en termes de types,  \`a partir de la liste des types de repr\'esentations supercuspidales g\'en\'eriques et non g\'en\'eriques
de $\Sp(4,F)$ \'etablie dans \cite{BS} (Gan et Takeda   d\'ecrivent ces $L$-paquets
dans \cite{GT} par une m\'ethode diff\'erente ne permettant pas une telle explicitation).

Dans un groupe classique plus g\'en\'eral et pour une repr\'esentation induite de la forme ci-dessus, la construction d'une paire couvrante (i) a   \'et\'e effectu\'ee dans certains cas particuliers~: le sous-groupe de Levi de Siegel d'un groupe symplectique dans  \cite{B}, le sous-groupe de Levi de Siegel d'un groupe classique dans \cite{GKS} o\`u la structure th\'eorique de l'alg\`ebre de Hecke (ii) est d\'etermin\'ee, une repr\'esentation
supercuspidale autoduale de niveau z\'ero du Levi de Siegel d'un groupe symplectique ou sp\'ecial orthogonal dans \cite{KM} qui donne en outre l'\'etude de r\'eductibilit\'e (iii).
Plus g\'en\'eralement, des
  paires couvrantes accompagn\'ees de la structure th\'eorique de l'alg\`ebre de Hecke associ\'ee sont   fournies par \cite{S5}, sous des hypoth\`eses assez larges,
  en particulier celles  du paragraphe \ref{pc} ci-dessous, que nous supposerons v\'erifi\'ees
  dans la suite de cette introduction : il s'agit essentiellement de demander que
    $\pi$ et $\sigma$ soient attach\'ees \`a des strates suf\-fisamment ``disjointes'' du point de vue de la th\'eorie des types et \`a des caract\`eres semi-simples compatibles (cf. remarque \ref{sigmapi}).

\bigskip

Dans presque tous les cas qui nous int\'eressent ici, la
 structure th\'eorique de  l'alg\`ebre  de Hecke $\mathcal H(G_n, \lambda )$   obtenue est  la m\^eme : il s'agit d'une alg\`ebre de convolution \`a deux g\'en\'erateurs $T_0$ et $T_1$ sur un groupe de Weyl affine qui ne d\'epend que de $\pi$. Le g\'en\'erateur $T_i$, $i=0,1$, est l'image par une injection d'alg\`ebres du g\'en\'erateur d'une alg\`ebre de Hecke
 $\mathcal L_i= \mathcal H(\mathcal G_i, \rho_i)$ de dimension~$2$, l'alg\`ebre d'entrelacements de l'induite parabolique d'une repr\'esentation cus\-pi\-dale  d'un sous-groupe de Levi maximal $\mathcal M_i$   dans un groupe r\'eductif fini. La connaissance de
 $\mathcal G_i$  et $\rho_i$ doit permettre de calculer la relation quadratique satisfaite par $T_i$ \`a l'aide des travaux de Lusztig (voir \cite{KM}). Ces deux relations quadratiques suffisent \`a d\'eterminer les parties r\'eelles des points de r\'eductibilit\'e de la repr\'esentation induite consid\'er\'ee (\S \ref{rqpr}) et donc \`a pr\'edire s'il s'agit d'une situation ``ordinaire'' avec r\'eductibilit\'es de parties r\'eelles $0$ ou $\pm \frac 1 2$, ou d'une situation fournissant un \'el\'ement de l'ensemble $\text{Red}(\sigma)$  cherch\'e.

     Mais le diable est dans les d\'etails :    $\rho_i$ est ``la partie de niveau z\'ero du type $\l$'', notion \`a laquelle on ne peut donner un sens que ``relativement \`a $\mathcal G_i$'' (en prenant un peu de libert\'e avec le vocabulaire). En effet $\rho_i$  est d\'etermin\'ee par une \'ecriture du type $\l$ sous la forme $\kappa_i \otimes \rho_i$ o\`u $\kappa_i$ est une {\it $\b$-extension relativement \`a $\mathcal G_i$} \cite{S5}. Il en r\'esulte que $\rho_i$ n'est a priori connue qu'\`a un caract\`ere (\'eventuellement muni de propri\'et\'es suppl\'ementaires) pr\`es et la structure de $\mathcal L_i$ peut  en  d\'ependre (voir les exemples  \ref{exemplebeta} et  \ref{exemple} et le paragraphe
     \ref{exemples}).
 C'est pourquoi la r\'ealisation du programme ci-dessus doit passer par une analyse approfondie de la notion de $\b$-extension, dont le pr\'esent article constitue une \'etape.

Une description plus pr\'ecise des alg\`ebres  $\mathcal L_i$  qui, redisons-le, d\'eterminent   les parties r\'eelles des points de r\'eductibilit\'e de $\ind_{P_n}^{G_n} \pi |\det|^s \otimes \sigma$,
fait appara\^itre que leur d\'ependance en $\sigma$ r\'eside exclusivement dans la d\'etermination de  $\rho_i$, c'est-\`a-dire dans la d\'etermination des $\b$-extensions attach\'ees \`a la situation,   l'effet de $\sigma$ \'etant simplement de modifier le choix de $\rho_i$ par torsion par un caract\`ere.
   Il est   naturel dans ces conditions   de comparer les parties r\'eelles des points de  r\'eductibilit\'e des
 induites     $\ind_{P_n}^{G_n} \pi |\det|^s \otimes \sigma$ et
    $\ind_{Q_n}^{H_n} \pi |\det|^s  $, o\`u  $H_n$ est le groupe classique de m\^eme nature que $G_0$ ayant
 un sous-groupe parabolique maximal  $Q_n$  de facteur de Levi     isomorphe \`a $\GL(n,F)$.
Comme on vient de le voir,
   cette comparaison se ram\`ene \`a une comparaison  entre repr\'esentations cuspidales d'un m\^eme groupe  fini  pouvant diff\'erer d'une torsion par un  caract\`ere.

    C'est la d\'etermination de ce caract\`ere,  intrins\`equement li\'e \`a la notion de $\b$-extension,  qui est l'enjeu du travail qui suit. Il s'agit en r\'ealit\'e de    comparer   $\b$-extensions dans le groupe $ G_n$ et $\b$-extensions dans son sous-groupe $H_n \times G_0$. La solution est donn\'ee dans le th\'eor\`eme \ref{theoreme}, r\'esultat principal de l'article. Elle est assez simple (si l'on omet d'\'enoncer les hypoth\`eses et notations)~: les  $\b$-extensions consid\'er\'ees dans     $ G_n$ et dans  $H_n \times G_0$ diff\`erent d'un caract\`ere qui est la si\-gna\-ture d'une permutation explicite, l'action par conjugaison   sur un   groupe  fini  attach\'e  \`a la situation.

     Revenons pour terminer \`a la motivation initiale : la  description explicite de $\text{Red}(\sigma)$. On s'attend \`a ce que les repr\'esentations autoduales $\pi$ intervenant dans $\text{Red}(\sigma)$ pr\'esentent une forte affinit\'e, du point de vue des types, avec $\sigma$ (comme dans l'exemple final de \cite{B2} o\`u
     la paire couvrante relative \`a $\pi \otimes \sigma$ est attach\'ee \`a  une strate
     {\it simple} et o\`u l'on obtient un point de r\'eductibilit\'e de partie r\'eelle $1$).
Or nous nous sommes plac\'es
     dans une situation o\`u  au contraire, les strates sous-jacentes \`a   $\pi$ et $\sigma$ sont  suppos\'ees ``disjointes''
   (hypoth\`eses du paragraphe  \ref{pc},  cf   remarque \ref{sigmapi}).  De fait
   c'est dans le cas o\`u $\pi$ est un caract\`ere autodual de $\GL(1,F)$ et o\`u
   la strate semi-simple sous-jacente \`a $\sigma$ n'a pas de composante nulle
   que nos r\'esultats trouvent leur  premi\`ere application. Dans le cas symplectique, chaque alg\`ebre  $\mathcal L_i$
   est l'alg\`ebre  d'entrelacement d'un caract\`ere quadratique ou trivial du sous-groupe de
   Borel de $\SL(2,k_F)$ et l'on sait \`a quel point sa structure d\'epend de ce caract\`ere
   (\S      \ref{exemples}) :  la signature de la permutation du th\'eor\`eme \ref{theoreme}
  d\'etermine ici la ramification du caract\`ere autodual $\pi$ figurant dans  $\text{Red}(\sigma)$.

\bigskip
    L'article est divis\'e en trois parties   :
    mise en place des propri\'et\'es des $\b$-extensions ;
    cons\-truction d'une repr\'esentation de Weil permettant d'obtenir le caract\`ere cherch\'e ;  application \`a l'\'etude de r\'eductibilit\'e.

La premi\`ere partie    est technique, centr\'ee sur les notions fondamentales de bijection cano\-nique et de $\b$-extension d\'efinies par   Stevens dans  \cite{S5}. Apr\`es avoir rappel\'e les d\'efinitions essentielles (\S \ref{facts}), on donne au paragraphe \ref{propbij} une caract\'erisation de la bijection canonique en termes d'entrelacement. Le paragraphe \ref{subord} est consacr\'e \`a la mise en place d'un foncteur de restriction de Jacquet
$\mathbf r_P$  dans une situation l\'eg\`erement plus g\'en\'erale que celle de  \cite{S5} (dont il s'inspire  largement, voir {\it loc. cit.} \S 5.3 et 6.1), o\`u le sous-groupe parabolique $P$ est  attach\'e \`a une d\'ecomposition qui est subordonn\'ee \`a la strate consid\'er\'ee, sans lui \^etre proprement subordonn\'ee.
Le r\'esultat essentiel (proposition \ref{diagcom2}) est  d\'emontr\'e au paragraphe \ref{Jacquet} :  il s'agit de la   compatibilit\'e entre le foncteur $\mathbf r_P$ et la bijection canonique sur laquelle repose la notion de $\b$-extension.

La deuxi\`eme partie, que l'on d\'ecrit ici dans le contexte simplifi\'e ci-dessus,   part d'un caract\`ere semi-simple $\theta$ dans $G_n$ convenablement d\'ecompos\'e par rapport au sous-groupe  $H_n \times G_0$ et relie les repr\'esentations associ\'ees $\eta$ et $\kappa$ dans $G_n$ \`a des objets analogues $\eta'$ et $\kappa'$ dans $H_n \times G_0$
(\S \ref{remarquables}). La relation obtenue est pr\'ecis\'ee au paragraphe 2.2 \`a
l'aide  d'un foncteur de restriction de Jacquet relatif \`a un sous-groupe parabolique de
Levi $\GL(n,F) \times G_0$.
On cons\-truit alors   une repr\'esentation ``de Weil''  $\mathbb W$ telle que la famille finie des repr\'esentations $\kappa$ et celle des $\kappa'$ soient reli\'ees par
$\kappa \simeq \mathbb W \otimes \kappa'$ (\S \ref{WeilRep}). Gr\^ace aux propri\'et\'es de cette repr\'esentation de Weil, on montre, lorsque le caract\`ere $\theta$ est attach\'e \`a un ordre autodual maximal,  que la repr\'esentation  $\kappa$ est une $\b$-extension si et seulement si  $\kappa'$ en est une (\S 2.4).
 L'usage de la compatibilit\'e entre restriction de Jacquet et   bijection canonique  montr\'ee en \ref{Jacquet} aboutit au   th\'eor\`eme \ref{theoreme}, qui d\'ecrit le caract\`ere par lequel diff\`erent les $\b$-extensions dans $G_n$ et     $H_n \times G_0$.

La troisi\`eme partie se rapproche de la motivation initiale en expliquant comment les r\'esultats obtenus s'appliquent \`a l'\'etude de r\'eductibilit\'e.
Nous commen\c cons par r\'esumer les r\'esultats de Stevens \cite{S5} dans le contexte
simplifi\'e d'une d\'ecomposition autoduale en trois morceaux (\S \ref{pc}, ind\'ependant des deux premi\`eres parties). On y voit (corollaire \ref{dernier})  que la notion de $\b$-extension est cruciale pour trouver les g\'en\'erateurs de l'alg\`ebre de Hecke de la paire couvrante consid\'er\'ee.
Nous d\'etaillons ensuite le lien entre les relations quadratiques sa\-tis\-faites par les deux  g\'en\'erateurs d'une alg\`ebre de Hecke correspondant \`a notre situation et les points de r\'eductibilit\'e des repr\'esentations induites associ\'ees (\S \ref{rqpr}, proposition  \ref{formule}). Il reste \`a tirer
  quelques conclusions : la compa\-rai\-son des $\b$-extensions dans $G_n$ et     $H_n \times G_0$
  faite dans le  th\'eor\`eme \ref{theoreme} permet de comparer les alg\`ebres de Hecke
associ\'ees, on passe des g\'en\'erateurs de l'une \`a ceux de l'autre par torsion par des carac\-t\`eres  donn\'es par ce th\'eor\`eme
  (\S \ref{conclusion}, proposition \ref{w0}). Nous terminons par    des exemples dans un groupe symplectique  (\S \ref{exemples}), illustrant la relation entre la notion essentielle de $\b$-extension et les points de r\'eductibilit\'e d'induites paraboliques.

\noindent  {\small
 {\it Remerciements. }  J'aimerais remercier Shaun Stevens pour m'avoir expliqu\'e certaines  subtilit\'es des   caract\`eres semi-simples  et $\b$-extensions et pour ses commentaires   d'une premi\`ere version de ce manuscrit, et Laure Blasco, Guy Henniart, Colette M\oe glin et Shaun Stevens pour des discussions tr\`es stimulantes
 \`a diff\'erents stades de ce travail.
 }

\section*{Notations}

 On travaille dans cet article sur les objets d\'efinis par Shaun Stevens dans \cite{S5}  et des articles ant\'erieurs,  et qui trouvent leur origine dans \cite{BK} et \cite{BK2}. 
 On respecte autant que possible   les notations de \cite{S5},  
  en les simplifiant parfois comme indiqu\'e ci-dessous. 
 
 Soit $F$ un corps local non-archim\'edien 
 de caract\'eristique r\'esiduelle impaire $p$   muni d'un automorphisme 
 not\'e $x \mapsto \bar x$ d'ordre $1$ ou $2$ et de points fixes $F_0$. On note $k_F$ son corps r\'esiduel, de cardinal $q_F=q$.  
 On consid\`ere  le groupe classique $G^+$ des automorphismes 
 d'un espace vectoriel $V$ de dimension finie sur $F$ conservant une forme $\epsilon$-hermitienne (relativement \`a l'automorphisme $x \mapsto \bar x$, avec $\epsilon = \pm 1$) non 
 d\'eg\'en\'er\'ee $h$ sur $V$.  Dans le cas orthogonal on d\'efinit $G$ comme le  groupe sp\'ecial orthogonal des \'el\'ements de $G^+$ de d\'eterminant $1$, dans les autres cas on pose $G=G^+$. 
 On note aussi $g \mapsto \bar g$ l'involution adjointe sur 
 $\End_F(V)$ associ\'ee \`a $h$, de sorte que 
 $G^+ = \{ g \in \GL_F(V) / \bar g = g^{-1} \}$. En g\'en\'eral on notera avec des 
 tildes les objets relatifs au groupe lin\'eaire $\tilde G = \GL_F(V)$, sans tilde 
 ceux relatifs \`a $G$, et avec un exposant $+$ ceux relatifs \`a $G^+$ (dont les  $p$-sous-groupes sont contenus dans $G$).

 Soit $M$ un sous-groupe de Levi de $G$ et $\sigma$ une repr\'esentation supercuspidale irr\'eductible de $M$ ; on note $\mathcal R^{[ \sigma, M]} (G)$ la sous-cat\'egorie pleine des repr\'esentations lisses complexes de $G$ dont chaque sous-quotient  irr\'eductible est sous-quotient  d'une repr\'esentation induite parabolique d'une repr\'esentation de $M$ tordue de $\sigma$ par un caract\`ere non ramifi\'e. 
 
 Soit $J$ un sous-groupe compact ouvert de $G$ et $\lambda$ une repr\'esentation lisse irr\'eductible de $J$ d'espace $E$, on note $\mathcal H(G, \lambda )$  l'alg\`ebre de Hecke de $\lambda$ dans $G$. C'est l'alg\`ebre de convolution (relativement \`a la mesure de Haar sur $G$ donnant \`a $J$ le volume $1$) des fonctions lisses \`a support compact 
 $ f : G \longrightarrow \End_\mathbb C(E) $ v\'erifiant : 
 
\centerline{ $\forall x, y \in J, \forall g \in G, 
 f(xgy) = \lambda(x) f(g) \lambda(y)$. }
 On note  $\text{Mod-}\mathcal H(G, \lambda ) $ la cat\'egorie des modules \`a droite unitaux sur cette alg\`ebre. 
\medskip

L'objet de base de cet article est une {\it strate gauche semi-simple} $[\L,n,0,\b]$ de  $\End_F(V)$ 
\cite[Definition 2.5]{S5} dont on note $V = V^1 \perp \cdots \perp V^l$    la   d\'ecomposition orthogonale de $V$ associ\'ee. Rappelons l'essentiel~: 
\begin{itemize}
	\item 
$\beta$ est un \'el\'ement de $\End_F(V)$ v\'erifiant 
$\bar \beta = -\beta$ et somme de ses restrictions $\beta_i$ \`a $V^i$, \'el\'ements de $\End_F(V^i)$ tels que $F[\beta_i]= E_i$ soit un corps pour  $1 \le i \le l$. De plus $E=F[\beta]$ est la somme $E =  E_1 \oplus \cdots \oplus E_l$ et le centralisateur $B$ de $\beta$ dans 
$\End_F(V)$ est la somme des commutants $B_i$ de $\beta_i$ dans $\End_F(V^i)$, $1 \le i \le l$. 
On note $\tilde G_E = B^\times$,   $  G_E^+ = B^\times \cap G^+$ et $  G_E  = B^\times \cap G $. 
\item 
$\L$ est une suite autoduale de $\oF$-r\'eseaux de $V$ v\'erifiant pour tout $r \in \mathbb Z$ : 
$\L(r) = \oplus_{i=1}^l \L(r) \cap V^i$,  et la suite de r\'eseaux $\L^i$ de $V^i$ d\'efinie par  
$\L^i (r) = \L(r) \cap V^i$ est une suite autoduale de $\oEi$-r\'eseaux, $1 \le i \le l$.  
On note $\mathfrak a_0(\L)$ le fixateur de la suite $\L$ dans $\End_F(V)$,   muni d'une filtration  par les $\mathfrak a_n(\L)$, $n\ge 1$ entier, form\'es des \'el\'ements d\'ecalant la suite de $n$ ; de m\^eme pour   
  $\mathfrak b_0(\L)= \mathfrak a_0(\L)\cap B$ et $\mathfrak b_n(\L)= \mathfrak a_n(\L)\cap B$. 
On note enfin :  
$\widetilde P(\L)= \mathfrak a_0(\L)^\times$, $  P^+(\L)= \widetilde P(\L)\cap G^+$,  
$  P (\L)= \widetilde P(\L)\cap G $, $  P^+(\L_\oE)=   P^+(\L) \cap B$, 
$  P (\L_\oE)=   P (\L) \cap B$, $P_1(\L_\oE)=  P (\L_\oE) \cap (1+\mathfrak b_1(\L))$  et $  P^0(\L_\oE)$ l'image inverse dans   $  P (\L_\oE)$ de la composante neutre du quotient 
 $  P (\L_\oE)/ P_1(\L_\oE)$ \cite[\S 2.1]{S5}. 
\end{itemize}
Les groupes   $\tilde H^1(\b, \L)$, $\tilde J^1(\b, \L)$, $\tilde J (\b, \L)$   d\'esignent les sous-groupes 
    de $\tilde G$ relatifs \`a une 
   strate semi-simple    d\'efinis dans \cite[\S 3]{S4} (apr\`es  \cite{BK} et \cite{BK2}). 
   Pour une strate semi-simple gauche on note  
$H^1(\b, \L)$, $J^1(\b, \L)$, $  J (\b, \L)$ leurs intersections avec $G$, et on note 
 $  J^+ (\b, \L)= \tilde J (\b, \L) \cap G^+ $. On pose enfin 
 $J^0 (\b, \L)= P^0(\L_\oE) J^1(\b, \L)$ -- rappelons que $J  (\b, \L)= P (\L_\oE) J^1(\b, \L)$. 
  On raccourcira  en g\'en\'eral  $H^1(\b, \L)$ etc. 
 en $H^1(\L)$ etc., 
 puisque la  plupart du temps, dans les strates gauches semi-simples $[\L,n,0,\b]$ 
 que l'on consid\'erera,   
  l'\'el\'ement $\b$ sera fix\'e. 
 De m\^eme on utilisera de fa\c con abusive la notation $\oE$, 
 comme dans $\oE$-r\'eseaux  ou $P^+(\L_{\oE})$  \cite[\S 2]{S5}, 
 en omettant les indices ou exposants auxquels le $E$ pourrait pr\'etendre, 
 de fa\c con \`a \'eviter d'\'ecrire   
 ``${\mathfrak o}_{E^{(0)}}$-r\'eseaux''  
 ou ``$P^+(\L_{{\mathfrak o}_{E^{(j)}}})$''.

 Dans toute la suite on d\'esigne par   $[\L,n,0,\b]$,  $[\L',n',0,\b]$, $[\L^m,n^m,0,\b]$,    $[\L^M,n^M,0,\b]$    
 des strates gauches semi-simples ayant en commun l'\'el\'ement $\b$, donc aussi  la   d\'ecomposition de $V$ 
 en $V = V^1 \perp \cdots \perp V^l$, 
 et telles que 
\begin{itemize}
	\item $\fb_0(\L^m)\subseteq  \fb_0(\L)$ ;  $\fb_0(\L)= \fb_0(\L')$ ;   $\fb_0(\L) \subseteq   \fb_0(\L^M)$ ; 
	\item $\fb_0(\L^M)$ est un $\oE$-ordre autodual maximal. 
\end{itemize}
 On fixe un caract\`ere semi-simple gauche $\theta$ de $H^1(\L)$ \cite[\S 3.6]{S4} et on note 
 $\theta'$, $\theta^m$,   $\theta^M$ les caract\`eres de $H^1(\L')$, $H^1(\L^m)$ 
  et $H^1(\L^M)$ qui lui correspondent par transfert
 \cite[Proposition 3.2]{S5}. On note $\eta$ (resp. $\eta'$, $\eta^m$,   $\eta^M$) l'unique repr\'esentation irr\'eductible de 
 $J^1(\L)$ (resp.  $J^1(\L')$,  $J^1(\L^m)$,   $J^1(\L^M)$) contenant $\theta$ (resp. $\theta'$, $\theta^m$,  $\theta^M$). 
 
 Si $[\L^\star,n^\star,0,\b]$, o\`u $\star$ d\'esigne n'importe quel symbole, est une autre telle strate, on utilisera de la m\^eme fa\c con les notations $\theta^\star$, $\eta^\star$ etc., ou bien aussi $\theta(\L^\star)$, $\eta(\L^\star)$ etc. 

\Section{Quelques propri\'et\'es des $\beta$-extensions}\label{S1}

  Pour la commodit\'e de la r\'edaction on rappelle avant de commencer un fait bien connu et  \'el\'ementaire  mais tr\`es utile.
 Soit $H$ et $H^\prime$ des sous-groupes ouverts compacts d'un groupe l.c.t.d. $K$ et 
 $\rho$ et $\rho^\prime$ des repr\'esentations de dimension finie de $H$ et $H^\prime$ respectivement. 
 Alors   
 \begin{equation} \label{obvious}
 \Hom_{H \cap H^\prime} ( \rho, \rho^\prime) \ne \{0\} 
 \Longrightarrow 
 \Hom_{K} ( \Ind_H^K \rho, \Ind_{H^\prime}^K  \rho^\prime) \ne \{0\} . 
\end{equation}

\subsection{Faits essentiels}\label{facts} 

Rappelons les r\'esultats suivants de \cite{S5}, qui jouent un r\^ole crucial dans la suite. 

 \begin{Proposition}[\bf  \cite{S5} Proposition 3.7] \label{etamM}    
Il existe une unique repr\'esentation irr\'eductible  $\eta(\L^m, \L)$ de 
$J^1(\L^m, \L)=  P_1(\L^m_\oE)J^1(\L)$ v\'erifiant : 
\begin{enumerate}
	\item  $\eta(\L^m, \L)_{|J^1(\L)} = \eta$ ; 
	\item Pour toute suite de $\oE$-r\'eseaux $\L^{\prime\prime} $ v\'erifiant $\fb_0(\L^{\prime\prime})= \fb_0(\L^m)$ et $\fa_0(\L^{\prime\prime})\subseteq \fa_0(\L)$ on a 
	$$ 
\Ind_{J^1(\L^{\prime\prime} )}^{  P_1(\L^{\prime\prime} )} \, \eta^{\prime\prime} 
\;  \simeq 
\; 
\Ind_{P_1(\L_{\oE}^m) J^1(\L )}^{  P_1(\L^{\prime\prime})} \,  \eta(\L^m, \L).   
	$$
\end{enumerate}
\end{Proposition}

\begin{proof} Il ne s'agit que d'\'etendre les notations de \cite[Proposition 3.7]{S5},   
qui suppose  $\fa_0(\L^m)\subseteq \fa_0(\L)$, 
au cas g\'en\'eral  $\fb_0(\L^m)\subseteq \fb_0(\L)$. Cette extension est implicite dans \cite{S5}  
puisque  l'auteur d\'efinit $\eta(\L^{\prime\prime}, \L)$ ($\eta_{m,M}$ dans les notations de \cite{S5})
et remarque que cette repr\'esentation   et son  groupe de d\'efinition $P_1(\L^{\prime\prime}_\oE)J^1(\L)$ 
ne d\'ependent que de  $\fb_0(\L^{\prime\prime})$ et non de la suite $\L^{\prime\prime}$ elle-m\^eme. 
Il suffit donc ici de d\'efinir  $\eta(\L^m, \L)   =  \eta(\L^{\prime\prime}, \L)$.  On rappelle que des suites $\L^{\prime\prime}$ v\'erifiant les conditions de l'\'enonc\'e existent toujours  \cite[Lemme 2.8]{S5}. 
 \end{proof}

 \begin{Proposition}[\bf  \cite{S5} Lemma 4.3] \label{bijcan}  
Il y a une bijection canonique  $\mathfrak B_{\L^m \, \L}  $ 
entre l'ensemble des prolongements $\kappa^m$ de $\eta^m$ \`a 
$J^+(\L^m)$ et l'ensemble des prolongements $\kappa$ de $\eta(\L^m, \L)$ \`a 
$P^+(\L^m_\oE)J^1(\L)$. Si $\fa_0(\L^m)\subseteq  \fa_0(\L)$ 
cette bijection associe \`a $\kappa^m$ l'unique repr\'esentation $\kappa$ de $P^+(\L^m_\oE)J^1(\L)$ telle que $\kappa$ et $\kappa^m$ induisent des repr\'esentations (irr\'eductibles) \'equivalentes de 
$P^+(\L^m_\oE)P^1(\L^m)$. 
 \end{Proposition}

On notera parfois simplement $\mathfrak B$ la bijection canonique d\'efinie ci-dessus. C'est cette bijection qui permet de d\'efinir la notion de $\beta$-extension dans \cite{S5} : 

\begin{Definition}[\bf  \cite{S5} Theorem 4.1, Definition 4.5]\label{betaextension}

\begin{enumerate}
	\item {\em Cas maximal.} Une repr\'esen\-ta\-tion $ \ \kappa^M \  $ de $J^+(\L^M)$ est une {\em $\beta$-extension} de
	$\eta^M$ si $\kappa^M$ est un prolongement de 
	 $\eta(\L^m, \L^M)$, pour une  suite autoduale $\L^m$ de $\oE$-r\'eseaux telle que 
	 $\fb_0(\L^m)$ soit un $\oE$-ordre autodual minimal contenu dans $\fa_0(\L^M)$.  La condition est alors v\'erifi\'ee pour toute telle suite $\L^m$.  Il existe des $\beta$-extensions et deux d'entre elles diff\`erent d'un caract\`ere de $P^+(\L_\oE^M)/ P_1(\L_\oE^M)$ trivial sur   les sous-groupes unipotents. 
	 \item {\em Cas g\'en\'eral.} Soit $\kappa^M$    une   $\beta$-extension  de
	$\eta^M$ \`a $J^+(\L^M)$. La {\em $\beta$-extension  de
	$\eta $ \`a $J^+(\L )$ relative \`a $\L^M$ et compatible \`a $\kappa^M$} est la repr\'esentation $\kappa$ de $J^+(\L)$ telle  que 
	$\mathfrak B_{\L  \, \L^M} (\kappa)$ soit la restriction de $ \kappa^M$ \`a $P^+(\L_\oE)J^1(\L^M) $. 
	
	On dira que  $\kappa$ est une $\beta$-extension  de
	$\eta $ \`a $J^+(\L )$ si elle en est une relativement \`a  une suite $\L^M$ convenable. 
\end{enumerate}
\end{Definition}

 Les faits suivants sont des cons\'equences imm\'ediates, voire tautologiques, de \cite[\S4]{S5}. 
 \begin{Lemma}\label{Faits}
\begin{enumerate}
	\item Si $\kappa$ correspond \`a $\kappa'$ par la bijection canonique 
	$\mathfrak B_{\L \, \L'}  $, alors 
	$\kappa$ est une $\b$-extension de $\eta$ relative \`a $\L^M$ 
	(compatible \`a $\kappa^M$) si et seulement si 
		$\kappa'$ est une $\b$-extension de $\eta'$ relative \`a $\L^M$ 
		(compatible \`a $\kappa^M$).  
\item La bijection canonique $\mathfrak B_{\L^m \, \L}  $,   
  entre prolongements de $\eta^m$ \`a $J^+(\L^m)$ 
	et prolongements de  $\eta(\L^m, \L)$ \`a $P^+(\L^m_{\oE})  J^1(\L)$,  est compatible \`a la torsion par les caract\`eres du groupe $P^+( \L^m_{\oE})/ P^1( \L^m_{\oE})$. 
\end{enumerate}
\end{Lemma} 

\begin{proof}
Le premier fait provient de ce que la compos\'ee de deux bijections cano\-niques est une bijection canonique (voir aussi la d\'emonstration 
de \cite[Lemma 4.3]{S5}).  Le second 
 est \'el\'ementaire. 
\end{proof}

\subsection{Une propri\'et\'e caract\'eristique de la bijection canonique}\label{propbij}

On d\'emontre dans ce paragraphe la propri\'et\'e suivante : 

 \begin{Proposition} \label{intertwining}  
 Supposons que $\fa_0(\L^m)\subseteq  \fa_0(\L)$.
 Soit $\kappa^m$ un prolongement   de $\eta^m$ \`a $J^+(\L^m)$ 
 et $\kappa $ un prolongement   de $\eta(\L^m, \L) $ \`a $P^+(\L^m_{\oE})  J^1(\L)$. 
 Alors 
$$ \kappa =  \mathfrak B_{\L^m \, \L}    (\kappa^m) \quad \iff \quad 
\Hom_{J^+(\L^m) \cap P^+(\L^m_{\oE})  J^1(\L)}(\kappa^m, \kappa) \ne \{ 0 \} $$ 
 \end{Proposition} 
 
\begin{proof}
Elle s'appuie sur deux lemmes qu'on \'etablit d'abord. 

\begin{Lemma}\label{projecteur}
Il existe un vecteur $v$ de l'espace de $\kappa^m$ tel que 
$$ \oint_{H^1(\L) \cap J^+(\L^m)} \theta(x^{-1}) \ \kappa^m(x) \ v \, dx  \ne 0 . 
$$
 \end{Lemma} 

\begin{proof}
Comme $H^1(\L)\subseteq P_1(\L)\subseteq P_1(\L^m)$, l'int\'egrale calcul\'ee est en fait 
$$ X(v) = \oint_{H^1(\L) \cap J^1(\L^m)} \theta(x^{-1}) \ \eta^m(x) \  v  \, dx    . 
$$
L'application $v \mapsto X(v) $ est un projecteur de l'espace de $\eta^m$ sur son sous-espace isotypique de type $\theta$ sous 
$H^1(\L) \cap J^1(\L^m)$. Il s'agit donc de montrer que celui-ci est non nul, et comme 
l'induite de $\theta^m$ \`a $J^1(\L^m)$ est multiple de $\eta^m$ il suffit d'\'etablir 
$$\Hom_{H^1(\L) \cap J^1(\L^m)} (\theta, \Ind_{H^1(\L^m)}^{J^1(\L^m)}\theta^m) \ne \{ 0\}. 
$$
On d\'eveloppe la restriction de $\Ind_{H^1(\L^m)}^{J^1(\L^m)}\theta^m $  \`a $H^1(\L)\cap J^1(\L^m)$ selon la formule de Mackey ; le premier terme est 
$\Ind_{H^1(\L)\cap H^1(\L^m)}^{H^1(\L)\cap J^1(\L^m))}\theta^m $, or par r\'eciprocit\'e de Frobenius on a bien 
$$\Hom_{H^1(\L) \cap J^1(\L^m)} (\theta, \Ind_{H^1(\L)\cap H^1(\L^m)}^{H^1(\L)\cap J^1(\L^m))}\theta^m) 
\simeq  \Hom_{H^1(\L) \cap H^1(\L^m)} (\theta,  \theta^m)   \ne \{ 0\}
$$ puisque $\theta^m$ est le transfert de $\theta$   \cite[Proposition 3.2]{S5}. 
\end{proof}

\begin{Lemma}\label{mult} 
La multiplicit\'e de $\eta$ dans $ \  \Ind_{P^+(\L^m_{\oE})  J^1(\L)}^{P^+(\L^m_{\oE}) P_1(\L^m)} \ \kappa \  $ est $1$. 
\end{Lemma}

\begin{proof}
De nouveau on a $J^1(\L)\subseteq P_1(\L)\subseteq P_1(\L^m)$, il suffit donc de consid\'erer la res\-tric\-tion de la repr\'esentation induite 
\`a $P_1(\L^m)$, soit $\Phi = \Ind_{P_1(\L^m_{\oE})  J^1(\L)}^{  P_1(\L^m)} \  \kappa_{|P_1(\L_{\oE}^m) J^1(\L )} $, qui est irr\'eductible  \cite[Proposition 3.7]{S5}.
On a 
$$\Hom_{J^1(\L)}(\eta, \Phi) \simeq \Hom_{  P_1(\L^m)}( \Ind_{ P_1(\L_{\oE}^m) J^1(\L )}^{  P_1(\L^m)}\Ind_{  J^1(\L)}^{  P_1(\L_{\oE}^m) J^1(\L )} \eta, \  \Phi) 
$$
Remarquons que  $\Ind_{  J^1(\L)}^{  P_1(\L_{\oE}^m) J^1(\L )} \eta \simeq \eta(\L^m, \L) \otimes \Ind_{  P^1(\L_{\oE})}^{  P_1(\L_{\oE}^m)  } 1 
\simeq \oplus_{\chi \in Z}   m_\chi \,  \eta(\L^m, \L) \otimes  
   \chi  $, o\`u $Z$ d\'esigne l'ensemble des classes d'isomorphisme 
de facteurs irr\'eductibles de $ \,  \Ind_{  P^1(\L_{\oE})}^{  P_1(\L_{\oE}^m)  } \, 1  \,  $
et $m_\chi$ la multiplicit\'e de $\chi \in Z$ dans cette induite. 

Pour  $\chi \in Z$,  
$I(\chi )= \Ind_{ P_1(\L_{\oE}^m) J^1(\L )}^{  P_1(\L^m)} 
 \eta(\L^m, \L) \otimes  \chi$ est irr\'eductible  car l'entrelacement dans $ P_1(\L^m) $ de $\eta(\L^m, \L) \otimes  \chi $  est contenu dans celui de $\eta$ : 
 $J^1(\L) G_E J^1(\L) \cap P_1(\L^m)= P_1(\L_{\oE}^m) J^1(\L )$. Ainsi $I(\chi)$ et $\Phi$ sont entrelac\'ees si et seulement si elles sont isomorphes, si et seulement si $\chi=1$ par d\'efinition de $\eta(\L^m, \L)$. Par ailleurs $m_1$ vaut $1$ par Frobenius. 
\end{proof}

Montrons maintenant   la proposition. Par d\'efinition, les repr\'esentations $\kappa$ et $\kappa^m$ se corres\-pondent par la bijection canonique si et seulement si il existe un isomorphisme $\mathcal E$ entre les repr\'esentations induites :  
$$
\Ind_{P^+(\L_{\oE}^m) J^1(\L )}^{P^+(\L_{\oE}^m) P_1(\L^m)} \kappa 
\quad \stackrel{\mathcal E}{\longrightarrow} \quad  \Ind_{J^+(\L^m)}^{P^+(\L_{\oE}^m) P_1(\L^m)} \kappa^m 
.
$$

La condition $\Hom_{J^+(\L^m) \cap P^+(\L^m_{\oE})  J^1(\L)}(\kappa^m, \kappa) \ne \{ 0 \}$ est suffisante par \ref{obvious} (les induites sont irr\'eductibles). Pour montrer qu'elle est n\'ecessaire   nous avons besoin des lemmes ci-dessus. 
On compose $\mathcal E$ avec l'injection canonique $\mathcal J$ de $\kappa$ dans son induite \`a droite, avec la projection canonique $\mathcal P$ 
de l'induite de $\kappa^m$ sur $\kappa^m$ \`a gauche : on obtient un entrelacement $\mathcal P \circ \mathcal E \circ \mathcal J$ de 
$\kappa$ dans $\kappa^m$ dont on va montrer qu'il est non nul. 

D'apr\`es le lemme \ref{mult}, l'image de $\kappa$ par $\mathcal J$ est la composante isotypique de type $\theta$ de l'induite de $\kappa$, qui s'envoie injectivement par $\mathcal E$ sur la composante isotypique de type $\theta$ de l'induite de $\kappa^m$. Il suffit alors de montrer 
que $\mathcal P$ n'est pas identiquement nulle sur cette composante. 

L'espace de l'induite de $\kappa^m$ peut se voir comme l'espace des fonctions $f$ de $P^+(\L_{\oE}^m) P_1(\L^m)$ dans l'espace de $\kappa^m$ v\'erifiant $f(gx) = \kappa^m(g) f(x)$ pour $ g \in J^+(\L^m)$ et on a $\mathcal P (f) = f(1)$. L'application 
$f \mapsto  f^\theta $ d\'efinie par $f^\theta(y) = \oint_{H^1(\L)} \theta (x^{-1}) f(yx) dx $ 
est un projecteur sur la composante de type $\theta$ sous $H^1(\L)$. Prenons $f$ dans l'image canonique de $\kappa^m$, c'est-\`a-dire de support 
$J^+(\L^m)$. On a 
$$
\mathcal P (f^\theta) = f^\theta (1) = \oint_{H^1(\L)} \theta (x^{-1}) f(x) dx = 
\oint_{H^1(\L)\cap J^+(\L^m)} \theta (x^{-1}) \kappa^m(x) f(1) dx.
$$ 
L'existence de $f(1)$ tel que cette int\'egrale soit non nulle est donn\'ee par le lemme 
\ref{projecteur}.
\end{proof}

\begin{Corollary} \label{intertwiningeta}  
 Si $\fa_0(\L^m)\subseteq  \fa_0(\L)$, alors 
 $\eta(\L^m, \L)$ est l'unique   prolongement   de $\eta $ \`a $P_1(\L^m_{\oE})  J^1(\L)$ v\'erifiant  
$$ 
\Hom_{J_1(\L^m) \cap P_1(\L^m_{\oE})  J^1(\L)}(\eta^m, \eta(\L^m, \L)) \ne \{ 0 \} .$$ 
 \end{Corollary}

\subsection{D\'ecompositions  subordonn\'ees et restriction de Jacquet}\label{subord} 

Rappelons les d\'efinitions de \cite[\S 5]{S5}. Soit  $V = \oplus_{j=-k}^k W^{(j)}$ 
  une d\'ecomposition 
  de $V$   autoduale   -- c'est-\`a-dire que l'orthogonal dans $V$  de $W^{(j)}$, pour $-k \le j \le k$, est $ \ \oplus_{s \ne -j} \;  W^{(s)}$ --   telle que, pour tout $j$,  $-k\le j \le k$  :  
\begin{itemize}
	\item $W^{(j)} = \oplus_{i=1}^l W^{(j)}\cap V^i$ ;
	\item   
	$W^{(j)}\cap V^i$ est  un 
	$E_{i}$-sous-espace de $ V^{i}$ pour tout $i$, $1\le i \le l$.  
\end{itemize}

Une telle   d\'ecomposition est  {\it subordonn\'ee} \`a 
la strate $[\L,n,0,\b]$ si on a   pour tout entier $r$~:  
$\L(r) = \oplus_{j=-k}^k \L(r) \cap W^{(j)}$;  on notera 
$\L^{(j)}$ la suite de r\'eseaux de $W^{(j)}$ d\'efinie par 
$\L^{(j)} (r) = \L(r) \cap W^{(j)}$. 
Elle est  {\it proprement subordonn\'ee} \`a 
  $[\L,n,0,\b]$ si   de plus, pour tout $i$, $1\le i \le l$,  chaque saut 
  de la suite de r\'eseaux  $\L^i(r)  $ intervient dans un et un seul des sous-espaces  $W^{(j)}\cap V^i$
  (voir \cite[\S 7]{BK}  et \cite[\S 5]{S5}).

Fixons  une  d\'ecomposition  autoduale $V = \oplus_{j=-k}^k W^{(j)}$   de $V$   proprement subordonn\'ee  \`a 
  $[\L^m,n,0,\b]$~; elle    
  est alors   subordonn\'ee \`a $[\L,n,0,\b]$. 
Soit  $M $  le sous-groupe de Levi de $G^+$ fixant la d\'ecomposition 
$V = \oplus_{j=-k}^k W^{(j)}$ et $P$  un sous-groupe parabolique de $G^+$ de facteur de Levi $M$ et radical unipotent $U$ ; on note $P^{-}$ et $U^{-}$ les oppos\'es de $P$ et $U$ par rapport \`a $M$. Il est montr\'e dans  
  \cite[\S 5.3]{S5}  que sous l'hypoth\`ese de subordination,  les sous-groupes 
  $H^1(\L)$ et $J^1(\L)$ ont une d\'ecomposition d'Iwahori par rapport \`a $(P,M)$ et  l'on peut d\'efinir :

\begin{itemize}
	\item un prolongement $\theta_P$ de $\theta$ \`a 
	$H^1_P(\L)= H^1(\L) (J^1(\L)\cap U)$, trivial sur $J^1(\L)\cap U$ ; 
	\item l'unique repr\'esentation irr\'eductible $\eta_P$ de $J^1_P(\L)= H^1(\L) (J^1(\L)\cap P)$ contenant   $\theta_P$ ; sa restriction \`a 	$H^1_P(\L)$ est   multiple de $\theta_P$. Elle    v\'erifie en outre : 
	
	$\eta_P$ est la repr\'esentation de $J^1_P(\L)$ dans les $J^1(\L)\cap U$-invariants de $\eta$ obtenue par restriction de $\eta$ \`a  $J^1_P(\L)$ et $\eta \simeq \Ind_{J^1_P(\L)}^{J^1(\L)}  \ \eta_P$.  
\end{itemize}

Dans les notations de  \cite[\S 5.3]{S5}   on a 
$H^1(\L) \cap M = H^1(\L^{(0)})\times\prod_{j=1}^k \widetilde H^1(\L^{(j)})$ ; 
le caract\`ere semi-simple $\theta$ est trivial sur $H^1(\L) \cap U$ et
$H^1(\L) \cap U^{-} $ et a pour restriction \`a $H^1(\L) \cap M$ un produit 
$\theta^{(0)} \otimes \bigotimes_{j=1}^k  ( \tilde\theta^{(j)} )^2$ o\`u $\theta^{(0)}$ 
est un caract\`ere semi-simple gauche et les $ \tilde\theta^{(j)}$, pour $j\ne 0$, sont des caract\`eres 
semi-simples. De m\^eme 
${\eta_P}_{| J^1(\L) \cap M}$ est une 
repr\'esentation de 
$J^1(\L) \cap M = J^1(\L^{(0)})\times\prod_{j=1}^k \widetilde J^1(\L^{(j)})$  
de la forme 
$\eta^{(0)} \otimes \bigotimes_{j=1}^k \tilde\eta^{(j)}$, o\`u  
  $\eta^{(0)}$ est l'unique repr\'esentation irr\'eductible de $J^1(\L^{(0)})$ 
	contenant $\theta^{(0)}$ 
et $\tilde\eta^{(j)}$, pour $j\ne 0$, est l'unique repr\'esentation irr\'eductible de $\widetilde J^1(\L^{(j)})$ 
	contenant $ (\tilde\theta^{(j)})^2$.

Sous l'hypoth\`ese de subordination {\it propre}, ici valide pour $\L^m$, on a de tels r\'esultats jusqu'au niveau de $J^+$   \cite[\S 5.3]{S5}  : le groupe $J^+(\L^m)$ a aussi une d\'ecomposition d'Iwahori et tout prolongement $\kappa^m $  de $\eta^m$ \`a  $J^+(\L^m)$ est induit de la repr\'esentation $\kappa^m_P$ de $J^+_P(\L^m)=  H^1(\L^m) (J^+(\L^m)\cap P)$ obtenue en prenant les $J^1(\L^m)\cap U$-invariants de $\kappa^m$ ; la repr\'esentation  $\kappa^m_P$  prolonge $\eta^m_P$. Noter qu'alors  $J^+(\L^m)\cap U = J^1(\L^m)\cap U$. 
  
En l'absence  
de subordination  {\it propre}, on peut n\'eanmoins utiliser les techniques de $J^1\cap U$-invariants 
pour \'etudier des prolongements de $\eta$ \`a un sous-groupe convenable de $J^+(\L)$. Gardons les 
notations  $J_P^+(\L)=  H^1(\L) (J^+(\L)\cap P)$, $\kappa_P$ prolongement de 
$\eta_P$ \`a   $J_P^+(\L)$, dans ce contexte \'elargi ; on v\'erifie facilement (crit\`ere de Mackey) que 
$\Ind_{J_P^+(\L)}^{J^+(\L)}\kappa_P$ n'est pas irr\'eductible si $J^+(\L)$ n'est pas \'egal \`a 
$J^1(\L)J_P^+(\L)$.  Mais ce sous-groupe lui-m\^eme 
est digne d'int\'er\^et. Il a une d\'ecomposition d'Iwahori relative \`a $(M,  P)$ et :  
\begin{Lemma}\label{invariantsbis} 
 Pour tout prolongement $\kappa $ de  $\eta $ \`a $$J^1(\L)J_P^+(\L)= (J^1(\L)\cap U^-) (J^+(\L)\cap M) 
 (J^+(\L) \cap U) $$ 
 le sous-espace des invariants par  $J^1(\L ) \cap U $ d\'efinit une repr\'esentation 
 $\kappa_{P }$ de  $J^+_{P }(\L )= (P^+ (\L _\oE) \cap P ) J_P^1(\L )$ prolongeant $\eta_{P }$  
 et telle que  $$\kappa \simeq \Ind_{J^+_{P}(\L )}^{J^1(\L)J_P^+(\L)}\kappa_{P}.$$ 
 \end{Lemma} 
 
\begin{proof} Il suffit de v\'erifier   l'\'egalit\'e suivante : 
\begin{equation}\label{JP}
J^+(\L) \cap P = (P^+ (\L _\oE) \cap P ) \  (J^1(\L )\cap P).
\end{equation}
Comme $J^1(\L )$ a une d\'ecomposition d'Iwahori par rapport \`a $(M,P)$, tout \'el\'ement  de $J^+(\L)$ peut s'\'ecrire 
$j = x j_-^{-1} j_P$ avec $x \in P^+ (\L _\oE)$, 
$j_P \in J^1(\L )\cap P $ et $j_- \in J^1(\L )\cap U^-$. Alors $j$ appartient \`a $P$ si et seulement si $ x j_-^{-1}$ 
appartient \`a $P$. Ecrivons dans ce cas $x= umj_-$ avec $u \in U$, 
$m \in M$. Par unicit\'e de la d\'ecomposition d'Iwahori 
on voit que, puisque $x$ commute \`a $F[\beta]^\times$ qui est contenu dans $M$, alors $j_-$ commute aussi \`a $F[\beta]^\times$. Ainsi  $j_-$ appartient \`a $P_1 (\L _\oE) \cap U^-$ et
$ x j_-^{-1}   $ appartient \`a $P^+ (\L _\oE) \cap P$, c.q.f.d.

Ce point  acquis, la d\'emonstration 
 reprend sans difficult\'e celle de \cite[\S 5.3]{S5}  ou de \cite[\S 7.2]{BK}. 
Noter que dans cette situation d'une d\'ecomposition subordonn\'ee \`a  $[\L,n,0,\b]$ sans lui \^etre proprement subordonn\'ee, on n'a aucune raison d'avoir \'egalit\'e entre $J^+(\L ) \cap U $ et $J^1(\L ) \cap U $. 
\end{proof}

 \begin{Corollary}\label{invariants} Pla\c cons-nous dans les hypoth\`eses suivantes : 
 {\em 
 \begin{equation}\label{hypotheses}
\left\{ \aligned 
 &V = \oplus_{j=-k}^k W^{(j)}    \text{ est proprement subordonn\'ee  \`a } 
  [\L^m,n,0,\b] ;  
  \\ 
  &P^+ (\L_\oE^m) = (P^+ (\L _\oE) \cap P ) \ P_1(\L_\oE) . 
  \endaligned \right.
 \end{equation}
 }
 Alors   le sous-groupe $J^+_{\L^m \L} = P^+(\L_\oE^m) J^1(\L )$   co\"\i ncide 
 avec $ \ (J^+ (\L ) \cap P ) \   J^1(\L ) \  $
   et poss\`ede une d\'ecomposition d'Iwahori par rapport \`a $(P , M)$. 
   
 Pour tout prolongement $\kappa $ de  $\eta $ \`a $P^+(\L_{\oE}^m) J^1(\L )$ 
 le sous-espace des invariants par   $J^1(\L ) \cap U$ 
 d\'efinit une repr\'esentation 
 $\kappa_{P }$ de  $J^+_{P }(\L )= (P^+ (\L _\oE) \cap P ) J_P^1(\L )$ prolongeant $\eta_{P }$  
 et telle que  $$\kappa \simeq \Ind_{J^+_{P}(\L )}^{P^+(\L_{\oE}^m) J^1(\L )}\kappa_{P}.$$ 
 \end{Corollary} 
 
Bien entendu on s'int\'eressera particuli\`erement aux prolongements de $\eta$ prolongeant de plus $\eta(\L^m, \L)$. Or les techniques de $J^1 \cap U$-invariants 
permettent, dans nos hypoth\`eses, de pr\'eciser la structure de cette repr\'esentation (comme dans \cite{S5}, d\'emonstration de la Proposition 6.3). 
Rappelons la d\'ecomposition ${\eta_P}_{| J^1(\L) \cap M}\simeq \eta^{(0)} \otimes \bigotimes_{j=1}^k \tilde\eta^{(j)}$ et notons : 

-- $\eta({\L^m}^{(0)}, \L^{(0)})$ le prolongement  de $\eta^{(0)}$ 
\`a $J^1({\L^m}^{(0)}, \L^{(0)})$ donn\'e par la proposition \ref{etamM} ; 

-- $\widetilde \eta({\L^m}^{(j)}, \L^{(j)})$ le prolongement  de $\widetilde \eta^{(j)}$
\`a $\widetilde J^1({\L^m}^{(j)}, \L^{(j)})$ donn\'e par \cite[Proposition 3.12]{S5}.

\begin{Proposition}\label{etamMP}
Sous les hypoth\`eses (\ref{hypotheses}), 
soit $\eta_P(\L^m, \L)$ la repr\'esentation du groupe 
$J^1_P(\L^m, \L)= (P_1(\L_\oE^m)\cap P) J^1_P(\L)$ 
obtenue par restriction de $\eta(\L^m, \L)$ au sous-espace de ses $J^1(\L) \cap U$-invariants. On a 
$\eta(\L^m, \L) \simeq \Ind_{J^1_{P}(\L^m,\L )}^{ J^1(\L^m, \L )} \ \eta_P(\L^m, \L)$. 
De plus $\eta_P(\L^m, \L)$ est d\'efinie par : 
$$
\aligned
\eta_P(\L^m, \L)  &= \eta_P  \  \text{ sur }  J^1_P(\L), 
\\
\eta_P(\L^m, \L)  &= 1 \  \text{ sur }   P^1(\L_\oE^m) \cap U,  
\\
\eta_P(\L^m, \L)  &= \eta({\L^m}^{(0)}, \L^{(0)})  \otimes \bigotimes_{j=1}^k  
\widetilde \eta({\L^m}^{(j)}, \L^{(j)}) 
\\ &\qquad \text{ sur }   J^1_P(\L^m, \L)\cap M = J^1({\L^m}^{(0)}, \L^{(0)}) \times \prod_{j=1}^k\widetilde J^1({\L^m}^{(j)}, \L^{(j)}) .  
\endaligned
$$
\end{Proposition}

\begin{proof} La premi\`ere assertion d\'ecoule du corollaire 
\ref{invariants}. Les d\'efinitions donn\'ees d\'eter\-mi\-nent bien une repr\'esentation $\phi$ de 
$J^1_P(\L^m, \L)$ : elles recouvrent tout le groupe, 
 elles sont   compatibles et on v\'erifie ais\'ement, gr\^ace au fait que $\theta$ est normalis\'e par 
$P^+(\L_\oE)$ qui contient $P_1(\L^m_\oE)$, que $\phi$ est un homomorphisme.   

D'autre part $J^1(\L^m, \L) =  J^1(\L) J^1_P(\L^m, \L) $ (une seule double classe) car 
$P^1(\L_\oE^m) = [(P^+ (\L_\oE) \cap P ) \ P_1(\L_\oE)] \cap P^1(\L_\oE^m) 
=  (P^1 (\L^m_\oE) \cap P ) \ P_1(\L_\oE) $.  
Ainsi l'induite $\Phi$ de $\phi$ \`a $J^1(\L^m, \L)$ est irr\'eductible et prolonge $\eta$  puisque  $$\Res_{J^1(\L)}\Phi = \Res_{J^1(\L)}\Ind_{J^1_P(\L^m, \L)}^{J^1(\L^m, \L)} \phi 
= \Ind_{J^1_P(\L)}^{J^1(\L)} \eta_P = \eta.
$$

Reste \`a voir que $\Phi$ est bien le prolongement de $\eta$ d\'efini \`a la  proposition  \ref{etamM}. Soit donc $\L^{''}$ une suite de
$\oE$-r\'eseaux telle que $\fb_0(\L^{\prime\prime})= \fb_0(\L^m)$ et $\fa_0(\L^{\prime\prime})\subseteq \fa_0(\L)$ ; il nous faut montrer que 
$ 
\Ind_{J^1(\L^{\prime\prime} )}^{  P_1(\L^{\prime\prime} )} \, \eta^{\prime\prime} 
\;  \simeq 
\; 
\Ind_{P_1(\L_{\oE}^m) J^1(\L )}^{  P_1(\L^{\prime\prime})} \,  \Phi   
	$. Ces deux induites sont   irr\'eductibles (voir preuve du lemme \ref{mult}) et par 
  (\ref{obvious}) il suffit de  prouver 
$$ 
\Hom_{J^1_P(\L^{\prime\prime}) \cap J^1_P(\L^m, \L)}(\eta_P^{\prime\prime}, \phi) \ne \{ 0 \}.$$ 
Les  groupes $J^1_P(\L^{\prime\prime})$ et $  J^1_P(\L^m, \L)$ ont   une d\'ecomposition d'Iwahori par rapport \`a $(P,M)$ et   les deux repr\'esentations consid\'er\'ees sont triviales sur les intersections avec $U$ et $U^{-}$. Enfin,  les restrictions \`a 
 $J^1_P(\L^{\prime\prime})\cap M$ et $  J^1_P(\L^m, \L)\cap M$ sont \'evidemment entrelac\'ees par d\'efinition de  $\eta({\L^m}^{(0)}, \L^{(0)})$  et $ 
\widetilde \eta({\L^m}^{(j)}, \L^{(j)}) $ et via le Corollaire \ref{intertwiningeta}, c.q.f.d.  
\end{proof}

Terminons par une variante de la Proposition \ref{intertwining}:  
\begin{Proposition} \label{intertwiningbis}   
Pla\c cons-nous dans  les hypoth\`eses (\ref{hypotheses}) et 
supposons   $  \fa_0(\L^m)$ contenu dans $ \fa_0(\L)  $.   
 Soit $\kappa^m$ un prolongement   de $\eta^m$ \`a $J^+(\L^m)$ 
 et $\kappa $ un prolongement   de $\eta(\L^m, \L) $ \`a $P^+(\L^m_{\oE})  J^1(\L)$. 
 Alors 
$$ 
\aligned
\kappa =  \mathfrak B_{\L^m \, \L}    (\kappa^m) \quad &\iff \quad 
\Hom_{J^+_P(\L^m) \cap J^+_P(\L)}(\kappa^m_P, \kappa_P) \ne \{ 0 \} 
\\
&\iff \quad 
\Hom_{J^+_{P^{-}}(\L^m) \cap J^+_P(\L)}(\kappa^m_{P^{-}}, \kappa_P) \ne \{ 0 \} 
\endaligned
$$ 
 \end{Proposition} 
\begin{proof} 
Les implications $\Longleftarrow$ sont  \'el\'ementaires via la proposition \ref{intertwining}
, en combinant  r\'eciprocit\'e de Frobenius et d\'ecomposition de Mackey (dont le premier terme suffit).

Reprenons maintenant la d\'emonstration de  \ref{intertwining} en posant $P' = P$ ou 
$P^{-}$ et avec les m\^emes notations. On veut montrer que 
la compos\'ee suivante est non nulle : 
$$
\kappa_P \stackrel{\mathcal J_P}{\hookrightarrow} \kappa 
\stackrel{\mathcal J}{\hookrightarrow}  
\Ind_{P^+(\L_{\oE}^m) J^1(\L )}^{P^+(\L_{\oE}^m) P_1(\L^m)} \kappa 
\quad \stackrel{\mathcal E}{\longrightarrow} \quad  \Ind_{J^+(\L^m)}^{P^+(\L_{\oE}^m) P_1(\L^m)} \kappa^m 
\stackrel{\mathcal P}{\longrightarrow} \kappa^m 
\stackrel{\mathcal P_{P'}}{\longrightarrow} \kappa^m_{P'}
$$ 
Les ingr\'edients sont les m\^emes :

--	La multiplicit\'e de $\eta_P$ dans $\Ind_{P^+(\L_{\oE}^m) J^1(\L )}^{P^+(\L_{\oE}^m) P_1(\L^m)} \kappa $ est $1$ par r\'eciprocit\'e de Frobenius puisque    celle de $\eta    \simeq \Ind_{J^1_{P}(\L )}^{  J^1(\L )}\eta_{P} $ est $1$ (lemme \ref{mult}).  
	L'image de $\mathcal E \circ j \circ j_P$ est donc la composante isotypique de type $\theta_P$ de  
	$\Ind_{J^+(\L^m)}^{P^+(\L_{\oE}^m) P_1(\L^m)} \kappa^m $.

	-- Il existe un vecteur $v$ de l'espace de $\kappa^m_{P'}$ tel que 
	$$ \oint_{H^1_P(\L) \cap J^+_{P'}(\L^m)} \theta_P(x^{-1}) \ \kappa^m_{P'}(x) \ v \, dx  \ne 0 . 
$$
En effet, il suffit (par Frobenius et Mackey comme pour le lemme \ref{projecteur}) de montrer que 
	 $\Hom_{H^1_P(\L) \cap H^1_{P'}(\L^m)} (\theta_P,  \theta^m_{P'})   \ne \{ 0\}$. L'intersection des groupes $H^1_P(\L)$ et $ H^1_{P'}(\L^m)$ se calcule terme \`a terme sur leur     d\'ecomposition  d'Iwahori. Or les ca\-rac\-t\`eres  sont triviaux sur les intersections avec   $U$ et 
$U^{-}$ et  sont tranferts l'un de l'autre sur l'intersection avec $M$.  
\end{proof}

\subsection{Bijections canoniques et restriction de Jacquet  }\label{Jacquet} 

Soit $     V = \oplus_{j=-k}^k W^{(j)}    $ une  d\'ecomposition  autoduale   proprement subordonn\'ee \`a 
  $[\L,n,0,\b]$~; elle
   est aussi proprement subordonn\'ee \`a 
$   [\L',n',0,\b]   $. 
Soit  $M $  le sous-groupe de Levi de $G^+$ fixant la d\'ecomposition 
$V = \oplus_{j=-k}^k W^{(j)}$ et $P$  un sous-groupe parabolique de $G^+$ de facteur de Levi $M$ et radical unipotent $U$ ; on note $P^{-}$ et $U^{-}$ les oppos\'es de $P$ et $U$ par rapport \`a $M$.

\begin{Lemma}\label{RP} 
L'application $\mathbf r_P$ : $\kappa \mapsto {\kappa_P}_{| J^+(\L) \cap M}$ est une bijection de l'ensemble des prolongements de 
$\eta$ \`a $J^+(\L)$ sur l'ensemble  des prolongements de ${\eta_P}_{| J^1_P(\L) \cap M}$ \`a 
$J^+_P(\L) \cap M$. 
\end{Lemma} 
\begin{proof}
D\'ecoule de  $J^+(\L) \cap U = J^1(\L) \cap U$ et 
$ J^+_P(\L)= (J^+(\L) \cap M ) J^1_P(\L)$.  
\end{proof}

La bijection canonique $\mathfrak B$  de 
 \cite[Lemma 4.3]{S5}   (Proposition \ref{bijcan} ci-dessus) est compatible \`a la bijection $\mathbf r_P$   : 
 
 \begin{Proposition}\label{diagcom1}
 Le diagramme suivant est commutatif : 
 $$
 \aligned 
  \{\text{Prolongements de } \eta \text{ \`a } J^+(\L)\} \quad 
 &\stackrel{\mathbf r_P}{\longrightarrow} \quad 
  \{\text{Prolongements de } {\eta_P}_{| J^1(\L) \cap M} \text{ \`a } J^+(\L) \cap M\} 
  \\
 \mathfrak B_{\L \L^\prime} \updownarrow \qquad \qquad \qquad \qquad 
 &\qquad \qquad \qquad \qquad \qquad \mathfrak B   \updownarrow 
  \\
    \{\text{Prolongements de } \eta' \text{ \`a } J^+(\L')\} \quad 
 &\stackrel{\mathbf r_P}{\longrightarrow} \quad 
  \{\text{Prolongements de } {\eta'_P}_{| J^1(\L') \cap M} \text{ \`a } J^+(\L') \cap M\} 
  \endaligned
 $$
 De plus, les bijections $ \ \mathfrak B \ $ et $ \ \mathbf r_P \ $ ci-dessus sont 
 compatibles \`a la torsion par les caract\`eres de 
 $P^+( \L_{\oE})/ P^1( \L_{\oE})= P^+( \L'_{\oE})/ P^1( \L'_{\oE})$. 
 \end{Proposition}

 La bijection canonique $\mathfrak B$ du c\^ot\'e droit est bien s\^ur d\'efinie comme le    produit 
 de $\mathfrak B^{(0)} =\mathfrak B_{\L^{(0)} {\L^\prime}^{(0)}}$ et des bijections canoniques 
 $\widetilde {\mathfrak B}^{(j)} $ 
 entre prolongements de 
 $\tilde\eta^{(j)}$ et de $\tilde{\eta^\prime}^{(j)}$ 
\cite[\S 4.3]{S5}. 
 En effet le caract\`ere semi-simple $\theta'$ est un produit 
 $\theta^{'(0)} \otimes \bigotimes_{j=1}^k (\tilde\theta^{'(j)})^2$  comme ci-dessus, o\`u 
 $\theta^{'(0)}$ est le transfert de $\theta^{(0)}$  et chaque $\tilde\theta^{'(j)}$, pour 
 $j > 0$, le transfert de $\tilde\theta^{(j)}$.

\begin{proof} 
Gr\^ace \`a la technique de \cite[\S 4]{S5}, bas\'ee sur \cite[Lemme 2.10]{S5}, il suffit de d\'emontrer cette assertion 
si 
$\fa_0(\L ) \subseteq \fa_0(\L^\prime ) $, ce que nous supposons d\'esormais. 

Les applications consid\'er\'ees sont des   bijections,  il suffit donc de montrer que  si ${\kappa_P}_{| J^+(\L) \cap M}$ 
et ${\kappa_P^\prime}_{| J^+(\L^\prime) \cap M}$ se  correspondent par la  
  bijection canonique, alors 
$\kappa^\prime $ et $\kappa$ aussi. L'hypoth\`ese se traduit en 
$  \ \Hom_{J^+ (\L ) \cap J^+ (\L^\prime) \cap M }(\kappa_P, \kappa_P^\prime) \ne \{ 0 \}  \  $ (Proposition \ref{intertwining}). Or $J^+_P (\L ) $ et $J^+_P (\L^\prime)$ ont tous deux une d\'ecomposition d'Iwahori  et $\kappa_P^\prime $ et $\kappa_P$ 
sont   toutes deux triviales sur les intersections avec 
$U$ et $U^{-}$ (comme $\eta_P^\prime $ et $\eta_P$). Ainsi 
$\Hom_{J^+_P (\L ) \cap J^+_P (\L^\prime)   }(\kappa_P, \kappa_P^\prime) \ne \{ 0 \} $, d'o\`u  
$\Hom_{P^+(\L_{\oE}) P_1(\L)  }(
\Ind_{  J^+_P(\L )}^{P^+(\L_{\oE}) P_1(\L)}\kappa_P,
 \Ind_{  J^+_P(\L^\prime )}^{P^+(\L_{\oE}) P_1(\L)}\kappa_P^\prime) \ne \{ 0 \} $ par (\ref{obvious}), et le r\'esultat. 
 \end{proof}

 \medskip 
Pour \'etudier le cas o\`u la subordination n'est pas propre, nous nous pla\c cons 
 d\'esormais sous les hypoth\`eses (\ref{hypotheses}) qui   permettent de d\'efinir 
 $\kappa_P$ pour tout prolongement de $\eta$ \`a 
 $P^+(\L_\oE^m)J^1(\L)$, et  
    d'\'enoncer des variantes du lemme   et de la proposition ci-dessus.
        
  \begin{Lemma}\label{RPbis}  Sous les hypoth\`eses (\ref{hypotheses}), 
l'application $\mathbf r_{P}$ : $\kappa \mapsto {\kappa_{P}}_{| J^+_P(\L) \cap M}$ est une bijection de l'ensemble des prolongements de 
$\eta(\L^m, \L)$ \`a  $P^+(\L_\oE^m)J^1(\L)$ sur l'ensemble  des prolongements \`a 
$J_P^+(\L) \cap M$ de la repr\'esentation
 $$\eta_{P|M}(\L^m, \L) = {\eta_P(\L^m, \L)}_{| J_P^1(\L^m,\L) \cap M}.$$
\end{Lemma} 

\begin{proof}
Certainement $\kappa \mapsto {\kappa_{P}}$ est une bijection de l'ensemble des prolongements de $\eta$ \`a 
$P^+(\L_\oE^m)J^1(\L)$  sur l'ensemble  des prolongements de $\eta_P$ \`a 
$J_P^+(\L) $ (Corollaire \ref{invariants}). C'est pour conserver une bijection en restreignant \`a $J_P^+(\L) \cap M$ qu'il faut travailler sur les prolongements de 
$\eta(\L^m, \L)$. Notons d'abord les propri\'et\'es 
indispensables :  
\begin{equation}\label{decompositions}
\aligned
J^+(\L) \cap M&=J_P^+(\L) \cap M=(P^+(\L_\oE)\cap M) \ 
(J^1(\L)\cap M)
\\
J_P^+(\L) \ \  &= (H^1(\L) \cap U^{-}) \ 
(J_P^+(\L)\cap M) \ (P^+(\L_\oE)\cap U) \ (J^1(\L)\cap U)
\\
P^+(\L_\oE)\cap U &= P^+(\L_\oE^m) \cap U = P^1(\L_\oE^m)\cap U
\\
P^+(\L_\oE)\cap M &= P^+(\L_\oE^m) \cap M
\endaligned
\end{equation}
qui d\'ecoulent de l'hypoth\`ese,  
du paragraphe pr\'ec\'edent  et du fait que la d\'ecomposition est proprement subordonn\'ee \`a $\L^m$. 
L'injectivit\'e s'en d\'eduit car 
$\eta_P(\L^m, \L)$ est triviale sur $H^1(\L)\cap U^{-}$ (comme $\eta_P$), sur $J^1(\L)\cap U$ 
(comme $\eta_P$) et sur $P^1(\L_\oE^m)\cap U$ 
(Proposition \ref{etamMP}).

Pour la surjectivit\'e, on remarque que 
$J_P^+(\L)  = (J_P^+(\L) \cap M) \ J^1_P(\L^m, \L)$. 
Il faut v\'erifier que si  
$\xi$ est un prolongement de ${\eta_{P|M}(\L^m, \L)}$ \`a 
$J_P^+(\L) \cap M$, alors $\phi(xy)=\xi(x)\eta_P(\L^m, \L)(y)$, $x \in  J_P^+(\L) \cap M$, 
$y \in J^1_P(\L^m, \L)$, d\'efinit une repr\'esentation de 
$J_P^+(\L)$. Cela revient \`a montrer  que $\eta_P(\L^m, \L)(y)\xi(x)=\xi(x)
\eta_P(\L^m, \L)(x^{-1}yx)$ ($x \in P^+(\L_\oE^m)\cap M$, 
$y $ comme ci-dessus). C'est imm\'ediat, gr\^ace \`a 
 la trivialit\'e de 
$\eta_P(\L^m, \L)(y)$ pour $y \in P^1(\L_\oE^m)\cap U$ (Proposition \ref{etamMP}).
\end{proof}

 \begin{Proposition}\label{diagcom2} 
 Sous les hypoth\`eses (\ref{hypotheses}), 
 le diagramme suivant est commutatif : 
 $$
 \aligned 
  \{\text{Prolongements de } \eta^m \text{ \`a } J^+(\L^m)\} \qquad  
  &\stackrel{\mathbf r_P}{\longrightarrow}  \quad 
  \{\text{Prolong. de } {\eta_P^m}_{| J^1(\L^m) \cap M} \text{ \`a } J^+(\L^m) \cap M\} 
  \\
 \mathfrak B_{\L^m \L} \updownarrow \qquad \qquad \qquad \qquad 
 &\qquad \qquad \qquad \qquad \qquad \mathfrak B \updownarrow 
  \\
    \{\text{Prolong. de } \eta(\L^m,\L) \text{ \`a }  P^+(\L_{\oE}^m) J^1(\L)\}      
  &\stackrel{\mathbf r_{P}}{\longrightarrow}   
  \{\text{Prolong. de } {\eta_{P|M}(\L^m, \L)} \text{ \`a } J^+(\L)   \cap \!  M\} 
  \endaligned
 $$
 \end{Proposition}
 \begin{proof}
 La d\'emonstration est celle de \ref{diagcom1}, on utilise les d\'ecompositions d'Iwahori  (\ref{decompositions}) 
 et le fait que les repr\'esentations $\kappa_P^m$ et  
 $\kappa_P$ sont triviales sur les intersections avec $U$ et $U^{-}$, comme $\eta_P^m$ et  $\eta_P(\L^m, \L)$ (proposition \ref{etamMP}). 
 \end{proof}

\setcounter{section}{1}

\section{Une repr\'esentation de Weil}\label{Sec2}

\subsection{Des sous-groupes remarquables}\label{remarquables}

Soit   $[\L,n,0,\b]$  une  
  strate gauche semi-simple et $V = V^1 \perp \cdots \perp V^l$   la   d\'ecomposition de $V$ associ\'ee. 
Soit $V = Z_1 \perp \cdots \perp Z_t$ une d\'ecomposition moins fine de $V$ : chaque $Z_j$, $1 \le j \le t$, 
est somme de certains $V^i$. La strate $[\L,n,0,\b]$ 
est alors somme directe de strates semi-simples $[S_j=\L\cap Z_j,n_j,0,\b_j=\b_{|Z_j}]$ dans chaque 
$\End_F(Z_j)$. 
On souhaite  examiner  en d\'etail la structure de $\eta^\PM$, l'unique repr\'e\-sen\-ta\-tion irr\'eductible de $J^1(\L^\PM)$ contenant le caract\`ere semi-simple gauche $\theta^\PM$ de  $H^1(\L^\PM)$, et  de ses prolongements $\kappa^\PM$ \`a $J^+(\L^\PM)$, en la comparant \`a celle des objets analogues d\'efinis dans le groupe 
$$G_\interieur^+ = G^+ \cap \prod_{1 \le j \le t} \GL_F(Z_j)$$ 
\`a partir des strates  $[S_j ,n_j,0,\b_j]$  et des caract\`eres semi-simples gauches correspondant  \`a $\theta^\PM$.
Remarquons que le commutant $G_E^+$ de $\beta$ dans $G^+$ 
v\'erifie 
$$G_E^+ \  \subseteq \   G^+ \cap \prod_{1 \le i \le l} \GL_F(V^i) 
\  \subseteq G_\interieur^+ .$$

Consid\'erons le groupe $J^1_\interieur(\L^\PM)=  
 J^1(\L^\PM) \cap G_\interieur^+ $.   
La d\'ecomposition $V\! =\! Z_1 \perp\! \cdots \! \perp Z_t$ est subordonn\'ee \`a la strate $[\L,n,0,\b]$ et les r\'esultats de \cite[\S 5.1, \S 5.2]{S5} s'appliquent~: 
$$ 
H^1(\L^\PM)\cap G_\interieur^+ = \prod_{1 \le j \le t} H^1(S_j),  \ \   
\theta_{|H^1(\L^\PM)\cap G_\interieur^+} = \otimes _{1 \le j \le t} \ 
\theta(S_j), $$
o\`u les  $\theta(S_j)$ sont des caract\`eres semi-simples gauches  attach\'es aux strates $[S_j ,n_j,0,\b_j]$  et transferts de 
$\theta$ au sens de {\it loc. cit.,} Proposition 5.5, et 
$$
J^1_\interieur(\L^\PM) =  \prod_{1 \le j \le t} J^1(S_j) . 
$$
 
Soit  $J^1_\ext(\L^\PM)$  l'orthogonal de $H^1(\L) J^1_\interieur(\L^\PM)$  pour la forme bilin\'eaire 
altern\'ee non d\'eg\'en\'er\'ee usuelle $(x,y) \longmapsto \, <x,y> =  \theta^\PM([x,y])$ sur 
$J^1(\L^\PM)/H^1(\L^\PM)$. 
La restriction de cette forme  aux sous-espaces orthogonaux 
$H^1(\L) J^1_\interieur(\L^\PM)/ H^1(\L^\PM)$ et $J^1_\ext(\L^\PM)/ H^1(\L^\PM)$ reste non d\'eg\'en\'er\'ee et l'on d\'efinit comme d'habitude la repr\'esentation irr\'eductible 
$\eta_\interieur^\PM$ de $J^1_\interieur(\L^\PM)$ (resp. $\eta_\ext^\PM$ 
de $J^1_\ext(\L^\PM)$) comme  l'unique repr\'esentation irr\'eductible (\`a isomorphisme pr\`es) dont la restriction \`a $H^1(\L^\PM) \cap G_\interieur^+$  (resp. 
$H^1(\L^\PM)$)  contient $\theta^\PM$ (et en est multiple). 
On reconna\^\i t en outre via {\it loc.cit.} que  
$\eta_\interieur^\PM$  est isomorphe \`a 
$\otimes_{1 \le j \, \le t} \ \eta(S_j)$. 

Comme $J^1_\interieur(\L^\PM)J^1_\ext(\L^\PM)= J^1(\L^\PM)$ et 
$(H^1(\L^\PM)J^1_\interieur(\L^\PM) ) \cap J^1_\ext(\L^\PM)= H^1(\L^\PM)$, on obtient une r\'ealisation de la repr\'esentation $\eta^\PM$ 
en posant pour $g \in J^1(\L^\PM)$ :
\begin{equation}\label{tensor}
\eta^\PM(g) = \eta^\PM_\ext(g_\ext) \otimes \eta^\PM_\interieur(g_\interieur) \  \text{ si }
g = g_\ext \ g_\interieur  \  \text{ avec } 
g_\ext \in J^1_\ext(\L^\PM),   \  g_\interieur \in   J^1_\interieur(\L^\PM).
\end{equation}

D\'efinissons dans  $\widetilde G = GL(V)$ et 
 $\widetilde G_\interieur  =  \prod_{1 \le j \le t} \GL_F(Z_j)$ les sous-groupes analogues 
  $\widetilde J^1_\interieur(\L^\PM) = \widetilde J^1(\L^\PM) \cap \widetilde G_\interieur$ 
  et son orthogonal 
 $\widetilde J^1_\ext(\L^\PM)$ pour la forme  $\tilde\theta^\PM([x,y])$,  o\`u $\tilde \theta$ est un caract\`ere semi-simple gauche de $\widetilde H^1(\L^\PM)$ de restriction $\theta$ \`a 
 $ H^1(\L^\PM)$.  Soit $Q$ un sous-groupe parabolique de $\widetilde G$ de facteur de Levi $\tilde G_\interieur$, $N$ son radical unipotent et $N^-$ le radical unipotent du parabolique oppos\'e 
	par rapport \`a  $\tilde G_\interieur$. D'apr\`es \cite[Proposition 5.2, Lemma 5.6(iii)]{S5} on a : 
	$$\widetilde J^1_\ext(\L^\PM) = (\widetilde J^1(\L^\PM) \cap N^-) 
	(\widetilde J^1(\L^\PM) \cap N ) \widetilde H^1(\L^\PM).
	$$
\begin{Lemma}\label{intersection}
\begin{enumerate}
	\item On a $  J^1_\ext(\L^\PM) = \widetilde J^1_\ext(\L^\PM)\cap G$,     ind\'ependant de~$\theta$. 
\item 
On a 
pour tout 
 $g \in G_E^+$ : 
 $$
 J^1(\L^\PM) \cap g J^1(\L^\PM) g^{-1} = 
 (  J^1_\ext(\L^\PM) \cap g J^1_\ext(\L^\PM) g^{-1} ) . 
 (J^1_\interieur(\L^\PM) \cap g J^1_\interieur(\L^\PM) g^{-1}).  
 $$
 \end{enumerate}
 \end{Lemma}
\begin{proof}
(i) De $ \  J^1_\interieur(\L^\PM) = \widetilde J^1_\interieur(\L^\PM)\cap G \ $ 
on d\'eduit par orthogonalit\'e :  $H^1_\ext(\L^\PM) \subseteq\widetilde J^1_\ext(\L^\PM)\cap G \subseteq J^1_\ext(\L^\PM)  $. 
L'\'egalit\'e d\'ecoule alors de 
$ \widetilde J^1 (\L^\PM)\cap G = ( \widetilde J^1_\interieur(\L^\PM)\cap G) ( \widetilde J^1_\ext(\L^\PM)\cap G)$ dont la d\'emonstration est un cas particulier de celle de (ii) ci-dessous.

(ii)
 La difficult\'e technique est que $G_\interieur^+$ n'est pas un sous-groupe de Levi de $G^+$. 
 Il faut  remonter dans $\widetilde G = GL(V)$ 
 et utiliser la d\'ecomposition d'Iwahori de $\widetilde J^1(\L^\PM)$ relativement 
\`a   $Q$    pour calculer l'intersection  
$\widetilde J^1(\L^\PM) \cap g \widetilde J^1(\L^\PM) g^{-1}$ terme \`a terme. Comme $g$ appartient \`a 
$G_E^+$ contenu dans $G_\interieur^+$,  on obtient ainsi l'\'egalit\'e analogue dans $\widetilde G$. Pour revenir \`a $G^+$ on prend les points fixes de l'involution adjointe : on applique  
l'argument de Stevens   \cite[Theorem 2.3]{S1}, en posant 
$H = \widetilde J^1_\interieur(\L^\PM) \cap g 
\widetilde J^1_\interieur(\L^\PM) g^{-1}$,  
$U =  \widetilde J^1_\ext(\L^\PM) \cap g \widetilde J^1_\ext(\L^\PM) g^{-1}$,  et en remarquant que 
$\widetilde J^1_\ext(\L^\PM) $ et 
$\widetilde J^1_\interieur(\L^\PM) $
se normalisent mutuellement puisque $[\widetilde J^1(\L^\PM) , \widetilde J^1(\L^\PM) ] \subset \widetilde H^1(\L^\PM)$.   \end{proof}

\begin{Proposition}\label{operateurs}
\begin{enumerate}
	\item 
Les repr\'esentations $\eta_\ext^\PM$ et $\eta^\PM_\interieur$ sont entrelac\'ees par $G_E^+$.
\item  
 Soit $g \in G_E^+$ et soit $E^\PM_\ext(g)$, resp. $E^\PM_\interieur(g)$,   un 
 op\'erateur d'entrelacement de 
 $\eta_\ext^\PM$, resp. $\eta^\PM_\interieur$,  en $g$.
Alors $E^\PM(g) = E^\PM_\ext(g) \otimes 
 E^\PM_\interieur(g)$ est un op\'erateur d'entrelacement de 
 $\eta^\PM$ en $g$.
 Tout op\'erateur d'entrelacement de 
 $\eta^\PM$ en $g$ est de cette forme. 
\item Les espaces d'entre\-la\-ce\-ment 
de $\eta_\ext^\PM$ et de $\eta^\PM_\interieur$ en $g \in G_E^+$ sont de dimension $1$.
\end{enumerate}
\end{Proposition}

\begin{proof}
$G_E^+$ entrelace $\theta^\PM$, donc $\eta_\ext^\PM$ et $\eta^\PM_\interieur$. Pour le deuxi\`eme point,  comme la dimension de l'espace d'entrelacement de $\eta^\PM$ en $g$ est $1$, il suffit de montrer que   
$E^\PM_\ext(g) \otimes 
 E^\PM_\interieur(g)$ entrelace effectivement $\eta^\PM$. Cela n'est qu'une v\'erification \'etant donn\'e le lemme ci-dessus. Le troisi\`eme point  d\'ecoule du deuxi\`eme vu 
 la propri\'et\'e semblable de $\eta^\PM$. 
\end{proof}

 Le groupe $P^+(\L^\PM_\oE)$, qui est contenu dans 
$ G_\interieur^+$, normalise $J^1(\L^\PM)$ et $H^1(\L^\PM)$ 
et entrelace le caract\`ere $\theta$ ; il normalise donc 
$H^1(\L^\PM) J^1_\interieur(\L^\PM)$ et son orthogonal $J^1_\ext(\L^\PM)$ pour la forme $
< , >$   sur 
$J^1(\L^\PM)/H^1(\L^\PM)$. 
Consid\'erons  un prolongement $\kappa^\PM$ de $\eta^\PM$ 
\`a $J^+(\L^\PM) = P^+(\L^\PM_\oE) J^1(\L^\PM)$ et un prolongement 
$\kappa^\PM_\interieur$ de $\eta^\PM_\interieur   $ \`a $P^+(\L^\PM_\oE) J^1_\interieur(\L^\PM)
=   \prod_{1 \le j \le t} J^+(S_j) $ : il en existe par 
 \cite[\S 4.1]{S5}. D'apr\`es la proposition \ref{operateurs}, pour tout $g \in P^+(\L^\PM_\oE)$,  l'op\'erateur d'entrelacement 
$\kappa^\PM(g)$   s'\'ecrit de fa\c con unique sous la forme 
$\kappa^\PM(g) = E^\PM_\ext(g) \otimes \kappa_\interieur^\PM(g)$ 
et les op\'erateurs $E^\PM_\ext(g)$ ainsi d\'efinis fournissent une repr\'esentation de 
$P^+(\L^\PM_\oE) $ dans l'espace de $\eta_\ext^\PM$, triviale sur $P_1(\L^\PM_\oE)$ 
(qui est contenu dans $H^1(\L^\PM)$). 
Autrement dit :

\begin{Corollary}\label{Eext}
Il existe des repr\'esentations   de $P^+(\L^\PM_\oE) /P_1(\L^\PM_\oE)$ 
dans l'espace de $\eta_\ext^\PM$  
  entrela\c cant $    \       \eta_\ext^\PM$ ; elles diff\`erent entre elles d'un caract\`ere. 
Tout choix d'une telle repr\'esentation $\mathcal X $ d\'etermine une bijection 
$ \Interieur_{\mathcal X } $ entre   prolongements de   $\eta^\PM$ 
\`a $J^+(\L^\PM)$ et   prolongements  de $ \eta_\interieur $
  \`a $ P^+(\L^\PM_\oE) J^1_\interieur(\L)   $, donn\'ee par : 
  $$\kappa^\PM(g) \  = \ \mathcal X(g) \ \otimes \  \Interieur_{\mathcal X }(\kappa)(g), \quad  \  g \in P^+(\L^\PM_\oE). 
  $$
\end{Corollary}

Une telle repr\'esentation $\mathcal X $ \'etant fix\'ee, on notera \'egalement $ \Interieur_{\mathcal X } $ la bijection reliant de la m\^eme fa\c con prolongements de   $\eta^\PM$ 
\`a $T J^1(\L^\PM)$ et   prolongements  de $ \eta_\interieur $
  \`a $ T J^1_\interieur(\L)   $, pour tout sous-groupe $T$ de $ P^+(\L^\PM_\oE)$ 
  contenant $ P_1(\L^\PM_\oE)$. 

\subsection{Mise en place de la restriction de Jacquet}

Donnons-nous \`a pr\'esent, pour $1 \le j \le t$, des d\'ecompositions des sous-espaces $Z_j$ subordonn\'ees aux strates  $[S_j ,n_j,0,\b_j]$ et autoduales (\cite[\S 5.3]{S5}) :
$$ Z_j = \oplus_{r=-m_j}^{m_j} W_j^{(r)},$$ 
o\`u l'orthogonal de $W_j^{(r)}$ dans $Z_j$ est 
$\oplus_{s\ne -r} W_j^{(s)}$. On demande que $W_j^{(0)}$ soit non nul pour au plus un $j$ et l'on somme ces d\'ecompositions de fa\c con disjointe, c'est-\`a-dire 
que la d\'ecomposition de $V$ obtenue, soit 
$ V= \oplus_{i=-m}^{m} W^{(i)},$ v\'erifie 
que pour tout $i$, $-m \le i \le m$, il y a un unique $j$, 
$1 \le j \le t$, et un unique $r$, $-m_j \le r \le m_j$, tels que  $W^{(i)} = W_j^{(r)}$ et $W^{(-i)} = W_j^{(-r)}$. 
On obtient une d\'ecomposition autoduale de $V$ subordonn\'ee \`a la strate 
$[\L,n,0,\b]$.

Soit $M$ le sous-groupe de Levi de $G^+$ stabilisant cette d\'ecomposition et $P$ un sous-groupe parabolique de $G^+$ de facteur de Levi $M$. On note $U$ le radical unipotent de $P$ et $U^-$ son oppos\'e par rapport \`a $M$. 
D'apr\`es \cite[Corollaire 5.10]{S5} les groupes 
$H^1(\L^\PM)$ et $J^1(\L^\PM)$ ont une d\'ecomposition d'Iwahori par rapport \`a $(M, P)$. Puisque $M$ est contenu dans $G_\interieur^+ $ et que les d\'ecompositions induites sur les 
$Z_j$ sont subordonn\'ees aux strates  $[S_j ,n_j,0,\b_j]$, le groupe 
$H^1(\L^\PM) J^1_\interieur (\L^\PM)$ a lui aussi une  d\'ecomposition d'Iwahori par rapport \`a $(M, P)$.

 Rappelons enfin 
(\cite[Proposition 7.2.3]{BK}, \cite[Lemma 5.6]{S5}) que l'on a une d\'ecomposition de l'espace symplectique $J^1(\L^\PM)/ H^1(\L^\PM)$ en 
$$
\frac{J^1(\L^\PM)}{ H^1(\L^\PM)} \  = \  \frac{J^1(\L^\PM)\cap M}{ H^1(\L^\PM)\cap M} \ \perp \  \left(\frac{J^1(\L^\PM)\cap U^-}{ H^1(\L^\PM)\cap U^- } \oplus 
\frac{J^1(\L^\PM)\cap U}{ H^1(\L^\PM))\cap U} \right) 
$$
dans laquelle les deux derniers sous-espaces sont totalement isotropes, en dualit\'e~; on a une d\'ecomposition analogue pour $H^1(\L^\PM) J^1_\interieur (\L^\PM)/H^1(\L^\PM) $. On en d\'eduit que 
$J^1_\ext (\L^\PM)$ a lui aussi une  d\'ecomposition d'Iwahori par rapport \`a $(M, P)$
et que   :  
\begin{itemize}
\item $J^1_\ext (\L^\PM) \   = \  (J^1_\ext(\L^\PM)\cap U^\PMopp)  \ \  (H^1(\L^\PM)\cap M)  \  \  (J^1_\ext(\L^\PM)\cap U^\PM)$. 
\item 
$J^1 (\L^\PM)\cap U^- = (J^1_\ext (\L^\PM)\cap U^- ) 
(J^1_\interieur (\L^\PM)\cap U^-)$, $J^1 (\L^\PM)\cap U = (J^1_\ext (\L^\PM)\cap U ) 
(J^1_\interieur (\L^\PM)\cap U)$, 
\item  $(H^1(\L^\PM)\cap U^\PM)(J^1_\interieur (\L^\PM)\cap U)$ 
est l'orthogonal de $J^1_\ext (\L^\PM)\cap U^-$ dans 
$J^1 (\L^\PM)\cap U$, 
\item  $  J^1_\ext (\L^\PM)\cap U $ 
est l'orthogonal de $(J^1_\interieur (\L^\PM)\cap U^-)(H^1(\L^\PM)\cap U^-)$ dans 
$J^1 (\L^\PM)\cap U$. 
\end{itemize}
 
Examinons \`a pr\'esent   les $J^1(\L^\PM) \cap U^\PM$-invariants de $\eta$ dans le mod\`ele (\ref{tensor}). 
Notons $X_\ext$ l'espace de $\eta_\ext^\PM$, $X_\interieur$ l'espace de $\eta_\interieur^\PM$ et 
$X = X_\ext \otimes X_\interieur$ celui de $\eta^\PM$. 
  L'espace $X_\ext^{J_\ext^1(\L^\PM) \cap U^\PM}$ est de dimension $1$ 
car $H^1(\L^\PM)(J_\ext^1(\L^\PM) \cap U^\PM) /H^1(\L^\PM) $ est totalement isotrope maximal dans 
$J_\ext^1(\L^\PM)   /H^1(\L^\PM)$  et 
$\theta_{|H^1(\L)\cap U}=1$. Le choix 
d'un vecteur non nul $x_\ext^\PM \in   X_\ext^{J_\ext^1(\L^\PM) \cap U^\PM}$ nous donne un isomorphisme :
\begin{equation}\label{xext}
\aligned 
X^{J^1(\L^\PM) \cap U^\PM} = X_\ext^{J_\ext^1(\L^\PM) \cap U^\PM} \otimes X_\interieur^{J_\interieur^1(\L^\PM) \cap U^\PM} \quad  & \stackrel{\Phi}{\longrightarrow}     \quad X_\interieur^{J_\interieur^1(\L^\PM) \cap U^\PM}
\\
x_\ext^\PM \otimes v \qquad &\mapsto \qquad  v 
\endaligned
\end{equation}
L'espace $X^{J^1(\L^\PM) \cap U^\PM}$ est une repr\'esentation de $$J^1_P(\L^\PM) = H^1 (\L^\PM)
(J^1(\L^\PM) \cap P)= H^1 (\L^\PM)(J^1_\interieur(\L^\PM) \cap P)
(J^1_\ext(\L^\PM) \cap U)$$ et l'isomorphisme (\ref{xext}) 
 commute \`a l'action de $  J^1_\interieur(\L^\PM)\cap P^\PM  
$. 
D'autre part $P^+(\L_\oE^\PM) \cap P^\PM$ normalise $J^1(\L^\PM) \cap U^\PM$,  
$J_\ext^1(\L^\PM) \cap U^\PM$ et $J_\interieur^1(\L^\PM) \cap U^\PM$ donc agit sur leurs invariants ; 
en particulier, pour tout choix de repr\'esentation $\mathcal X$ comme dans le corollaire \ref{Eext},  $P^+(\L_\oE^\PM) \cap P^\PM$ agit sur $x_\ext^\PM$ par un caract\`ere 
 (trivial sur $P_1(\L_\oE^\PM) \cap P^\PM$).

 Donnons-nous alors   un prolongement $\kappa^\PM$ de $\eta^\PM$ 
\`a $J^1(\L^\PM)J^+_P(\L^\PM)$  et un prolongement 
$\kappa^\PM_\interieur$ de $\eta^\PM_\interieur$ \`a  $J_\interieur^1(\L^\PM)J^+_{\interieur,P}(\L^\PM)$.  Le lemme \ref{invariantsbis} 
permet d'en \'etudier les $J^1\cap U^\PM$-invariants ou $J^1_\interieur \cap U^\PM$-invariants, 
qui fournissent des repr\'esentations 
$\kappa_P$ de $J^+_P(\L^\PM)= (P^+(\L_\oE^\PM) \cap P^\PM)J^1_P(\L^\PM)$ et $\kappa_{\interieur,P^\PM}^\PM $ 
de $J^+_{\interieur,P}(\L^\PM)= (P^+(\L_\oE^\PM) \cap P^\PM)J^1_{\interieur,P}(\L^\PM)$.
(On pr\'ef\`ere la notation simplifi\'ee  $X_{\interieur,P}$ \`a la notation exacte 
$X_{\interieur,P_\interieur}$  avec $P_\interieur = P \cap G_\interieur$.) 
Rappelons d'autre part  (cf \ref{JP}) 
que   $J^+_{\interieur,P}(\L^\PM)/J^1_{\interieur,P}(\L^\PM)$ est canoniquement isomorphe \`a 
$ P^+(\L_\oE^\PM) \cap P^\PM/ P_1(\L_\oE^\PM) \cap P^\PM  $, ce qui permet de consid\'erer tout caract\`ere du second groupe comme un caract\`ere du premier. Avec ces conventions, on a en r\'esum\'e :

  \begin{Proposition}\label{eextU} 
Soit $\mathcal X$  une  repr\'esentation de $ \, P^+(\L^\PM_\oE) \, / \, P_1(\L^\PM_\oE)$ 
  entrela\c cant $ \ \eta_\ext^\PM    $. Notons    
 $\chi_U$ le caract\`ere de $P^+(\L_\oE^\PM) \cap P^\PM$ par lequel ce groupe agit  sur 
les $J_\ext^1(\L^\PM) \cap U^\PM$-invariants de $ \ \eta_\ext^\PM$. 
Soit $J^+_{\interieur,P}(\L^\PM)= (P^+(\L_\oE^\PM) \cap P^\PM)J^1_{\interieur,P}(\L^\PM)$ 
et soit $\kappa^\PM$ un prolongement  de $\eta^\PM$ 
\`a $J^1(\L^\PM)J^+_P(\L^\PM)$. Alors 
  $$ \ {\kappa^\PM_P}_{|J^+_{\interieur,P}(\L)}  =  \  \chi_U   \,     \Interieur_{\mathcal X }(\kappa)_P.$$ 
\end{Proposition} 
 
 \subsection{Une repr\'esentation \`a la Weil}\label{WeilRep}
 
 Il est une situation particuli\`ere dans laquelle on peut 
  \'elucider le caract\`ere  $\chi_U$ introduit en \ref{eextU}, en construisant une repr\'esentation  entrela\c cant $\eta_\ext$.
Soit donc $V = W^{(-1)} \oplus W^{(0)} \oplus  W^{(1)}$ une d\'ecomposition autoduale de $V$ subordonn\'ee \`a la strate 
$[\L,n,0,\b]$ et telle que $W^{(0)}$ soit somme de certains $V^i$ ; posons $Z_1 = W^{(-1)}  \oplus  W^{(1)}$, $Z_2 = W^{(0)}$. Soit $P$ le sous-groupe parabolique de $G^+$ stabilisant le drapeau $\{0\} \subset  W^{(-1)} \subset 
W^{(-1)} \oplus W^{(0)}
 \subset V$, 
 $U$ son radical unipotent,  $M$ le fixateur de la d\'ecomposition, $P^-$ et $U^-$ les oppos\'es de $P$ et $U$ par rapport \`a $M$.

Pour simplifier les notations, on raccourcit $J^1_\ext(\L^\PM)$, $H^1(\L^\PM)$ 
etc. en $J^1_\ext$, $H^1 $ etc.
Le caract\`ere $\theta^\PM$ de $H^1$, trivial sur $H^1 \cap U^\PM$ et $H^1 \cap U^\PMopp$,  se prolonge en un caract\`ere $\theta^\PM_{U^\PM}$ de $H^1 (J^1_\ext \cap U^\PM)$ trivial sur 
$J^1_\ext \cap U^\PM$. On peut alors r\'ealiser la repr\'esentation $\eta_\ext^\PM $ de $J^1_\ext$ comme l'action par translations \`a droite dans l'espace de fonctions 
\begin{equation}\label{espace}
 \calW = \{f : J^1_\ext \longrightarrow \mathbb C \  / \  f(ht) = \theta^\PM_{U^\PM}(h) f(t), \   h \in H^1 (J^1_\ext \cap U^\PM),  \  
t \in J^1_\ext \}, 
\end{equation}
isomorphe  par restriction   \`a l'espace des fonctions de $J^1_\ext \cap U^\PMopp / H^1 \cap U^\PMopp$ dans $\mathbb C$. 

On v\'erifie ais\'ement que pour $ g \in P^+(\L_\oE^\PM)\cap P^\PM$, l'op\'erateur $\Omega(g)$ d\'efini par 
$$
\Omega(g) (f) (x) = f(g^{-1} x g) , \ \  (f \in \calW,  \  x \in J^1_\ext,  \  g \in P^+(\L_\oE^\PM)\cap P^\PM ) 
 $$
conserve $\calW$.  Il existe (\S \ref{remarquables}) une repr\'esentation projective 
 $g \mapsto \Omega (g)$  de $P^+(\L_\oE^\PM)/P_1(\L_\oE^\PM)$ dans $\calW$ entrela\c cant 
$\eta_\ext^\PM$ et prolongeant cette  vraie repr\'esentation de $P^+(\L_\oE^\PM)\cap P^\PM$. 
La donn\'ee de $\Omega $ va nous permettre de construire  une vraie repr\'esentation par   l'\'el\'egante m\'ethode de 
Neuhauser  (\cite{N} ; voir aussi \cite{Sz}). On d\'efinit tout d'abord les sous-espaces $\calW^+$ et $\calW^-$ de $\calW$ form\'es des fonctions paires ($ \forall v \in J^1_\ext \cap U^\PMopp ,  f(v)=f(v^{-1})$) pour le premier, 
impaires ($\forall v \in J^1_\ext \cap U^\PMopp , f(v)=-f(v^{-1})$) pour le second.  
On pr\'etend que  
\begin{Lemma}
$ \ 
\calW^+  $  et  $ \  \calW^-  $  sont stables par les op\'erateurs $ \  \Omega(g)$,  $ g \in  P^+(\L_\oE^\PM)/P_1(\L_\oE^\PM)$. 
\end{Lemma}
\begin{proof}
L'\'el\'ement $ \epsilon : \epsilon_{|Z_1}= I_{Z_1}, \epsilon_{|Z_2}=- I_{Z_2} $,   
 appartient au centre de $P^+(\L_\oE^\PM) $. 
 Un bref calcul matriciel montre que, 
  pour tout $u \in  J^1_\ext \cap U^\PMopp$, 
   $ (\epsilon u \epsilon^{-1}) u$ appartient \`a $G_\interieur$  donc \`a son intersection 
  avec $J^1_\ext \cap U^\PMopp$, soit : 
  $ \ \epsilon u \epsilon^{-1} \in     (H^1  \cap U^\PMopp)u^{-1}$. 
  Les sous-espaces $\calW^+$ et $\calW^-$ sont donc les sous-espaces propres de $\Omega(\epsilon)$ pour les valeurs propres $1$ et $-1$ respectivement, d'o\`u le r\'esultat (certes 
il s'agit d'une repr\'esentation projective, cependant un isomorphisme $\Omega(g)$   ne peut \'echanger $\calW^+$ et $\calW^-$  dont les dimensions diff\`erent de $1$).  
\end{proof} 
   Soit $\Omega^+$ la restriction de $\Omega$ \`a $\calW^+$. D'apr\`es \cite[Theorem 4.3]{N},     o\`u  l'on examine  les d\'eterminants des repr\'esentations projectives $\Omega$ et $\Omega^+$,  on d\'efinit une vraie repr\'esentation         en posant :  
$$ 
\mathbb W (g) = \frac{\det \Omega(g)}{\det \Omega^+(g)^2} \  \  \Omega(g) , \  \  g \in  P^+(\L_\oE^\PM)/P_1(\L_\oE^\PM) . 
$$
Calculons     le caract\`ere 
$\mathbf w_U$  correspondant. 
Dans le mod\`ele ci-dessus, la droite des $J^1_\ext \cap U^\PM$-invariants de $\eta_\ext^\PM$ 
a pour base la fonction $f_1$ de support $H^1 (J^1_\ext \cap U^\PM)$ et valeur $1$ en $1$, 
  fix\'ee par les op\'erateurs $\Omega(u)$, $u \in P^+(\L_\oE^\PM)\cap P^\PM$. 
D'autre part $\calW$ a pour base l'ensemble des fonctions $f_v$, $v \in J^1_\ext \cap U^\PMopp / H^1_\ext \cap U^\PMopp$, d\'efinies par leur support $  \, H^1 (J^1_\ext \cap U^\PM) \, v $ et leur valeur $1$ en $v$. 
 \begin{Lemma}\label{tech1}
Soit $u \in P^+ (\L_\oE^\PM) \cap U^\PM $, $x \in J^1_\ext (\L^\PM)\cap U^\PM$ et $v \in  J^1_\ext (\L^\PM) \cap U^\PMopp$.  
Le commutateur $[u, x]$ appartient \`a $H^1(\L^\PM) \cap G_\interieur^+$ et 
le commutateur $[u, v]$ appartient \`a $H^1(\L^\PM)( J^1_\ext (\L^\PM) \cap U^\PM)$. 
\end{Lemma} 
\begin{proof} Le commutateur de deux \'el\'ements de $U$ 
appartient \`a $G_\interieur^+$ (v\'erification matricielle imm\'ediate) donc $[u, x]$ appartient \`a $J^1_\ext (\L^\PM)\cap U^\PM \cap G_\interieur^+ \subset H^1  (\L^\PM)\cap U^\PM \cap G_\interieur^+$.  
Le commutateur $[u, v]$ appartient \`a $J^1_\ext (\L^\PM)$, il suffit de montrer qu'il est ortho\-gonal \`a 
$J^1_\ext (\L^\PM)\cap U^\PM$. On calcule donc  
pour $x \in J^1_\ext (\L^\PM)\cap U^\PM$~: 
$$
\theta^\PM([uvu^{-1}v^{-1},x])=
\theta^\PM([uvu^{-1},x][v^{-1},x]) 
=\theta^\PM([v,u^{-1}xu]) \theta^\PM([v,x^{-1}]) 
=\theta^\PM([v, u^{-1}xux^{-1}]) 
$$
qui vaut $1$ 
puisque $[u^{-1},x]$ appartient \`a $H^1(\L^\PM)$. 
\end{proof}
On a donc pour  $u \in P^+(\L_\oE^\PM)\cap U^\PM$ :
$ \Omega(u) f_v = \theta^\PM_{U^\PM}([u^{-1}, v]) f_v$. En outre $\Omega(u)$ envoie fonction paire sur fonction paire, d'o\`u : 
$ \theta^\PM_{U^\PM}([u^{-1}, v]) =   \theta^\PM_{U^\PM}([u^{-1}, v^{-1}])$ pour $v \in J^1_\ext \cap U^\PMopp$. Il en r\'esulte   $ \det \Omega(u)= \det \Omega^+(u)^2 $, donc 
$\mathbb W (u)$ est \'egal \`a $ \Omega(u)$ et fixe $f_1$.

Etudions maintenant les op\'erateurs $\Omega(x)$ pour  $x $ \'el\'ement de $ P^+(\L_\oE^\PM)\cap M$, qui normalise $J^1_\ext \cap U^\PMopp$ et $H^1_\ext \cap U^\PMopp$. On  
  a $\Omega(x) f_v = f_{xvx^{-1}}$, autrement dit ces op\'erateurs agissent 
par permutation de la base des $f_v$ et leur d\'eterminant est la signature de la permutation correspondante. Il en est de m\^eme pour $\cal W^+$, de base les $f_v + f_{v^{-1}}$, donc 
  $\det \Omega^+(x)^2=1$. 
En conclusion : 

\begin{Proposition}\label{Weil}
Soit $\mathbf w_U$ le caract\`ere de $P^+(\L_\oE^\PM)\cap P$,  
 trivial sur  $P^+(\L_\oE^\PM)\cap U$, 
  qui \`a $x \in P^+(\L_\oE^\PM)\cap M $ 
associe la signature de la permutation  $v \mapsto xvx^{-1}$ de $J^1_\ext(\L^\PM) \cap U^\PMopp / H^1(\L^\PM) \cap U^\PMopp$. 
Il existe une 
  repr\'esentation $\mathbb W$ de $P^+(\L_\oE^\PM) / P_1(\L_\oE^\PM)$ dans l'espace   de $\eta_\ext^\PM$,  entrela\c cant   $\eta_\ext^\PM$, dont la   
	 restriction   \`a $P^+(\L_\oE^\PM)\cap P$ est donn\'ee, dans le mod\`ele (\ref{espace}) de 
	$\eta_\ext^\PM$, par : 
	$$
	\mathbb W (g) (f) (x) = \mathbf w_U (g) f(g^{-1}x g)  \quad (g \in P^+(\L_\oE^\PM)\cap P, 
	\  f \in \calW, \   x \in J^1_\ext). 
	$$ 
	 En particulier, le caract\`ere $\mathbf w_U$  est l'action de $P^+(\L_\oE^\PM)\cap P$ sur la droite des   $J^1_\ext \, \cap U^\PM$-invariants de $\eta_\ext^\PM$.  
\end{Proposition} 

\begin{Remark}
{\rm 
On obtient un caract\`ere    $\mathbf w_U$   \`a valeurs dans $\{\pm 1\}$,  donc trivial sur tous les sous-groupes d'ordre 
impair, en particulier sur tous les $p$-sous-groupes.
Rappelons aussi que $J^1_\ext(\L^\PM)$   ne d\'epend pas de $\theta$ 
(Lemme \ref{intersection}), donc   
  $\mathbf w_U$ n'en d\'epend pas non plus. 
}
\end{Remark}

\begin{Remark} {\rm Au lieu d'\'etudier directement la repr\'esentation de $P^+(\L_\oE^\PM)/P_1(\L_\oE^\PM)$ comme ci-dessus, on aurait pu utiliser l'homomorphisme naturel de ce groupe dans le groupe symplectique 
$\cal S$ de $ J^1_\ext \cap U^\PMopp / H^1 \cap U^\PMopp \oplus J^1_\ext \cap U^\PM  / H^1 \cap U^\PM $ et se ramener \`a la repr\'esentation de Weil de $\cal S$   telle qu'elle est d\'ecrite par G\'erardin  \cite{G}. 
Cette approche s'av\`ere plus compliqu\'ee. D'une part il n'est  pas clair que l'action naturelle de $P^+(\L_\oE^\PM)/P_1(\L_\oE^\PM)$ sur le groupe d'Heisenberg 
correspondant nous donne une section conjugu\'ee \`a la section standard (de   $\cal S$ dans le groupe des automorphismes du groupe d'Heisenberg). 
D'autre part la connaissance du caract\`ere du stabilisateur de $J^1_\ext \cap U^\PM  / H^1_\ext \cap U^\PM $ dans $\cal S$ par lequel il agit sur un invariant de son radical unipotent 
ne suffit  pas \`a d\'ecrire sa restriction \`a  $P^+(\L_\oE^\PM) \cap M$. 
  }
\end{Remark}

\subsection{Lien avec les $\beta$-extensions}\label{2.4}

Nous allons \`a pr\'esent appliquer le corollaire \ref{Eext} pour   comparer $\beta$-extensions dans $G^+$ et dans $G_\interieur^+$. 
Les deux ingr\'edients n\'ecessaires sont la restriction de Jacquet, qui interviendra via la 
proposition \ref{eextU} et qui commute aux bijections canoniques (proposition \ref{diagcom2}), 
et un bon choix de la repr\'esentation $\mathcal X $ (proposition \ref{Weil}). 
Nous nous pla\c cons dans  des hypoth\`eses communes \`a ces r\'esultats : 
\begin{enumerate}
	\item 
$V = W^{(-1)} \oplus W^{(0)} \oplus  W^{(1)}$ est une d\'ecomposition autoduale de $V$
proprement subordonn\'ee \`a la strate 
$[\L,n,0,\b]$, donc  subordonn\'ee \`a  
$[\L^M,n^M,0,\b]$, et telle que $W^{(0)}$ soit somme de certains $V^i$  ;   $Z_1 =W^{(-1)}  \oplus  W^{(1)} $, $Z_2 = W^{(0)}$. 
 
La strate 
  $[\L,n,0,\b]$ 
est   donc somme directe des strates semi-simples  $[S_1 ,n_1,0,\b_1]$ et $[S_2 ,n_2,0,\b_2]$
 avec   $S_j= \L\cap Z_j $, $j=1,2$. Il en est de m\^eme pour la strate   $[\L^M,n^M,0,\b]$,  somme directe des strates semi-simples $[S_1^M ,n_1^M,0,\b_1]$ et $[S_2^M ,n_2^M,0,\b_2]$. 

 \item    $P$ est le sous-groupe parabolique de $G^+$ stabilisant le drapeau $\{0\} \subset  W^{(-1)} \subset 
W^{(-1)} \oplus W^{(0)}
 \subset W^{(-1)} \oplus W^{(0)} \oplus W^{(1)}$, 
 $U$ est  son radical unipotent,  $M$ est le fixateur de la d\'ecomposition, $P^-$ et $U^-$ 
 sont les oppos\'es de $P$ et $U$ par rapport \`a $M$. 
 \item On demande $P^+ (\L_\oE ) = (P^+ (\L_\oE^M) \cap P ) \ P_1(\L_\oE^M)$.
 \end{enumerate}  
 Rappelons que $\fb_0(\L^M)$ est un $\oE$-ordre autodual maximal. On applique les r\'esultats du paragraphe pr\'ec\'edent \`a $\L^M$  en ajoutant l'exposant $M$ aux notations si n\'ecessaire. Alors : 
 \begin{Proposition}\label{casmaximal} 
 Soit $\mathbb W^M$ 
  la repr\'esentation   de $P^+(\L_\oE^M) $ dans l'espace  de $\eta_\ext^M$   
d\'efinie dans la proposition   \ref{Weil} et soit $ \, \Interieur_{\mathbb W^M}$ la bijection correspondante.  Alors $\kappa^M$,  prolongement  de $\eta^M$ \`a $J^+(\L^M)$, est une $\beta$-extension si et seulement si 
$\Interieur_{\mathbb W^M} (\kappa^M)$, prolongement  de $\eta_\interieur^M$ \`a $J^+_\interieur (\L^M)$,  
est une $\beta$-extension. 
 \end{Proposition}
 \begin{proof} 
 Il suffit de montrer, pour une  suite autoduale $\L^m$ de $\oE$-r\'eseaux 
 telle que  $\fb_0(\L^m)$ soit un $\oE$-ordre autodual minimal contenu dans 
  $\fb_0(\L^M)$, que les conditions suivantes sont \'equivalentes : 
\begin{enumerate}
	\item $\kappa^M$ est un prolongement de $\eta(\L^m, \L^M)$ ; 
	\item $\Interieur_{\mathbb W^M} (\kappa^M)$ est un prolongement de  $\eta_\interieur(\L^m, \L^M)$. 
\end{enumerate}
 On choisit, comme on le peut, une  suite  $\L^m$  
  telle que la d\'ecomposition 
	 $V = W^{(-1)} \oplus W^{(0)} \oplus  W^{(1)}$ soit proprement subordonn\'ee \`a la strate
	 $[\L^m,n^m,0,\b]$ et telle que  $\fb_0(\L^m) \subset \fb_0(\L )$. 
\begin{itemize}
	\item L'hypoth\`ese (iii) entra\^\i ne 
	$P^+ (\L_\oE ) = (P^1 (\L_\oE^M) \cap U^- ) (P^+ (\L_\oE^M) \cap M  )(P^+ (\L_\oE^M) \cap U  )$ donc $P_1(\L_\oE^m)\cap U^- \subset 
  P^1 (\L_\oE^M) \cap U^-$. Comme   
  $P^1 (\L_\oE^M)   \subset P^1 (\L_\oE^m)$, on a  en fait  
	$P_1(\L_\oE^m)\cap U^- =  P_1(\L_\oE)\cap U^- =P_1(\L_\oE^M)\cap U^-$, contenu dans $J^1(\L^\PM)$. 
	
	\item Rappelons que $\eta(\L^m, \L^M)$ est une repr\'esentation de $ P^1 (\L_\oE^m) J^1(\L^M)$,   poss\`edant une d\'ecomposition d'Iwahori relative \`a $(P,M)$ et 
	d'intersection avec $U^-$  :  $J^1(\L^M) \cap U^-$. On peut comme au paragraphe \ref{subord} d\'efinir la repr\'esentation $\eta_P(\L^m, \L^M)$ du groupe  $ (P^1 (\L_\oE^m)\cap P) J^1_P(\L^M)$
	: c'est  l'action sur les 
	$J^1(\L^M) \cap U $-invariants de $\eta(\L^m, \L^M)$. On v\'erifie que $\eta(\L^m, \L^M)$ 
	est la repr\'esentation induite de $\eta_P(\L^m, \L^M)$. 
	\item Le groupe $P^+(\L_{\oE}  ) J^1(\L^M)$ contient $P^+(\L_{\oE}^m ) J^1(\L^M)$, de sorte 
	que la condition voulue sur $\kappa^M$ peut se voir comme  une condition sur sa restriction \`a 
	$P^+(\L_{\oE}  ) J^1(\L^M)$, \`a laquelle  le lemme \ref{invariantsbis} s'applique. 
	Ainsi 
	$\kappa^M$ prolonge   $\eta(\L^m, \L^M)$ si et seulement si 
	$\kappa^M_P$ prolonge   $\eta_P(\L^m, \L^M)$. 
	\item Les m\^emes constructions s'appliquent dans  $G^+_\interieur $  aux objets induits dans les 
	groupes d'isom\'etries de $Z_1$ et $Z_2$ ; on continue d'utiliser la notation $P$ 
	au lieu d'utiliser une notation sp\'ecifique pour le parabolique  $P \cap G_\interieur^+$. 
	 Elles aboutissent \`a :
$\Interieur_{\mathbb W^M} (\kappa^M)$ prolonge $\eta_\interieur(\L^m, \L^M)$   si et seulement si 
	$\Interieur_{\mathbb W^M} (\kappa^M)_P$ prolonge   $\eta_\interieur(\L^m, \L^M)_P$. 
	\item On peut d\'ecrire $\eta_{P}(\L^m, \L^M)$ de fa\c con analogue   \`a la proposition \ref{etamMP} ci-dessus et \`a \cite{S5}, 
preuve de la proposition 6.3 : c'est la repr\'esentation de $ (P^1 (\L_\oE^m)\cap P) J^1_P(\L^M)$ prolongeant $\eta_{P}(  \L^M)$ qui est triviale sur $  P^1 (\L_\oE^m)\cap U$ 
et \'egale \`a  $\eta({\L^m}^{(0)}, {\L^M}^{(0)})  \otimes   
\widetilde \eta({\L^m}^{(1)}, {\L^M}^{(1)}) $ 
sur $  P^1 (\L_\oE^m)\cap M$. Il en est de m\^eme pour    $\eta_\interieur(\L^m, \L^M)_P$ avec les 
modifications \'evidentes. 
\end{itemize}
 Les conditions 	``$\kappa^M_P$ prolonge   $\eta_P(\L^m, \L^M)$'' et 
	``$\Interieur_{\mathbb W^M} (\kappa^M)_P$ prolonge   $\eta_\interieur(\L^m, \L^M)_P$'' 
  se lisent sur $J_{\interieur, P}^+(\L^M)$. Par 
		    la proposition \ref{eextU}, la restriction de $\kappa^M_P$ \`a $J_{\interieur, P}^+(\L^M)$ est \'equivalente \`a $\mathbf w^M_U \Interieur_{\mathbb W^M} (\kappa^M)_P$. 
Comme $\mathbf w_U^M $ est trivial sur le ($p$-)groupe de d\'efinition de $\eta_\interieur(\L^m, \L^M)_P$, 
 on a le r\'esultat annonc\'e. 
 \end{proof}
\begin{Remark}
{\rm 
On montre de la m\^eme fa\c con que  $\kappa^M$   prolonge  $\eta(\L, \L^M)$     si et seulement si 
$\Interieur_\mathbb W (\kappa^M)$  prolonge   $\eta_\interieur(\L, \L^M)$.
}
\end{Remark}

Nous pouvons enfin rassembler tous les r\'esultats et d\'ecrire le lien entre 
  $\b$-extensions dans $G^+$ et dans $G_\interieur^+$    en termes de l'application
  $\mathbf r_{P}$ du paragraphe \ref{Jacquet} et du caract\`ere $\mathbf w_U^M$ de la proposition  \ref{Weil}  appliqu\'ee \`a la suite $\L^M$. Le caract\`ere $\mathbf w_U^M$   de  $P^+(\L_\oE^M )\cap P/ P^1(\L_\oE^M )\cap P$ s'identifie \`a un  caract\`ere de $P^+ (\L_\oE ) = (P^+ (\L_\oE^M) \cap P ) \ P_1(\L_\oE^M)$, trivial sur  $P^1 (\L_\oE )  $. 
  On rappelle que $P^+ (\L_\oE )= (P^1 (\L_\oE ) \cap U^-) (P^+ (\L_\oE ) \cap M)(P^1 (\L_\oE )\cap U)$. 
 
 \begin{Theorem}\label{theoreme}
  Soit $ \kappa(\L)$ un prolongement de $ \eta(\L)$ \`a $J^+(\L)$ et 
 $\kappa (S_1)\otimes \kappa(S_2)$ un prolongement de   $\eta (S_1)\otimes \eta(S_2)$ \`a 
 $J^+_\interieur(\L)= J^+(S_1) \times J^+(S_2)$, reli\'es par la condition 
   $$\mathbf r_{P}\left(\kappa (\L)\right)= \mathbf r_{P} \left(
 \kappa (S_1)\otimes \kappa (S_2)\right).$$ 
 Soit $\mathbf w_U^M$ le caract\`ere de $P^+(\L_\oE )$, trivial sur $P^1(\L_\oE )$,     
  qui \`a $x \in P^+(\L_\oE  )\cap M $ 
associe la signature de la permutation  $v \mapsto xvx^{-1}$ de $J^1_\ext(\L^M) \cap U^- / H^1(\L^M) \cap U^-$. 

\noindent
 La repr\'esentation $\kappa(\L)$ est une $\b$-extension de 
 $ \eta(\L)$ relativement \`a $\L^M$   si et seulement si la  repr\'esentation $ \mathbf w_U^M (\kappa (S_1)\otimes \kappa(S_2))$ est une $\b$-extension de 
 $\eta (S_1)\otimes \eta(S_2)$ relativement \`a $\L^M$. 
 \end{Theorem}
 \begin{proof} Elle s'appuie sur le diagramme suivant : 
   $$
 \aligned 
  \{\text{Prolongements de } \eta(\L) \text{ \`a } J^+(\L)\}   
  &\stackrel{\mathfrak B }{\longleftrightarrow}    
  \{\text{Pr. de } \eta(\L,\L^M) \text{ \`a }  P^+(\L_{\oE} ) J^1(\L^M)\}
  \\
 \mathbf r_P\downarrow \qquad \qquad \qquad \qquad 
 &\qquad \qquad \qquad    \mathbf r_{P} \downarrow 
  \\
    \{\text{Pr. de } {\eta_P(\L)}_{| J^1(\L ) \cap M} \text{ \`a } J^+(\L ) \cap M\}  
  &\stackrel{\mathfrak B}{\longleftrightarrow}   
  \{\text{Pr. de } {\eta_{P|M}(\L , \L^M)} \text{ \`a } J^+(\L^M)   \cap \!  M\}  
  \\
   \ \downarrow \qquad \qquad \qquad \qquad 
 &\qquad \qquad \qquad     \Phi \  \downarrow 
   \\
    \{\text{Pr. de } {\eta_P(S_1)}_{| J^1(S_1 )\cap M} \otimes \eta(S_2) \text{ \`a } J^+_\interieur(\L ) \cap M\}  
  &\stackrel{\mathfrak B}{\longleftrightarrow}   
  \{\text{Pr. de }\eta_{P|M}(S_1,S_1^M)\otimes \eta(S_2, S_2^M)   \\
&\qquad \qquad \qquad \qquad \qquad \text{\`a } J^+_\interieur(\L^M)   \cap \!  M\} 
  \\
 \mathbf r_P\uparrow \qquad \qquad \qquad \qquad 
 &\qquad \qquad \qquad     \mathbf r_{P} \uparrow 
 \\
    \{\text{Pr. de } \eta(S_1)\otimes \eta(S_2) \text{ \`a } J^+_\interieur(\L)\}   
  &\stackrel{\mathfrak B }{\longleftrightarrow}    
  \{\text{Pr. de } \eta(S_1, S_1^M)\otimes \eta(S_2,S_2^M) 
  \\
&\qquad \qquad \qquad \qquad \qquad  \text{\`a }  P^+(\L_{\oE} ) J^1_\interieur(\L^M)\}
  \endaligned
 $$
 Expliquons ce diagramme. Les deux premi\`eres lignes reproduisent le diagramme commutatif de  la proposition  \ref{diagcom2} appliqu\'ee \`a $\L$ et $\L^M$, avec \'echange des lignes et des colonnes. 
 Les deux derni\`eres lignes reproduisent le m\^eme diagramme commutatif dans le groupe $G_\interieur ^+$, avec les paires $(S_1, S_1^M)$ et $(S_2, S_2^M)$, et invers\'e ; noter que 
 $\eta(S_2, S_2^M)= \eta_{P|M}(S_2, S_2^M)  $.

 Soit de nouveau $\mathbb W^M$ 
  la repr\'esentation   de $P^+(\L_\oE^M) $ dans l'espace  de $\eta_\ext^M$   
d\'efinie dans la proposition   \ref{Weil}, soit $ \, \Interieur_{\mathbb W^M}$ la bijection correspondante et soit $\Phi$ l'isomorphisme associ\'e par la proposition \ref{eextU}. Alors 
$\Phi$ entrelace des repr\'esentations de $J^+_{\interieur,P}(\L^M)$,  en particulier  
 de 
$J^+_{\interieur,P}(\L^M) \cap M = J^+_\interieur(\L^M)   \cap   M = J^+(\L^M ) \cap M 
= (J^+(S_1^M) \cap M) \times J^+(S_2^M)  $ (cf  (\ref{JP})),  et $\Phi$ 
est le produit tensoriel par $(\mathbf w_U^M)^{-1}= \mathbf w_U^M$. La compos\'ee de haut en bas, soit  $\mathbf r_P^{-1} \circ \Phi \circ \mathbf r_P$, des applications verticales de droite du diagramme   est la bijection $ \, \Interieur_{\mathbb W^M}$.

Partons alors de $ \kappa(\L)$, prolongement de $ \eta(\L)$ \`a $J^+(\L)$. C'est une $\b$-extension relativement \`a $\L^M$ si et seulement si son image 
$\mathfrak B (\kappa(\L))$, prolongement de $\eta(\L,\L^M) $ \`a $ P^+(\L_{\oE} ) J^1(\L^M)$, 
se prolonge \`a $J^+(\L^M)$ en une $\b$-extension de $ \eta(\L^M)$. Puisque $\mathbb W^M$ 
  est une  repr\'esentation   de $P^+(\L_\oE^M) $ tout entier, cette condition \'equivaut 
  \`a :  $ \, \Interieur_{\mathbb W^M}(\mathfrak B (\kappa(\L)))$ se prolonge \`a $J^+_\interieur(\L^M)$ en une $\b$-extension de $ \eta_\interieur(\L^M)$ (proposition \ref{casmaximal}). Ceci signifie exactement que :
  
  {\it $ \mathfrak B^{-1}[ \, \Interieur_{\mathbb W^M}(\mathfrak B (\kappa(\L)))]$ est une $\b$-extension de 
  $ \eta(S_1)\otimes \eta(S_2)$ relativement \`a $\L^M$. }
  
  Il reste \`a faire le lien avec les applications $\mathbf r_P$ de la colonne de gauche.  L'application faisant commuter le diagramme du milieu est 
  encore le produit tensoriel par $ \mathbf w_U^M $ puisque la bijection canonique est compatible \`a la torsion par les caract\`eres de $ P^+(\L_{\oE} )/   P^1(\L_{\oE} )$ 
  (lemme \ref{Faits}). Autrement dit : 
  
  $  \mathbf r_P (\kappa(\L)) = \mathbf w_U^M \ \mathbf r_P \left( \mathfrak B^{-1}[ \, \Interieur_{\mathbb W^M}(\mathfrak B (\kappa(\L)))]\right)= \mathbf r_P\left(  \mathbf w_U^M \, \mathfrak B^{-1}[ \, \Interieur_{\mathbb W^M}(\mathfrak B (\kappa(\L)))]\right)  .$
 \end{proof} 
 
 \begin{Example}\label{exemplebeta}
 {\rm
 Pla\c cons-nous dans un groupe symplectique avec $\dim W^{(1)} = 1$, de sorte que 
 $G \simeq \Sp(2N+2, F)$ et $G_\interieur \simeq \Sp( 2, F)\times \Sp(2N , F)$. Prenons dans 
 $Z_1$ la strate nulle $[S_1 ,0,0,0]$  attach\'ee \`a une suite de r\'eseaux dont le fixateur est le sous-groupe d'Iwahori standard $I$ de  $\Sp( 2, F)$, dans $Z_2$ une strate semi-simple sans composante nulle  $[S_2 ,n_2,0,\b_2]$. Alors  $\eta(S_1)$ est la repr\'esentation triviale de $J^1(S_1)$,   radical pro-unipotent de $J(S_1)=I$. Un prolongement  $\kappa (S_1)$ de  $\eta(S_1)$ \`a 
 $I$ est un caract\`ere  $\Xi$  de $I$ de la forme  $\Xi\left(\begin{smallmatrix} a & b \cr c & d \end{smallmatrix}\right) = \xi(a)$ pour  un caract\`ere $\xi$ de $\oFx$ trivial sur $ 1 + \pF$. Un tel caract\`ere est
 une $\b$-extension relativement \`a une des deux suites de r\'eseaux minimales 
 $S_1^+$ et $S_1^-$ 
 extraites de $S_1$ si et seulement si $\kappa (S_1)$ est le caract\`ere trivial de $I$ 
(Corollaire \ref{intertwiningeta}, Proposition \ref{intertwining}).

 Soient $\xi $ et $\Xi$ comme ci-dessus, soit $\kappa (S_2)$   un prolongement de 
$\eta (S_2)$ \`a $J(S_2)$ et 
soit $\kappa_P(\L) $ le prolongement  de $\eta_P(\L)$ \`a $J_P(\L)$   dont la restriction \`a  $J_P(\L)\cap M$ 
est $\xi \otimes \kappa (S_2)$. Le th\'eor\`eme 
\ref{theoreme} dit que $\Ind_{J_P(\L)}^{J (\L)} \kappa_P(\L) $ est une $\b$-extension relativement \`a la suite de r\'eseaux   $\L^M = S_1^+ \perp S_2$ si et seulement si 
$ \mathbf w_U^M (\Xi \otimes \kappa(S_2))$ est une $\b$-extension de 
 $1_{J^1(S_1)}\otimes \eta(S_2)$ relati\-vement \`a $\L^M$. D'apr\`es nos hypoth\`eses  $\fb_0(S_2)$ est un  ordre autodual maximal, de sorte que $\mathbf w_U^M  \kappa(S_2) $ est une $\b$-extension si et seulement si $\kappa(S_2) $ en est une (D\'efinition \ref{betaextension} ; $\mathbf w_U^M $ 
 est trivial sur les sous-groupes unipotents). Il reste : 
 
  {\it $\Ind_{J_P(\L)}^{J (\L)} \kappa_P(\L) $ est une $\b$-extension relativement \`a   $\L^M  $ si et seulement si 
$  \Xi = \mathbf w_U^M   $. }

La d\'etermination de $\mathbf w_U^M   $ n\'ecessite d'expliciter le quotient  $J^1_\ext(\L^M) \cap U^- / H^1(\L^M) \cap U^-$. A titre d'exemple, supposons que l'\'el\'ement $\b_2$ engendre une extension totalement ramifi\'ee de degr\'e $2N$ dans laquelle il a pour valuation 
$-(2Nk +1)$, $k \in \mathbb N$~; il est en particulier minimal. On obtient 
 $J^1_\ext(\L^M) \cap U^- $ sous la forme de  matrices par blocs 
 $\left(\begin{smallmatrix} 1 & 0& 0 \cr Y & I & 0 \cr z & X & 1\end{smallmatrix}\right)$ appartenant \`a $\Sp(2N+2, F)$ avec $z \in \pF$ et $X \in \left( \pF^{k+1}, \cdots, \pF^{k+1}, \pF^{k}, \cdots, \pF^{k}\right)$ ($N$ fois $ \pF^{k+1}$ et $N$ fois $\pF^{k}$). La description de  $ H^1(\L^M) \cap U^-$ 
 est analogue avec cette fois $N+1$ fois $ \pF^{k+1}$ et $N-1$ fois $\pF^{k}$, de sorte que le quotient cherch\'e est isomorphe \`a $\oF/ \pF$. Pour $x \in \oFx$, l'action de $\iota(x)
 \in P^+(\L_\oE  )\cap M $ est la multiplication par $x^{-1}$ dans $\oF/ \pF$, la signature de cette permutation est l'unique caract\`ere d'ordre $2$ de $k_F^\times$. En particulier, 
  {\it $\Ind_{J_P(\L)}^{J (\L)} \kappa_P(\L) $ est une $\b$-extension relativement \`a   $\L^M  $ si et seulement si 
$  \xi^2=1 $ et $\xi \ne 1$. }
 
 Le cas $\L^M = S_1^- \perp S_2$ est semblable.
 }
 \end{Example}
 
\setcounter{section}{2}

\Section{Application \`a la r\'eductibilit\'e}\label{S5}

 \subsection{Une paire couvrante et son alg\`ebre de Hecke}\label{pc} 
 
Soit $[\L ,n ,0,\b ]$ une strate semi-simple gauche et    $V = W^{(-1)} \oplus W^{(0)}\oplus W^{(1)} $ une d\'ecomposition autoduale de $V$ telle que $W^{(0)}$ contienne  tous les $V^i$ sauf un ;  
quitte \`a r\'eordonner les $V^i$ on peut donc supposer $W^{(0)} = V^1\cap W^{(0)} \perp V^2 \perp \cdots \perp V^{l}$ et 
$\left( W^{(-1)}  \oplus W^{(1)}\right) \perp  V^1\cap W^{(0)} = V^1$. 
On demande que cette d\'ecomposition soit 
{\it exactement subordonn\'ee} \`a la strate $[\L ,n ,0,\b ]$, c'est-\`a-dire     proprement subordonn\'ee et telle que   $\mathfrak b_0(\L^{(0)})$ soit un $\oE$-ordre autodual maximal de 
$\End_F(W^{(0)})\cap B$ et $\mathfrak b_0(\L^{(1)})$   un ${\mathfrak o}_{E_1}$-ordre maximal de 
$\End_{E_1}(W^{(1)}) $  \cite[Definition 6.5]{S5}.   Soit 
  $e_1$ la p\'eriode sur $E_1$ de la suite de r\'eseaux $\L^1$ ;   
quitte \`a \'echanger $W^{(-1)}$  et $ W^{(1)}$, on suppose  comme dans \cite[\S 6.2]{S5}  que l'unique entier $q_1$ v\'erifiant $-\frac{e_1}{2} < q_1 \le \frac{e_1}{2} $ et 
$\L^{(1)}(q_1 +1) \subsetneq \L^{(1)}(q_1  )$ est positif ou nul ; comme la d\'ecomposition est proprement subordonn\'ee on a de fait $q_1 > 0$.

Soit $P$   le sous-groupe parabolique de $G $ stabilisant le drapeau $ \   W^{(-1)} \subset 
W^{(-1)} \oplus W^{(0)}
 \subset V$, 
 $U$    son radical unipotent,  $M$  le fixateur de la d\'ecomposition autoduale, $P^-$ et $U^-$ 
   les oppos\'es de $P$ et $U$ par rapport \`a $M$.
La d\'ecomposition $V = W^{(-1)} \oplus W^{(0)}\oplus W^{(1)} $  d\'etermine des injections des  alg\`ebres $\End_F(W^{(i)})$, $i=0, 1, -1$, dans $\End_F(V)$, obtenues par un prolongement nul sur les deux autres sous-espaces. 
L'involution adjointe $g \mapsto \bar g$ sur 
 $\End_F(V)$ se restreint \`a $\End_F(W^{(0)})$ en une involution not\'ee de la m\^eme fa\c con, 
 et induit un isomorphisme not\'e  $x \mapsto \hat x$ de $\End_F(W^{(1)})$ sur  $\End_F(W^{(-1)})$.   
On identifie le groupe $\GL_F(W^{(1)})= \widetilde G^{(1)}$ \`a un sous-groupe de $G $ au moyen du monomorphisme 
$\iota$ qui \`a $g \in \widetilde G^{(1)}$ associe 
$\iota(g)= (\hat g^{-1}, I_{W^{(0)}} ,  g  )  \in  G  \cap M$.  
On note $G^{(0)} = G \cap \GL_F(W^{(0)})$ ; ainsi $M \simeq G^{(0)} \times \iota(\widetilde G^{(1)})$. 

Soit $E_1^0$ le corps des points fixes de l'involution adjointe dans $E_1$ et $e_0$   l'indice de ramification de $E_1$ sur $E_1^0$ ; 
 fixons une uniformisante   $\varpi_1$ de $E_1$ telle que 
$\bar \varpi_1 = (-1)^{e_0-1} \varpi_1$. 
  Notons enfin $f_1$ la 
 $E_1/E_1^0$-forme $\epsilon$-hermitienne non d\'eg\'en\'er\'ee sur $V^1$ associ\'ee \`a la restriction de $h$. 
 Suivant toujours \cite[\S 6.2]{S5} on choisit une base ordonn\'ee $\mathcal B^{(1)}= \{v_1, \cdots, v_d\}$ du r\'eseau 
$\L^{(1)}(0)$ sur ${\mathfrak o}_{E_1}$. Les \'el\'ements $v_{-i}$, $1 \le i \le d$, 
d\'efinis par 
$ 
f_1(v_j, v_{-i}) = \varpi_1 \delta_{ij} 
$ ($1 \le j \le d$)  constituent une base ordonn\'ee $\mathcal B^{(-1)}= \{v_{-1}, \cdots, v_{-d}\}$ du r\'eseau 
$\L^{(-1)}(1)$ sur ${\mathfrak o}_{E_1}$. On 
d\'efinit  les \'el\'ements $I_{1, -1}$, $I_{-1,  1}$ de $B_1$ et  $s_1$ et 
$s_1^\varpi$ de $G_E^+$  comme suit : 
$$
\aligned 
 &I_{-1,  1}\, (v_i)   = v_{-i} \ \text{ et } \  I_{-1,  1}\, (v_{-i}) = 0 \ \text{ pour } 1 \le i \le d \  ; 
 \\  
  &I_{ 1,  -1}\, (v_{-i})   = v_i \  \text{ et } \  I_{ 1,  -1}\, (v_{ i}) = 0 \  \text{ pour } 1 \le i \le d  \  ; 
  \\ 
  &s_1  = I_{-1,  1} + \epsilon (-1)^{e^0-1} \,  I_{ 1,  -1} + I_{W^{(0)}} \   ;  \quad   
 s_1^\varpi  = \varpi_1^{-1} I_{-1,  1} + \epsilon \varpi_1 \, I_{ 1,  -1} + I_{W^{(0)}}. 
\endaligned 
$$
Aucun des \'el\'ements $s_1$ et $ s_1^\varpi$ n'appartient \`a 
 $P^+(\L_\oE)$ 
 car la d\'ecomposition est exactement subordonn\'ee \cite[Definition 6.8, Lemma 6.7]{S5}, 
 tandis que pour toute suite autoduale $\L^M$ de $\oE$-r\'eseaux telle que $
 \mathfrak b_0(\L^M)$ soit un $\oE$-ordre autodual maximal contenant $\mathfrak b_0(\L)$, 
 l'un de ces deux \'el\'ements appartient \`a   $P^+(\L^M_\oE)$  \cite[fin du \S 6.2] {S5}.    
Remarquons que  
$$s_1^2 = \iota\left( \epsilon (-1)^{e^0-1} I_{W^{(1)}}\right), \  
 (s_1^\varpi)^2=  \iota\left( \epsilon   I_{W^{(1)}} \right), \  s_1^\varpi s_1  = \iota\left(\epsilon \varpi_1\right) , \ 
 s_1 s_1^\varpi = \iota\left( \epsilon (-1)^{e^0-1} \varpi_1^{-1} \right) .$$
On obtient un  automorphisme involutif $\sigma_1$ de $\widetilde G^{(1)}$ en posant :
$ \  
\iota(\sigma_1(g))= s_1 \iota(g) s_1^{-1}. 
$ 
  \begin{Lemma}[\bf  \cite{S5}   Lemma 7.11, Corollary 7.12] \label{s1pi}
  Soit $J_P = J_P  (\L)= H^1 (\L) (J  (\L)\cap P)$. 
Posons $ \ \zeta =\iota\left(   \varpi_1^{-1} \right)$ et $\zeta^\prime = \iota\left(   \varpi_F^{-1} \right)$ o\`u $\varpi_F$ est une uniformisante de $F$. Alors 
$\zeta^\prime$ est un \'el\'ement fortement positif 
ou fortement n\'egatif du centre de $M$ et on~a  
$$
J_P s_1 J_P  s_1^\varpi J_P  = J_P  \zeta J_P  \quad \text{ et } 
\quad 
(J_P \zeta J_P )^{e(E_1/F)} = J_P   \zeta^{e(E_1/F)} J_P  = 
J_P  \zeta^\prime J_P  . 
$$
\noindent
Les m\^emes r\'esultats sont valides en rempla\c cant $J_P  (\L)$ par $J_{P^-}  (\L)$.  
 \end{Lemma}

\begin{proof} On reprend la d\'emonstration de   \cite[Lemma 7.11]{S5} en notant 
$J^1 = J^1(\L)$ etc. Le 
   changement de $J_P^0$ en $J_P $ ou $J_P^+$ n'affecte que l'intersection avec $M$, normalis\'ee par 
  $s_1$ et $s_1^\varpi$,  
  la conjugaison par $s_1$ ou $s_1^\varpi$ \'echange $U $ et $U^-$, et il suffit en fait de travailler dans $\widetilde G$.  Notons $e^{(i)}$ la projection de $V$ sur $W^{(i)}$ de noyau la somme des $W^{(j)}$ pour $j \ne i$ et 
d\'etaillons les calculs de {\it loc. cit.}. On \'ecrit $\tilde J^1 \cap U^- = 1 +e^{(1)}\mathfrak J e^{(0)}+ e^{(1)}\mathfrak J e^{(-1)}+ e^{(0)}\mathfrak J e^{(-1)}$ d'o\`u 
$$
s_1(\tilde J^1 \cap U^-) s_1^{-1}  = 1 + I_{-1,  1}e^{( 1)}\mathfrak J e^{(0)} + I_{-1,  1}e^{( 1)}\mathfrak J I_{-1,  1} + e^{(0)} \mathfrak J I_{-1,  1}e^{( 1)} .
$$ 
Comme $e^{(i)}$ appartient \`a $\mathfrak b_0(\L)$,  $I_{-1,  1}e^{( 1)}$ appartient \`a 
$\mathfrak b_0(\L)$ 
et $\mathfrak b_1(\L) \mathfrak J \subset \mathfrak H$, on en d\'eduit 
$$
s_1(\tilde J^1 \cap U^-) s_1^{-1} \subset \tilde H^1 \cap U    $$
assertion un peu  plus forte que {\it loc. cit.}(ii).  
On a de m\^eme 
$$
(s_1^\varpi)^{-1}(\tilde J^1 \cap U ) s_1^\varpi \subset \tilde H^1 \cap U^- ,  
$$
qui est  l'assertion (iii) de  {\it loc. cit.}.  Ces  propri\'et\'es  suffisent \`a conclure puisque 
les groupes $J_{P } $ et $J_{P^-} $  sont  les groupes 
$(H^1 \cap U^-) (J  \cap M) (J^1 \cap U)$ et 
$(J^1 \cap U^-) (J  \cap M) (H^1 \cap U)$. 
\end{proof}

Partons d'un caract\`ere semi-simple $\theta$ de $H^1(\b, \L)$ et r\'esumons, d'apr\`es \cite[\S 5.3]{S5}, les propri\'et\'es  
 des objets associ\'es que voici : 
\begin{itemize}
	\item 
son prolongement $\theta_P$ \`a $H^1_P(\L)= H^1 (\L) (J^1 (\L)\cap U)$ trivial sur 
$J^1 (\L)\cap U$ ; 
\item l'unique repr\'esentation irr\'eductible $\eta_P$ de  $J^1_P(\L)= H^1 (\L) (J^1 (\L)\cap P)$ contenant $\theta_P$ ; 
\item  l'unique repr\'esentation irr\'eductible $\eta$ de 
$J^1 (\L)$ contenant $\theta$, qui v\'erifie  $\eta \simeq \Ind_{J^1_P(\L)}^{J^1 (\L)} \eta_P$~; 
\item un prolongement $\kappa$ de $\eta$ \`a $J^+ (\L)$ ; 
\item la repr\'esentation $\kappa_P$ de   $J^+_P(\L)= H^1 (\L) (J^+ (\L)\cap P)$ obtenue par restriction de   $\kappa$ \`a  l'espace des $J^1 (\L)\cap U$-invariants de $\eta$, qui prolonge   $\eta_P$ et v\'erifie $\kappa \simeq \Ind_{J^+_P(\L)}^{J^+ (\L)} \kappa_P$.  
\end{itemize}
  Les groupes consid\'er\'es ont tous une d\'ecomposition d'Iwahori relative \`a $(M,P)$, les repr\'esen\-ta\-tions $\theta_P$, $\eta_P$ et $\kappa_P$ sont triviales sur $H^1 (\L)\cap U^-$ 
et $J^1 (\L)\cap U= J^+ (\L)\cap U$ et leur composante en $M$ v\'erifie : 
$$
\aligned 
&{\theta_P}_{|H^1 (\L)\cap M} \ = \theta_{|H^1 (\L)\cap M} \ = \ \theta^{(0)} \otimes \tilde \theta^{(1)} ; 
\\
&{\eta_P}_{|J^1 (\L)\cap M} \ \simeq \ \eta^{(0)} \otimes \tilde \eta^{(1)} ; 
\\
&{\kappa_P}_{|J^+ (\L)\cap M} \ \simeq \ \kappa^{(0)} \otimes \tilde \kappa^{(1)} ;
\endaligned 
$$
o\`u  $\theta^{(0)}$ est un caract\`ere semi-simple gauche de 
$H^1 (\b^{(0)}, \L^{(0)} )$, $\eta^{(0)}$ est l'unique repr\'esentation irr\'eductible   de 
$J^1 (\b^{(0)}, \L^{(0)} )$ contenant $\theta^{(0)} $, $\kappa^{(0)}$ 
est un prolongement de $\eta^{(0)}$ \`a $J^+ (\b^{(0)}, \L^{(0)} )$~; 
  $\tilde \theta^{(1)}$ est un ca\-rac\-t\`ere semi-simple de $\widetilde H^1 (\b^{(1)}, \L^{(1)} )= \widetilde H^1 (2\b^{(1)}, \L^{(1)} )$  attach\'e \`a la strate 
$[\L^{(1)} ,n^{(1)} ,0,2 \b^{(1)} ]$, $\tilde \eta^{(1)}$ est l'unique repr\'esentation irr\'eductible   de   $\widetilde J^1 (\b^{(1)}, \L^{(1)} )$ contenant $\tilde \theta^{(1)}$
et $\tilde \kappa^{(1)}$ est un prolongement de  $\tilde \eta^{(1)}$  \`a $\widetilde J  (\b^{(1)}, \L^{(1)} )$. En outre : 
\begin{Lemma}[\bf  \cite{S5} Proposition 6.3, Lemma 6.9, Corollary 6.10]\label{equiv} 
L'involution $\sigma_1$ stabilise le 
  groupe $\widetilde J (\b^{(1)}, \L^{(1)} )$.  
De plus, si $\kappa$ est une $\beta$-extension de $\eta$, alors $\kappa^{(0)}$ est une $\b^{(0)}$-extension,  $\tilde \kappa^{(1)}$ est une 
 $2 \b^{(1)}$-extension \'equivalente \`a $\tilde \kappa^{(1)}\circ \sigma_1$ et 
$s_1$ et $s_1^\varpi$ entrelacent $\kappa_P$.  
\end{Lemma}

 Les groupes $J (\L)/  J^1(\L)$ et $J_P(\L)/  J^1_P(\L)$ sont  isomorphes \`a 
 $P (\L_\oE) / P^1(\L_\oE)$, lui-m\^eme isomorphe \`a $P (\L_\oE^{(0)}) / P^1(\L_\oE^{(0)}) \times \iota(\widetilde P  (L_\oE^{(1)}) / \widetilde P^1(\L_\oE^{(1)}))$ \cite[Lemma 5.3]{S5}. 
 Donnons-nous une repr\'esentation cus\-pi\-dale irr\'eductible $\rho$ de ces groupes, 
 que l'on peut \'ecrire sous la forme 
    $\rho^{(0)} \otimes \iota(\widetilde\rho ^{(1)})$ moyennant les identifications 
    ci-dessus, et  
  consid\'erons les repr\'esentations $\lambda_P$ de $J_P (\L)$ et $\lambda$ de $J (\L)$ d\'efinies par 
 $$
 \lambda_P = \kappa_P \otimes  \rho  , \quad   
 \lambda = \kappa \otimes  \rho .  
$$
D'apr\`es \cite[Lemma 6.1]{S5} -- dont la d\'emonstration n'utilise pas le fait que $\kappa$ soit une $\beta$-extension -- 
 elles sont irr\'eductibles, on a 
$\l = \cInd_{J_P }^{J } \l_P$  et l'isomorphisme d'alg\`ebres naturel de
$\mathcal H(G , \l_P)$ sur $\mathcal H(G , \l)$ envoie une fonction de support $J_P  y J_P $, $y \in G_E $, sur  une fonction de support $J  y J $.

  Fixons, pour $i= 2, \cdots, l$,  une $\oEi$-base   de la suite de r\'eseaux $\L^i$, 
  fixons   une $\oEun$-base   de la suite de r\'eseaux $\L^1\cap W^{(0)}$, et prenons la base $   \mathcal B^{(-1)} \cup  \mathcal B^{( 1)}$ de $W^{(-1)}\oplus W^{(1)}$. Ceci d\'efinit un tore d\'eploy\'e maximal de $G_E^+ $   de normalisateur $N^+$. Soit 
  $N_\L^+ $ l'intersection de $N^+$ avec le normalisateur dans $G_E^+$ de $P^0(\L_\oE )\cap M$. 
  On a 
  $N_\L^+= N_{\L^{(0)}}^+ \times N_{\L^{\ne 0}}^+ $ o\`u $N_{\L^{(0)}}^+$ et  $N_{\L^{\ne 0}}^+$  sont les objets analogues d\'efinis dans  $G_E^+ \cap \GL_F(W^{(0)}) = G_{E_1}^+\cap \GL_F(W^{(0)}\cap V^1) \times G_{E_2}^+
  \times \cdots G_{E_l}^+$ et $G_{E_1}^+ \cap \GL_F(W^{(-1)}\oplus W^{(1)}) $ 
  respectivement. 
   Soit $\tau$ une composante irr\'eductible de la restriction de $\rho$ \`a $P^0(\L_\oE)$ ;  
 on peut l'\'ecrire  comme plus haut  sous la forme   
  $\tau^{(0)} \otimes \iota(\widetilde\rho ^{(1)})$, de sorte que 
  $$
  \aligned 
  N_\L^+(\tau)  &:= \{n \in N_\L^+   \ / \ ^n\tau \simeq \tau\} 
  = N_{\L^{(0)}}^+(\tau^{(0)})  \times N_{\L^{\ne 0}}^+(\iota(\widetilde\rho ^{(1)})) 
  \\
  \text{et }  N_\L (\tau)  &:=  N_\L^+(\tau) \cap G = \left(N_{\L^{(0)}}^+(\tau^{(0)})  \times N_{\L^{\ne 0}}^+(\iota(\widetilde\rho ^{(1)})) \right) \cap G . 
  \endaligned 
  $$
  
  Rappelons que $s_1$ et $ s_1^\varpi$  appartiennent \`a $G_E$ 
 dans tous les cas except\'e celui du groupe sp\'ecial orthogonal avec $\beta_1 = 0$ 
 et $\dim_F W^{(1)} $ impaire  \cite[\S 6.2]{S5}. Pla\c cons-nous dans cette situation d'exception et supposons $V^1\cap W^{(0)} \ne \{0\}$. Le groupe $ G_{E_1}^+\cap \GL_F(W^{(0)}\cap V^1)$ est alors un groupe orthogonal et l'on peut choisir $p \in 
P^+(\L_\oE^{(0)}\cap V^1)$ de d\'eterminant $-1$   tel que $p^2=1$ 
 \cite[Case (i) p.350]{S5} ; si $V^1\cap W^{(0)}$ est de dimension impaire on choisit 
 $p=-1$. Alors  $p s_1$ et $ p s_1^\varpi$   v\'erifient toutes les propri\'et\'es de $ s_1$ et $  s_1^\varpi$ utilis\'ees pr\'ec\'edemment et appartiennent \`a $G_E$. Dans tous les autres cas posons $p=1$. 
  Soit enfin $W$ le groupe de Weyl affine \`a deux g\'en\'erateurs engendr\'e par $p s_1$ et $ p s_1^\varpi$  
  (par abus de langage - plus exactement c'est le quotient du groupe engendr\'e par $ps_1$, $ ps_1^\varpi$ et $\iota(\mathfrak o_{E_1}^\times)$ 
  par $\iota(\mathfrak o_{E_1}^\times)$). Alors  $N_{\L^{\ne 0}} / N_{\L^{\ne 0}} \cap P (\L_\oE^{\ne 0})  $ est    isomorphe \`a $W$ lorsque $ps_1$ appartient \`a $G_E$, c'est-\`a-dire
 dans tous les cas except\'e celui du  {\it groupe sp\'ecial orthogonal avec $\beta_1 = 0$,  
   $\dim_F W^{(1)} $ impaire et  $V^1\cap W^{(0)} = \{0\}$}.

	\begin{Proposition}[\bf  \cite{S5} \S 6.4]\label{support}
Posons $J_P  = J_P  (\L)$ et 	soit  $\kappa $  un prolongement  de $\eta$ \`a $J^+ (\L)$.  
\begin{enumerate}
	\item  Dans les trois cas suivants : 
	
\begin{enumerate}
	\item la repr\'esentation $\widetilde\lambda ^{(1)}$   n'est pas \'equivalente \`a $\widetilde\lambda ^{(1)}   \circ  \sigma_1$,
	\item $p$
	n'entrelace pas  $\tau^{(0)}$,
	\item $ps_1$ n'appartient pas \`a $G_E$, 
\end{enumerate}
 le support de l'alg\`ebre $  \mathcal H(G, \lambda_P)$ 
	est $J_P (I_{G^{(0)}}(\lambda^{(0)})  \times \iota(E_1^\times)) J_P$. 
\item 	
	Si $\widetilde\lambda ^{(1)}$ 
   est  \'equivalente \`a $\widetilde\lambda ^{(1)}\circ \sigma_1$ et 
   $p$
	 entrelace   $\tau^{(0)}$ 
   et $ps_1$  appartient   \`a $G_E$,   le support de  $  \mathcal H(G, \lambda_P)$
	est $J_P \, ( I_{G^{(0)}}(\lambda^{(0)})  \rtimes W) \,  J_P$.  
	\end{enumerate}
\end{Proposition} 
  
  \begin{proof} 
 Supposons tout d'abord que  $\kappa $  est une $\beta$-extension standard de $\eta$ ; en particulier $\tilde \kappa^{(1)}$ est   \'equivalente \`a $\tilde \kappa^{(1)}\circ \sigma_1$
 et la condition ``$\widetilde\lambda ^{(1)}$ 
   est  \'equivalente \`a $\widetilde\lambda ^{(1)}\circ \sigma_1$'' est v\'erifi\'ee si et seulement si $\widetilde\rho ^{(1)}$ 
   est  \'equivalente \`a $\widetilde\rho ^{(1)}\circ \sigma_1$ (variante de \cite[d\'emonstration de la Proposition 5.3.2]{BK}).  
  L'appli\-ca\-tion de \cite{S5}  (Proposition  6.15, Corollary 6.16 et   d\'emonstration de la Proposition  6.18) montre que l'entrelacement de $\l_P$ est contenu dans 
  $J_P N_\L(\tau) J_P$.  D'autre part 
  l'entrelacement de $\l_P$ dans $M$ est exactement l'entrelacement de 
  $\l^{(0)} \otimes \iota(\widetilde\l ^{(1)})$ \cite[Proposition 6.3 (ii)]{BK1}. 
  Si   $\widetilde\rho ^{(1)}$   n'est pas \'equivalente \`a $\widetilde\rho ^{(1)}   \circ  \sigma_1$ ou si $p s_1$ n'appartient pas \`a $G_E$, 
   $N_{\L^{\ne 0}}(\iota(\widetilde\rho ^{(1)}))$ se r\'eduit  \`a 
  $\iota(E_1^\times)$ qui entrelace effectivement la repr\'esentation, d'o\`u le r\'esultat. 
  Il en est de m\^eme si $p$
	n'entrelace pas  $\tau^{(0)}$, car $ps_1$ alors n'entrelace pas $\tau$. 
  Dans le second cas au contraire, $ps_1$  entrelace   $\tau$, donc $\rho$, donc $\lambda_P$.

 Un prolongement $\kappa$ arbitraire  de $\eta$ \`a $J^+ (\L)$ peut s'\'ecrire sous la forme $\kappa^s \otimes \chi$ o\`u   $\kappa^s $ est une 
$\b$-extension standard de $\eta$ et $\chi$ est un caract\`ere de $P^+(\L_\oE)/ P^1(\L_\oE)$  (par unicit\'e des op\'erateurs d'entrelacement de $\eta$ \`a scalaire pr\`es \cite[Proposition 3.5]{S5}). Ecrivons $\chi =  \chi^{(0)} \otimes \iota(\widetilde\chi^{(1)})$. Alors l'entrelacement de $\l_P = (\kappa^s \otimes \chi)_P\otimes \rho = \kappa^s_P \otimes \chi \rho$ est d\'etermin\'e par  l'action de $\sigma_1$ sur $\widetilde\chi^{(1)} \otimes \widetilde\rho ^{(1)}$. Comme $\widetilde {\kappa^s}^{(1)}$ est \'equivalente \`a $\widetilde {\kappa^s}^{(1)}\circ \sigma_1$, la premi\`ere partie donne le r\'esultat. 
Noter que si $\kappa$ est une $\b$-extension  le lemme 
\ref{equiv}  s'applique aussi bien \`a $\kappa$ qu'\`a $\kappa^s$, le caract\`ere $\widetilde\chi^{(1)}$ est donc $\sigma_1$-invariant (voir aussi \cite[Corollary 6.13]{S5}), de sorte que $\widetilde\rho ^{(1)}$ est 
$\sigma_1$-invariante si et seulement si $\widetilde\rho ^{(1)}\otimes \widetilde\chi^{(1)}$ l'est.  
\end{proof} 

\begin{Corollary}\label{cover} Soit $\kappa$
un prolongement  de $\eta$ \`a $J^+ (\L)$ et $\lambda = \kappa\otimes \rho$ comme ci-dessus. 
Notons ${J}^{(0)}=  J (\b^{(0)}, \L^{(0)} )$ et 
$ \widetilde {J}^{(1)}= \widetilde J (\b^{(1)}, \L^{(1)} )$. 
La paire $(J_P , \lambda_P)$ est une paire couvrante de 
$(  {J}^{(0)}\times \iota(  \widetilde {J}^{(1)} ), \  
\lambda^{(0)} \otimes \iota(\tilde \lambda^{(1)}))$ dans~$G$. Il en est de m\^eme de la paire 
analogue $(J_{P^-}, \lambda_{P^-})$. 
 
 	Dans le cas (i) de la proposition,  
  l'alg\`ebre $  \mathcal H(G, \lambda_P)$ est isomorphe \`a 
  $  \mathcal H(M, {\lambda_P}_{|J\cap M})$.  En particulier, la repr\'esentation induite parabolique d'une repr\'esentation irr\'eductible  de $M$ de type ${\lambda_P}_{|J\cap M}$ est toujours irr\'eductible.

   	Dans le cas (ii) de la proposition, si  $I_{G^{(0)}}(\lambda^{(0)}) = {J}^{(0)}$, 
l'alg\`ebre $  \mathcal H(G, \lambda_P)$ est une alg\`ebre de convolution sur $W$. Sa structure est d\'etermin\'ee par les relations quadratiques v\'erifi\'ees par  deux g\'en\'erateurs $T_1$ et $T_1^\varpi$ de supports respectifs $J_P p s_1 J_P$ et $J_P p s_1^\varpi J_P$. 
\end{Corollary}

\begin{proof}  
Dans le cas (i) de la proposition, 
l'entrelacement de  $\l_P  $ est form\'e de doubles classes d'\'el\'ements de $M$, 
la paire est donc  couvrante par \cite[Comments 8.2]{BK1} ; les cons\'equences sur l'irr\'eductibilit\'e sont prouv\'ees dans {\it loc. cit. } et rappel\'ees au paragraphe suivant. 

Dans l'autre cas, 
l'entrelacement de  $\l_P  $ est $J_P W J_P$. Gr\^ace au lemme \ref{s1pi}, qui reste valide pour $p s_1  $ et $  p s_1^\varpi $,  et aux r\'esultats de 
{\it loc. cit.}, il suffit   de montrer que les \'el\'ements de  $  \mathcal H(G, \lambda_P)$ de support 
  $J_P p s_1 J_P$ et $J_P  p s_1^\varpi J_P$ sont inversibles, ce que l'on peut faire comme dans 
  \cite{GKS} \S 2.2 (en v\'erifiant  que les d\'emonstrations des lemmes 2.10 et 
  2.13    s'appliquent mot pour mot) ou en passant directement aux \'enonc\'es plus pr\'ecis ci-apr\`es.  
  
Pour passer de $P$ \`a $P^-$ 
 il suffit de remarquer que  l'action de $\kappa |_{J_{P^-}} $ sur les $J^1(\L) \cap U^-$-invariants de $\eta$ est aussi \'egale  \`a $\kappa^{(0)} \otimes \tilde \kappa^{(1)}$
sur $J \cap M$.
C'est \'el\'ementaire en prenant le mod\`ele de 
$\eta \simeq \cInd_{J_P^1}^{J^1} \eta_P$ form\'e des fonctions de 
$J^1\cap U^- / H^1 \cap U^-$ dans l'espace de $\eta_P$ : 
les $J^1(\L) \cap U^-$-invariants sont les fonctions constantes.  
\end{proof}
 
 Pour obtenir des pr\'ecisions sur les relations quadratiques du   second cas ci-dessus  nous suivons encore \cite{S5}, \S 7.1 et 7.2.2. Commen\c cons par construire deux suites de $\oE$-r\'eseaux autoduales  $\mathfrak M_0= \mathfrak M_0^1 \perp \L^2 \cdots \perp \L^l$ et $\mathfrak M_1=\mathfrak M_1^1 \perp \L^2 \cdots \perp \L^l$ comme dans {\it loc. cit.} \S 7.2.2, \`a partir des 
   suites de $\mathfrak o_{E_1}$-r\'eseaux $\mathfrak M_0^1 $ et $\mathfrak M_1^1$ 
dans $V^1$,  de p\'eriode $2$ sur $E_1$,  caract\'eris\'ees par~: 
 $$\mathfrak M_0^1 (0)  = \L^1(1-q_1)  , \quad \mathfrak M_0^1 (1) = \L^1(q_1),\qquad 
  \mathfrak M_1^1(0)= \L^1( 1 -\frac{e_1}{2}), \quad  \mathfrak M_1^1(1)= \L^1(  \frac{e_1}{2}).$$ 
 Les ordres $\mathfrak b_0(\mathfrak M_0)$ et $\mathfrak b_0(\mathfrak M_1)$ sont des $\oE$-ordres autoduaux maximaux contenant $\mathfrak b_0(\L)$ et 
 la d\'ecomposition $ V = W^{(-1)} \oplus W^{(0)}\oplus W^{(1)}  $
  est  subordonn\'ee aux strates $[\mathfrak M_i ,n_{\mathfrak M_i} ,0,\b ]$ ($i=0,1$). 
  Enfin $s_1$ appartient \`a $P^+(\mathfrak  M_{0 ,\oE})$, $s_1^\varpi$ appartient \`a $P^+(\mathfrak  M_{1 ,\oE})$ et  l'on a :
\begin{equation}\label{plusplus}
P^+  (  \L_\oE ) = 
\left(P^+ (  \mathfrak M_{1 ,\oE} )\cap P^-
\right) P^1(  \mathfrak M_{1 ,\oE} ) = 
\left(P^+ (  \mathfrak M_{0 ,\oE} )\cap P 
\right) P^1(  \mathfrak M_{0 ,\oE} ) .
\end{equation} 
En particulier, la repr\'esentation     $\rho = \rho^{(0)} \otimes \iota(\widetilde\rho ^{(1)})$     de $J(\L)/J^1(\L) \simeq 
  P(   \L_\oE  )  / P^1(   \L_\oE   ) $  s'\'etend en une repr\'esentation  toujours not\'ee 
  $\rho$ 
  du sous-groupe parabolique $P (   \L_\oE )  / P^1(  \mathfrak M_{1 ,\oE} ) $ de $P  (  \mathfrak M_{1 ,\oE} )$ tri\-viale sur son radical unipotent $P^1 (   \L_\oE )  / P^1(  \mathfrak M_{1 ,\oE} ) $, et de m\^eme  avec $ \mathfrak M_{0 ,\oE}$. Avec ces conventions on a~: 
  \begin{Proposition}[\bf  \cite{S5}(7.3)]\label{injection1}
	Si $\kappa $ est une $\beta$-extension de $\eta$ relative \`a $ \mathfrak M_{0  }$,  
  il y a un homomorphisme d'alg\`ebres injectif 
	$$
	j_{ \, 0} : \   \mathcal H(P  (  \mathfrak M_{0 ,\oE} ), \rho  ) \ 
	\hookrightarrow \ 
	  \mathcal H(G, \lambda_P) 
$$
	pr\'eservant le support :   $\text{ \rm Supp}(j_0(\phi))
	= J_P \text{ \rm Supp} (\phi) J_P$. 
	Il en est de m\^eme en rempla\c cant $ \mathfrak M_{0  }$ par $ \mathfrak M_{1  }$, 
	avec un homomorphisme $j_1$.  
\end{Proposition}

\begin{proof} Il  faut   justifier l'usage   de \cite[\S 7.1]{S5},  qui s'occupe de repr\'esentations de groupes de la forme $J^0(\L)$, pour nos repr\'esentations de groupes $J(\L)$. 
Par hypoth\`ese, il exis\-te une  
 $\beta$-extension $\kappa(\mathfrak M_{0  })$, repr\'esentation de   $J^+(\mathfrak M_{0  })$, 
dont la restriction   \`a 
$P^+(\L_\oE)J^1(\mathfrak M_{0  })$ correspond  \`a la repr\'esentation    $\kappa$ de  $J^+(\L)$ par la bijection canonique $\mathfrak B$. 
Reprenant la d\'emonstration de \cite{S5} Proposition 7.1 pas \`a pas, on obtient un isomorphisme d'alg\`ebres de Hecke, pr\'eservant le support, de 
$ \mathcal H(G, \kappa(\mathfrak M_{0  })|_{P(\L_\oE)J^1(\mathfrak M_{0  })}\otimes \rho   )$ sur 
$ \mathcal H(G, \kappa|_{J(\L)}\otimes \rho  )$. 
L'isomorphisme de la proposition 7.2 de {\it loc. cit.} 
reste valide puisque $\kappa(\mathfrak M_{0  })|_{P(\L_\oE)J^1(\mathfrak M_{0  })}$ se prolonge \`a $J(\mathfrak M_{0  })$. Le dernier isomorphisme de la composition de \cite[(7.3)]{S5} provient   de \cite[Lemma 6.1]{S5}. 
  \end{proof}

\begin{Corollary}\label{dernier} 
Soit $\kappa$
un prolongement  de $\eta$ \`a $J^+ (\L)$ et $\lambda = \kappa\otimes \rho$ comme ci-dessus.
Supposons  que la repr\'esentation $\widetilde\lambda ^{(1)}$ est   \'equivalente \`a $\widetilde\lambda ^{(1)}\circ \sigma_1$, que    $p$
	 entrelace   $\tau^{(0)}$ 
   et que $ps_1$  appartient   \`a $G_E$. 
Soient $\chi_0$ et $\chi_1$ des caract\`eres de $P^+(\L_\oE)/P^1(\L_\oE)$ tels que 
$\kappa \otimes \chi_0^{-1}$ soit une $\beta$-extension   relative \`a $ \mathfrak M_{0  }$ 
et $\kappa \otimes \chi_1^{-1}$ une $\beta$-extension   relative \`a $ \mathfrak M_{1 }$. 
Moyennant une normalisation convenable, 
 les relations quadratiques v\'erifi\'ees par  les  g\'en\'erateurs $T_1$ et $T_1^\varpi$ de $  \mathcal H(G, \lambda_P)$ sont respectivement les relations quadratiques 
 v\'erifi\'ees par   
\begin{itemize}
	\item 
 le g\'en\'erateur de 
  $\mathcal H(P  (  \mathfrak M_{0 ,\oE} ), \rho  \otimes \chi_0 )$, de support   $P(\L_\oE ) p s_1 P(\L_\oE )$, 
 \item le g\'en\'erateur de 
  $\mathcal H(P  (  \mathfrak M_{1 ,\oE} ),   \rho \otimes \chi_1 )$, de support  $P(\L_\oE ) p s_1^\varpi P(\L_\oE )$.
  \end{itemize} 
\end{Corollary}  
 
 Ce 
  corollaire   fournit le moyen de calculer les relations quadratiques cherch\'ees, \`a homoth\'etie pr\`es, dans un groupe r\'eductif fini. Nous indiquons dans le paragraphe suivant les conclusions qu'une telle connaissance permet de tirer en termes de r\'eductibilit\'e d'induites.

\subsection{Relations quadratiques et points de r\'eductibilit\'e}\label{rqpr}
Nous expliquons dans cette section comment 
les coefficients des relations quadratiques sa\-tis\-faites par les deux g\'en\'erateurs de l'alg\`ebre de Hecke ci-dessus, sous les hypoth\`eses communes aux corollaires \ref{cover} et  \ref{dernier}, d\'eterminent enti\`erement les parties  r\'eelles des quatre points de r\'eductibilit\'e de l'induite correspondante. 

Commen\c cons par rappeler la propri\'et\'e caract\'eristique  des paires couvrantes dans le groupe classique $G$, pour des repr\'esentations supercuspidales  d'un sous-groupe de Levi maxi\-mal $M \simeq \GL(n,F)\times  L$. On se donne une repr\'esentation supercuspidale $\pi$ de 
$\GL(n,F)$, un type  $(J_n, \lambda_n)$ pour $\pi$, une repr\'esentation supercuspidale $\sigma$ de 
$L$, un type  $(J_0, \lambda_0)$ pour $\sigma$, un sous-groupe parabolique $P$ de $G$ de facteur de Levi $M$, et on suppose construite une paire couvrante $(J, \lambda)$ de $(J_n \times J_0, \lambda_n \otimes \lambda_0)$ dans $G$. 
Par d\'efinition \cite[\S7]{BK1} 
il existe un morphisme d'alg\`ebres injectif $t : \mathcal H(M,  \lambda_n \otimes \lambda_0) \longrightarrow \mathcal H(G, \lambda ) $ rendant le diagramme suivant commutatif~: 
\begin{equation}\label{diagBK}
\begin{array}{ccc} \mathcal R^{[\pi\otimes\sigma, M]} (G)  &  \stackrel{   }{\longrightarrow}  & \text{Mod-}\mathcal H(G, \lambda ) 
\\
 \ind_P^G \uparrow   &   & t_\ast  \uparrow  
\\
\mathcal R^{[\pi\otimes\sigma, M]}  (M)   &   \stackrel{\mathfrak h }{\longrightarrow} & \text{Mod-}\mathcal H(M,  \lambda_n \otimes \lambda_0) 
\end{array}
\end{equation}
Nous travaillons ici avec l'induite normalis\'ee $   \ind_P^G $ et avec des modules \`a droite 
sur les alg\`ebres de Hecke. Le foncteur $t_\ast$ associe \`a un module $X$ le module 
$\Hom_{\mathcal H(M, \,  \lambda_n \otimes \lambda_0) }(\mathcal H(G, \lambda ) , X)$, dans lequel $t$ d\'etermine la structure de module de $\mathcal H(G, \lambda )$, et 
le foncteur $\mathfrak h $ associe \`a une repr\'esentation $\tau$ le module 
$\Hom_{J_n \times J_0}( \lambda_n \otimes \lambda_0, \tau)$. 

Les fl\`eches horizontales de (\ref{diagBK}) sont des \'equivalences de cat\'egories, ce diagramme implique donc  
que pour tout $s \in \mathbb C$, la repr\'esentation $ \  \ind_P^G \, \pi |\det\!|^s \otimes \sigma $ 
est r\'eductible si et seulement si le $\mathcal H(G, \lambda )$-module 
$  \  t_\ast  \, \mathfrak h \, (\pi |\det\!|^s\otimes \sigma) \  $ est r\'eductible.

Choisissons pour $\sigma$ un type  $(J_0, \lambda_0)$ construit par Stevens \cite[Theorem 7.14]{S5}  de sorte que $\sigma \simeq \Ind_{J_0}^{L}\lambda_0$. Alors 
 l'alg\`ebre 
$\mathcal H(M,  \lambda_n \otimes \lambda_0)$ est isomorphe \`a 
$\mathcal H(\GL(n,F),  \lambda_n  ) $.

 Choisissons pour $\pi$ 
un type simple maximal de  Bushnell-Kutzko  $(J_n, \lambda_n)$,   
attach\'e \`a une strate simple  
  $[\L_n , t_n, 0,  \beta_n]  $. Posons 
  $E_n = F[\beta_n] $ et notons $\varpi_{E_n}$ une uniformisante de 
  $E_n$. D'apr\`es  \cite[\S6]{BK}, l'entrelacement de   
   $\lambda_n$ est \'egal \`a    $ \  \widehat{J_n}= \Ex_n \,  J_n = \varpi_{E_n}^\mathbb Z \,  J_n \, $ et on a une bijection 
   $\Lambda_\tau \   \longmapsto \  \tau = \Ind_{\widehat J_n}^{\GL(n,F)} \, \Lambda_\tau$ 
   entre l'ensemble des prolongements de $\lambda_n$ \`a $\widehat J_n$ et l'ensemble des 
   classes d'\'equivalence de  repr\'esentations irr\'eductibles de la classe d'inertie de $\pi$.   Soit $Z$ un \'el\'ement de $\mathcal H(\GL(n,F),  \lambda_n  ) $ de support $ \varpi_{E_n}   J_n$ ; on a par \cite[\S 5.5]{BK} : 
 $$\mathcal H(\GL(n,F),  \lambda_n  )  \simeq 
   \mathbb C[Z, Z^{-1}] .  $$ 
   En particulier, les repr\'esentations irr\'eductibles de cette alg\`ebre sont des caract\`eres,   d\'etermi\-n\'es par leur valeur en $Z$. 
 Le groupe $X^0$ des caract\`eres non ramifi\'es de $\GL(n,F)$ agit sur 
   $\mathcal H(\GL(n,F),  \lambda_n  )$  par  
 $\chi(f)(x)=\chi(x) f(x)$ ($\chi \in X^0$, $f  \in \mathcal H(\GL(n,F),  \lambda_n  )$, $x \in  \GL(n,F)$), et si $\tau$ est une repr\'esentation irr\'eductible comme ci-dessus on a : 
   $ \  \mathfrak h(\tau \otimes \chi^{-1}) = \mathfrak h( \tau )\circ \chi  \ $  
soit :  
\begin{equation}\label{caract}
 \mathfrak h(\tau \otimes \chi^{-1})(Z) = \chi(\varpi_{E_n} ) \mathfrak h(\tau )(Z) .
 \end{equation}

Le plongement de $\GL(n,F)$ dans $M \subset G$ d\'etermine une involution $g \mapsto \hat g^{-1}$ 
sur $\GL(n,F)$ d\'eduite de l'involution adjointe (voir le d\'ebut du paragraphe pr\'ec\'edent) ;
on dira que $\pi$ est {\it autoduale }si sa classe d'isomorphisme est fix\'ee par cette involution. (Cela signifie que $\pi$ est \'equivalente \`a sa contragr\'ediente dans les cas symplectique et orthogonal, \`a la contragr\'ediente de sa compos\'ee avec la conjugaison de $F$ sur $F_0$
 dans le cas unitaire.) 
Si aucune des repr\'esentations $\pi |\det\!|^s$ n'est autoduale on sait que la repr\'esentation induite est toujours irr\'eductible (\cite[Corollary 1.8]{S} ; voir aussi \cite[Lemmas 4.1 and 5.1]{Z}).

Supposons $\pi$ autoduale et 
faisons les hypoth\`eses suivantes   :  
\begin{enumerate} 
\item L'alg\`ebre $\mathcal H(G, \lambda )$ a deux g\'en\'erateurs 
$T_0$ et $T_1$ v\'erifiant des relations quadratiques de la forme 
 $ \  
(T_i - \omega_i^1)(T_i - \omega_i^2)=0 
\quad (i=0,1) $ et on a 
$ t(Z) = T_0 T_1$. 
\item  Les quotients des valeurs propres sont des nombres r\'eels strictement n\'egatifs. On ordonne alors les valeurs propres de fa\c con que  : 
$ \   \omega_i^1 / \omega_i^2 = -q^{r_i} $ avec $r_i \ge 0$.     
\end{enumerate}
Soit   $\xi$   un  caract\`ere  de
  $\mathcal H(M, \lambda_n \otimes \lambda_0)$. La repr\'esentation $t_\ast(\xi) $ est de dimension $2$, elle  est r\'eductible si et seulement si elle contient un caract\`ere, 
  c'est-\`a-dire si et seulement si 
   il existe un caract\`ere 
$\alpha$ de 
 $\mathcal H(G, \lambda)$ tel que  
 $\xi = \alpha \circ t$ \cite[Proposition 1.13]{BB}.  Ainsi la r\'eductibilit\'e de  $ \  \ind_P^G \, \pi |\det\!|^s \otimes \sigma $ est obtenue exactement pour quatre valeurs de $s$ (compt\'ees avec multiplicit\'e), donn\'ees par 
 $ \mathfrak h \, (\pi |\det\!|^s) (Z) = \omega_0^j \omega_1^k,$ $ j, k = 1,2 $, 
 c'est-\`a-dire  
\begin{equation}\label{4car}
 \mathfrak h \, (\pi |\det\!|^s) (Z) \ \in \  \{ \alpha_0 = \omega_0^2 \omega_1^2, \ \  -q^{r_0}  \alpha_0 , \  \  -q^{r_1}  \alpha_0, \  \    q^{r_0+r_1}  \alpha_0\} .
 \end{equation}

On sait par ailleurs que $\mathcal R^{[\pi , \GL(n,F)]}  (\GL(n,F))$ contient exactement 
deux repr\'esentations irr\'eductibles autoduales non \'equivalentes :   $ \pi \  $ et $ \  \pi  \  |\det|^{\frac{e}{2n} \frac{2i\pi}{\log q}}$, 
pour un unique entier $e$ divisant $n$, qui est l'indice de ramification de $E_n$ 
sur $F$  \cite[Lemma 6.2.5]{BK}. 
  On sait enfin \cite[Theorem 1.6]{S} que
    pour chaque repr\'esentation supercuspidale autoduale $\pi'$ de 
$\GL(n,F)$  l'ensemble des points de r\'eductibilit\'e {\it r\'eels}  de $ \  \ind_P^G \, \pi' |\det\!|^s \otimes \sigma $  
est soit vide, soit de la forme $\{a, -a\}$ pour un r\'eel positif ou nul $a$. 
Ainsi il existe au maximum quatre repr\'esentations donnant lieu \`a r\'eductibilit\'e 
 (compt\'ees avec multiplicit\'e), \`a savoir : 
$  \pi \circ |\det|^{\pm s_1}$ et  
  $\pi \circ |\det|^{\pm s_2 + \frac{e}{2n} \frac {2i \pi}{\log q} } $, 
 pour un $s_1 \in \mathbb R$ et un  
  $s_2 \in \mathbb R $,  
 dont les images par $\mathfrak h$ sont, via (\ref{caract}), les caract\`eres ayant pour valeur en $Z$ un \'el\'ement de l'ensemble : 
 \begin{equation}\label{4rep} 
  \{  q^{s_1 \frac n e } \  \mathfrak h \, (\pi  ) (Z), q^{-s_1 \frac n e } \ \mathfrak h \, (\pi  ) (Z), - q^{s_2 \frac n e } \  \mathfrak h \, (\pi  ) (Z), - q^{s_2 \frac n e } \ \mathfrak h \, (\pi  ) (Z)  \} . 
 \end{equation}

 Sous les hypoth\`eses (i) et (ii), un sous-ensemble $X$ convenable de (\ref{4rep}) doit co\"\i ncider avec (\ref{4car}). Comme (\ref{4car})  contient des paires d'\'el\'ements de quotient n\'egatif, le sous-ensemble $X$ doit avoir la m\^eme propri\'et\'e : c'est   (\ref{4rep}) tout entier. Les hypoth\`eses faites entra\^\i nent donc que les induites consid\'er\'ees ont effectivement des points de r\'eductibilit\'e. En comparant les quotients r\'eels positifs de valeurs de   (\ref{4car}) et  (\ref{4rep}) on obtient la formule suivante : 
 
\begin{Proposition}\label{formule} 
Supposons que l'alg\`ebre $\mathcal H(G, \lambda )$ a deux g\'en\'erateurs 
 v\'erifiant des relations quadratiques de la forme 
 $ \  
(T_i - \omega_i^1)(T_i - \omega_i^2)=0 
\  (i=0,1) $ dont les 
  quotients des valeurs propres s'\'ecrivent 
$ \   \omega_i^1 / \omega_i^2 = -q^{r_i} $ avec $r_i \ge 0$, et  tels que $ t(Z) = T_0 T_1$.  
Les param\`etres $r_0$, $r_1$  d\'eterminent alors les parties r\'eelles des points de r\'eductibilit\'e de $ \  \ind_P^G \, \pi |\det\!|^s \otimes \sigma $ : ce sont les \'el\'ements de l'ensemble 
 $$ \{ \  \pm \frac{e}{2n} \, (r_0+r_1),  \ \pm \frac{e}{2n} \, (r_0-r_1) \  \} 
.$$ 
\end{Proposition}
 
 \begin{Remark}\label{piexacte}
 {\rm 
Cette proposition ne permet pas de d\'eterminer s\'epar\'ement les deux points de r\'educ\-ti\-bi\-lit\'e r\'eels 
correspondant \`a chacune des  deux repr\'esentations autoduales  $ \pi \  $ et $ \  \pi  \  |\det|^{\frac{e}{2n} \frac{2i\pi}{\log q}}$. Elle ne fournit que l'ensemble total des parties r\'eelles des points de r\'educ\-ti\-bi\-lit\'e pour une repr\'esen\-tation autoduale de la classe d'inertie de $\pi$.}
 \end{Remark} 
 
 \begin{Remark} 
 {\rm Lorsque la structure de $\mathcal H(G, \lambda )$ est donn\'ee par un \'enonc\'e semblable au corollaire  \ref{dernier} ci-dessus, l'hypoth\`ese (i) est automatiquement 
 v\'erifi\'ee. L'hypoth\`ese (ii) l'est  aussi par \cite[\S 6.7]{M}, apr\`es 
 \cite{HL} :  le quotient des valeurs propres est $-p^{\, c}$, i.e. l'oppos\'e du quotient des degr\'es des deux repr\'esentations irr\'eductibles formant l'induite de la cuspidale dans le groupe r\'eductif fini. }
 \end{Remark} 
 
 \begin{Remark} 
 {\rm 
Pour des repr\'esentations $\sigma$ et $\pi$ dont les types v\'erifient les hypoth\`eses du paragraphe \ref{pc} (voir \`a ce sujet la remarque \ref{sigmapi} plus loin) on retrouve  les r\'esultats de r\'eductibilit\'e connus par ailleurs. En effet, si $G$ est un groupe symplectique ou unitaire   les \'enonc\'es de ce paragraphe  se simplifient (car $p=  1$ et $  s_1$ appartient \`a $G_E$) et, compl\'et\'es par la proposition \ref{formule}, fournissent l'irr\'eductibilit\'e de  $ \  \ind_P^G \, \pi |\det\!|^s \otimes \sigma $ si aucune des repr\'esentations $\pi |\det\!|^s $ n'est autoduale,  
 et sa r\'eductibilit\'e pour un unique $s \ge 0$ si $\pi$ est autoduale. 
 Il en est de  m\^eme si $G$ est 
 sp\'ecial orthogonal en dimension impaire : le seul cas o\`u   $  s_1$ n'appartient pas \`a $G_E$ est celui  o\`u $\b_1=0$, les espaces $V^2, \cdots, V^l$ sont  alors de dimension paire,  car chacun porte  une strate simple gauche non nulle, donc 
 $W^{(0)} \cap V^1$ est de dimension impaire d'o\`u  $p=-1$ et $ps_1$ appartient \`a $G_E$.

 Le cas du groupe sp\'ecial orthogonal en dimension paire est plus d\'elicat.  L'irr\'eductibilit\'e de  $ \  \ind_P^G \, \pi |\det\!|^s \otimes \sigma $ si aucune des repr\'esentations $\pi |\det\!|^s $ n'est autoduale est donn\'ee par le corollaire \ref{cover}. 
 Via les lemmes de Zhang cit\'es plus haut, on est assur\'es de l'existence de points de r\'eductibilit\'e r\'eels, pour $\pi$ autoduale, dans tous les cas sauf le suivant : 
 $G$ est un groupe sp\'ecial orthogonal et $n$ est impair.  De fait, si $n$ est pair, $s_1$ appartient \`a $G_E$, $p=1$ et l'on retrouve les situations d\'ecrites plus haut. 
Supposons $n$ impair ;   l'autodualit\'e de $\pi$ entra\^\i ne $n=1$ par 
 \cite[Theorem 1]{B}. Le cas d'exception est donc celui de $\GL_1(F) \times \SO(2t, F)$, Levi de $  \SO(2t+2, F)  $. D'apr\`es la proposition \ref{support} il existera  des points de r\'eductibilit\'e si et seulement si $p$
	 entrelace   $\tau^{(0)}$ 
   et $ps_1$  appartient   \`a $G_E$ -- en particulier il n'en existe pas si 
    $W^{(0)} \cap V^1 = \{ 0 \}$. On pourrait comparer cette condition avec celle de 
    Jantzen \cite{J} selon laquelle l'induite a des points de r\'eductibilit\'e si et seulement si $\sigma$ est \'equivalente \`a sa conjugu\'ee   par un \'el\'ement du groupe orthogonal de d\'eterminant $-1$ (si  $W^{(0)} \cap V^1 = \{ 0 \}$ il est facile de voir 
que     cette condition n'est jamais satisfaite). Ce n'est pas notre propos ici. 

 Ces questions de r\'eductibilit\'e ont \'et\'e beaucoup \'etudi\'ees, en premier lieu par Shahidi 
 \cite[Theorem 8.1]{Sh} ;     pour le cas d'exception citons par exemple \cite{Mg}, \cite{Z}, \cite{J}...
 }
 \end{Remark}

\subsection{Conclusion}\label{conclusion}  

Passons enfin \`a l'\'etude de la r\'eductibilit\'e  sous les hypoth\`eses communes aux   paragraphes  \ref{pc} et \ref{2.4}, c'est-\`a-dire   
une d\'ecomposition autoduale 
  $V = W^{(-1)} \oplus W^{(0)}\oplus W^{(1)} $ exactement subordonn\'ee \`a $[\L ,n ,0,\b ]$, avec  $W^{(0)} =  V^2 \perp \cdots \perp V^{l}$ et 
$ W^{(-1)}  \oplus W^{(1)}   = V^1$. On pose 
$Z_1 = W^{(-1)}  \oplus  W^{(1)}$, $Z_2 = W^{(0)}$,  
comme au paragraphe \ref{2.4}, et on note $G_1^+$ (resp. $G_{(0)}^+$) le groupe d'automorphismes de la restriction de la forme $h$ \`a $Z_1$ (resp. $Z_2$), $G_1$ (resp. $G_{(0)}$) sa composante neutre, 
de sorte que $G_\interieur^+ = G_{(0)}^+ \times G_1^+$.

Soient $P $ et $M$ comme  en \ref{pc} et $P_1$, $M_1$ leurs intersections avec $G_1$. 
Le corollaire \ref{cover} nous d\'ecrit une paire couvrante $(J_P , \lambda_P)$, dans $G$, de la paire  
$(  {J}^{(0)}\times \iota(  \widetilde {J}^{(1)} ), \  
\lambda^{(0)} \otimes \iota(\tilde \lambda^{(1)}))$ dans~$M = G_{(0)} \times \iota(\GL_F(W^{(1)}))$.
Supposons que  $I_{G^{(0)}}(\lambda^{(0)}) = {J}^{(0)}$ ;      la paire  $(  {J}^{(0)} , \  
\lambda^{(0)}  )$ est alors  un type pour la classe d'inertie de la repr\'esentation supercuspidale 
$\sigma = \cInd_{{J}^{(0)}}^{G_{(0)}} \lambda^{(0)} $ de $G_{(0)}$ (\cite[Corollary 6.19]{S5} ;   
$\kappa^{(0)}$ est tordue d'une $\b$-extension par un caract\`ere, comme dans la d\'emonstration de la proposition  \ref{support}). 
D'autre part   $\mathfrak b_0(\L^{(1)})$ est un ${\mathfrak o}_{E_1}$-ordre maximal de 
$\End_{E_1}(W^{(1)}) $ donc
$(    \widetilde {J}^{(1)}  , \  
 \tilde \lambda^{(1)} )$ est un type pour une classe d'inertie de   repr\'esentations supercuspidales de $\GL_F(W^{(1)})$ \cite[Theorem 6.2.2]{BK} dont on note $\pi$ un \'el\'ement.   Le diagramme   \ref{diagBK} relie alors l'\'etude de r\'eductibilit\'e des induites  
$ \  \Ind_P^G \, \pi |\det\!|^s \otimes \sigma $ \`a la structure de $\mathcal H(G, \lambda_P )$.

Notre but est de    comparer les parties r\'eelles des points de   r\'eductibilit\'e des induites paraboliques      $ \  \Ind_P^G \, \pi |\det\!|^s \otimes \sigma $ dans $G$  
aux parties r\'eelles des points de   r\'eductibilit\'e  des   induites paraboliques      $ \  \Ind_{P_1}^{G_1} \, \pi |\det\!|^s  $ 
dans $G_1$. Dans $G_1$ nous avons une situation tout \`a fait analogue 
c'est-\`a-dire une paire couvrante $(J_{P_1}(\L^1), \l_{P_1})$ de  la paire $(    \iota(  \widetilde {J}^{(1)} ), \  
  \iota(\tilde \lambda^{(1)}))$ dans $ M_1$ (ce qui d\'efinit $\l_{P_1}$).

 D'apr\`es le corollaire \ref{cover}, si la repr\'esentation $\widetilde\lambda^{(1)} $ 
    n'est pas \'equivalente \`a $\widetilde\lambda ^{(1)} \circ \sigma_1$ 
    ou si $s_1$ n'appartient pas \`a $G_E$,  les alg\`ebres 
  $\mathcal H(G, \lambda_P )$ et $\mathcal H(G_1, \l_{P_1} )$ sont isomorphes  
     \`a $  \mathcal H(M, {\lambda_P}_{|J\cap M})$ et $  \mathcal H(M_1, {\lambda_{P_1}}_{|J_{P_1}(\L^1) \cap M_1})$ respectivement et les induites $ \  \Ind_P^G \, \pi |\det\!|^s \otimes \sigma $ et    $ \  \Ind_{P_1}^{G_1} \, \pi |\det\!|^s  $  sont toujours irr\'eductibles. Le premier cas est celui  o\`u la classe d'inertie de $\pi$ ne contient pas de repr\'esentation auto\-duale et l'on retrouve le fait qu'une repr\'esentation supercuspidale autoduale poss\`ede un type autodual. 
     
   Pla\c cons-nous  d\'esormais dans le cas autodual o\`u  $\widetilde\lambda^{(1)} $ 
     est  \'equivalente \`a $\widetilde\lambda ^{(1)} \circ \sigma_1$, et supposons que 
     $s_1$  appartient  \`a $G_E$.  
L'\'enonc\'e du corollaire \ref{dernier}, dont on reprend les notations,  
  peut \^etre pr\'ecis\'e en ne faisant intervenir que les intersections avec $G_1$. En effet 
  $P  (  \mathfrak M_{0 ,\oE} )$ est \'egal au produit 
  $P  (  \mathfrak M_{0 ,\oE}^{(0)} ) \times P  (  \mathfrak M_{0 ,\oE}^1 )$, de m\^eme pour 
  $P  (  \mathfrak M_{1 ,\oE} )$ et  $P  (  \L\oE  )$. Comme $\mathfrak b_0(\L^{(0)})$ est un $\oE$-ordre autodual maximal de 
$\End_F(W^{(0)})\cap B$ on a $P  (  \mathfrak M_{0 ,\oE}^{(0)} )=  P  (  \mathfrak M_{1 ,\oE}^{(0)} )= P  (  \L_\oE^{(0)} )$, de sorte que (cf. (\ref{plusplus})) les alg\`ebres :
   $$
   \aligned 
   &\mathcal H(P  (  \mathfrak M_{0 ,\oE} ), \rho  \otimes \chi_0 ) \  \simeq \ 
   \mathcal H(P  (  \mathfrak M_{0 ,\oE}^1 ), \iota(\rho^{(1)}  \otimes \chi_0^{(1)}) )
  \\  \text{ et } \ 
   &\mathcal H(P  (  \mathfrak M_{1 ,\oE} ), \rho  \otimes \chi_1 ) \  \simeq \  
   \mathcal H(P  (  \mathfrak M_{1 ,\oE}^1 ), \iota(\rho^{(1)}  \otimes \chi_1^{(1)}) )  
   \endaligned 
   $$
  d\'eterminent les relations quadratiques satisfaites par les g\'en\'erateurs de $\mathcal H(G, \lambda_P )$. 
De la m\^eme fa\c con, les relations quadratiques satisfaites par les g\'en\'erateurs de $\mathcal H(G_1, \l_{P_1} )$ sont d\'etermin\'ees par les alg\`ebres 
$$  
   \mathcal H(P  (  \mathfrak M_{0 ,\oE}^1 ), \iota(\rho^{(1)}  \otimes \chi_{1,0}^{(1)}) ) \ 
  \text{ et } \  
   \mathcal H(P  (  \mathfrak M_{1 ,\oE}^1 ), \iota(\rho^{(1)}  \otimes \chi_{1,1}^{(1)}) )  
   $$
o\`u  $\chi_{1,0}$ et $\chi_{1,1}$ sont des caract\`eres de $P^+(\L_\oE^1)/P^1(\L_\oE^1)$ tels que 
$\kappa_1 \otimes \chi_{1,0}^{-1}$ soit une $\beta$-extension   relative \`a $ \mathfrak M_{0  }^1$ 
et $\kappa_1 \otimes \chi_{1,1}^{-1}$ une $\beta$-extension   relative \`a $ \mathfrak M_{1 }^1$.

Dans les deux cas il s'agit d'alg\`ebres $ \mathcal H(\mathcal G, \gamma )$ relatives \`a une repr\'esentation supercuspidale 
auto\-duale $\gamma$ du sous-groupe de Levi de Siegel d'un groupe r\'eductif fini $\mathcal G$, soit
$$P (\L_\oE^1)/P^1(\L_\oE^1)\simeq \widetilde P (\L_\oE^{(1)})/\widetilde P^1(\L_\oE^{(1)})
\simeq \GL(f,k_{E_1})
$$
o\`u $  k_{E_1}$ est le corps r\'esiduel de $E_1$ et $f $ la dimension de $W^{(1)}$ sur $E_1$.  
(On rappelle comme au paragraphe \ref{rqpr} que la notion d'autodualit\'e correspond ici \`a 
$\gamma \circ \sigma_1 \simeq\gamma$, et que l'action de $\sigma_1$ sur $\GL(f,k_{E_1})$ 
est \'equivalente \`a $g \mapsto { }^t g^{-1}$ si $\beta_1=0$ ou si $E_1 $ est   ramifi\'ee sur $E_1^0$,  \`a $g \mapsto { }^t \bar g^{-1}$ si  $E_1 $ est  non ramifi\'ee sur $E_1^0$, la barre d\'esignant alors l'action de l'\'el\'ement non trivial du groupe de Galois de 
$k_{E_1}$ sur $k_{E_1^0}$.) 
Les relations quadratiques correspondantes peuvent \^etre calcul\'ees au cas par cas \`a partir des travaux de Lusztig. Quoi qu'il en soit, il arrive que ces relations soient ind\'ependantes 
du choix de la repr\'esentation cuspidale autoduale $\gamma$. Lorsque c'est le cas, la proposition  \ref{formule} implique que les parties r\'eelles des points de r\'eductibilit\'e de $ \  \Ind_P^G \, \pi |\det\!|^s \otimes \sigma $ et       $ \  \Ind_{P_1}^{G_1} \, \pi |\det\!|^s  $ sont les m\^emes. 

\begin{Example}\label{exemple} {\rm 
  Si $E_1 $ est  non ramifi\'ee sur $E_1^0$ et de degr\'e maximal 
$[E_1:F]= \dim_F W^{(1)}$, le groupe $ \mathcal G$ est isomorphe dans les deux cas \`a $U(1,1)(k_{E_1}/k_{E_1^0})$. Les relations quadratiques sont toujours homoth\'etiques \`a 
$X^2 = (q_{E_1} -1 ) X + q_{E_1}$, avec quotient des valeurs propres $-q_{E_1}$. 
L'application de la proposition 
\ref{formule} donne pour parties r\'eelles des points de r\'eductibilit\'e $0$ (double) et $\pm \frac 1 2 $. 
 Si $E_1 $ est  ramifi\'ee sur $E_1^0$ et de degr\'e maximal, 
 le groupe $ \mathcal G$ est isomorphe  \`a $\SL(2,k_{E_1 }))$ ou  $O(1,1)(k_{E_1 }))$. Dans $\SL(2,k_{E_1 }))$ la relation quadratique est soit $X^2 = (q_{E_1} -1 ) X + q_{E_1}$, soit $X^2=1$, avec quotient des valeurs propres $-q_{E_1}$ ou $-1$, elle est sensible \`a la torsion ; dans 
$O(1,1)(k_{E_1 }))$ c'est $X^2=1$. Les parties r\'eelles des points de r\'eductibilit\'e sont soit $0$ (double) et $\pm \frac 1 2 $, soit $0$ (quadruple). 

Dans ces deux cas, pour toute repr\'esentation $\sigma$ v\'erifiant les hypoth\`eses du pr\'esent paragraphe (voir la remarque \ref{sigmapi}), les points de r\'eductibilit\'e de $ \  \Ind_P^G \, \pi |\det\!|^s \otimes \sigma $   sont   de partie  r\'eelle  strictement inf\'erieure  \`a $1$. 
}
\end{Example} 

\bigskip 

Si les relations {\it ne sont pas  ind\'ependantes } du choix de $\gamma$, la comparaison cherch\'ee impose de comparer les caract\`eres $\chi_0$ et $\chi_1$ d'une part, $\chi_{1,0}$ et $\chi_{1,1}$ d'autre part ; c'est l\`a qu'intervient le th\'eor\`eme \ref{theoreme}.
Soit  donc  $\mathbf w_U^0$ (resp.  $\mathbf w_{U^-}^1$) le caract\`ere de $P^+(\L_\oE )$, trivial sur $P^1(\L_\oE )$,     
  qui \`a $x \in P^+(\L_\oE  )\cap M $ 
associe la signature de la permutation  $v \mapsto xvx^{-1}$ de $ \ J^1_\ext(\mathfrak M_0) \cap U^- / \ H^1(\mathfrak M_0) \cap U^-$ 
(resp. $J^1_\ext(\mathfrak M_1) \cap U  / H^1(\mathfrak M_1) \cap U  $). Rappelons que ces caract\`eres ne d\'ependent que de la strate $[\L ,n ,0,\b ]$ (Lemme \ref{intersection}).

 Les hypoth\`eses du paragraphe \ref{2.4} sont v\'erifi\'ees par le sous-groupe parabolique $P$ et la suite de r\'eseaux 
$\mathfrak M_0$, ainsi que par $P^-$ et $\mathfrak M_1$   (\ref{plusplus}) et 
les repr\'esentations $\kappa$ et $\kappa_1$ sont reli\'ees par la condition :
   $$
   \mathbf r_{P}\left(\kappa  \right)= \mathbf r_{P} \left(
 \kappa^{(0)}\otimes \kappa_1\right) = \kappa^{(0)}\otimes \iota(\widetilde \kappa^{(1)})
 = \mathbf r_{P^-}\left(\kappa  \right)= \mathbf r_{P^-} \left(
 \kappa^{(0)}\otimes \kappa_1\right) .
 $$ 
D'apr\`es le th\'eor\`eme \ref{theoreme}, 
$\kappa \otimes \chi_0^{-1}$   est une $\beta$-extension   relative \`a $ \mathfrak M_{0  }$ 
si et seulement si 
$ \mathbf w_U^0 \chi_0^{-1}(\kappa^{(0)}\otimes \kappa_1  ) $ 
est une $\beta$-extension relative \`a $ \mathfrak M_{0  }$  ; de m\^eme 
  $\kappa \otimes \chi_1^{-1}$  est une $\beta$-extension   relative \`a    $ \mathfrak M_{1 }$ 
si et seulement si 
  $ \mathbf w_{U^-}^1 \chi_1^{-1}(\kappa^{(0)}\otimes \kappa_1  ) $  
est une $\beta$-extension relative \`a   $ \mathfrak M_{1 }$. 
Sur la composante en $W^{(0)}$ la torsion par $ \mathbf w_U^0$ ou $ \mathbf w_{U^-}^1$ ne change pas le fait d'\^etre une $\b$-extension car ces caract\`eres sont triviaux sur tous les sous-groupes unipotents \cite[Theorem 4.1]{S5}. La condition se simplifie donc comme suit : 

{\it 
 $\kappa \otimes \chi_0^{-1}$   est une $\beta$-extension   relative \`a $ \mathfrak M_{0  }$ 
si et seulement si $(\chi_0^{(0)})^{-1} \kappa^{(0)}$ est une $\b$-extension et 
$ \mathbf w_U^0 \chi_0^{-1}  \kappa_1    $ 
est une $\beta$-extension relative \`a $ \mathfrak M_{0  }^1$~; }

\noindent  et  de m\^eme 
en rempla\c cant $ \mathfrak M_{0  }$ par $ \mathfrak M_{1 }$ et $ \chi_0$ par  $\chi_1$. 
Notons au passage que $\chi_0$ et $\chi_1$ diff\`erent d'un caract\`ere autodual  
 \cite[Remark 6.12]{S5} tandis que $ \mathbf w_U^0$ 
 et  $\mathbf w_{U^-}^1 $ sont quadratiques ou triviaux. 
Finalement : 

\begin{Proposition}\label{w0}
Notons  $ \ \gamma_0= \iota(\rho^{(1)}  \otimes \chi_0^{(1)})$ et $ \ \gamma_1= \iota(\rho^{(1)}  \otimes \chi_1^{(1)})$. Les g\'en\'erateurs de 
 $\mathcal H(G, \lambda_P )$ se calculent dans 
$$\mathcal H(P  (  \mathfrak M_{0 ,\oE}^1 )/ P^1  (  \mathfrak M_{0 ,\oE}^1 ), \ \gamma_0 )
  \ \text{  et } \  
   \mathcal H(P  (  \mathfrak M_{1 ,\oE}^1 )/ P^1 (  \mathfrak M_{1 ,\oE}^1 ), \ \gamma_1 )$$ 
   tandis que ceux de $\mathcal H(G_1, \l_{P_1} )$ se calculent dans 
   $$\mathcal H(P  (  \mathfrak M_{0 ,\oE}^1 )/P^1(  \mathfrak M_{0 ,\oE}^1 ), \  \gamma_0  \ \mathbf w_U^0)
  \ \text{  et } \  
   \mathcal H(P  (  \mathfrak M_{1 ,\oE}^1 )/P^1(  \mathfrak M_{1 ,\oE}^1 ),  \  \gamma_1 \  \mathbf w_{U^-}^1),$$  
  o\`u   le   caract\`ere   $\mathbf w_U^0$ (resp.   $\mathbf w_{U^-}^1$)  de $P (\L_\oE^1 )/P^1(\L_\oE^1 )$ 
  associe   \`a $x \in P (\L_\oE^1  )  $ 
  la signature de la permutation  $v \mapsto xvx^{-1}$ de $ \ J^1_\ext(\mathfrak M_0^1) \cap U^- / \ H^1(\mathfrak M_0^1) \cap U^-$ (resp.   $ \ J^1_\ext(\mathfrak M_1^1) \cap U  / \ H^1(\mathfrak M_1^1) \cap U $). 
\end{Proposition}

\begin{Remark}\label{sigmapi} {\rm 
Bien entendu ce r\'esultat  ne permet pas de traiter la r\'eductibilit\'e de 
$ \  \Ind_P^G \, \pi |\det\!|^s \otimes \sigma $ pour n'importe quelle paire 
$(\pi, \sigma)$, puisque les types attach\'es \`a $\pi$ et $\sigma$ ont \'et\'e 
construits \`a partir des \'el\'ements suivants  : 
\begin{itemize}
	\item 
  une strate semi-simple gauche 
$[\L ,n ,0,\b ]$ 
et   une d\'ecomposition autoduale  exactement subordonn\'ee \`a cette strate $V = W^{(-1)} \oplus W^{(0)}\oplus W^{(1)} $    telle que $V^1 = W^{(-1)}  \oplus W^{(1)} $  ; 
\item 
  un caract\`ere semi-simple gauche $\theta$ dont la restriction $\theta^{(0)} \otimes \iota(\tilde \theta^{(1)})$  \`a 
$H^1 (\L)\cap M$ 
d\'etermine   un caract\`ere semi-simple gauche $ \theta^{(0)} $ intervenant dans 
$\sigma$ et un caract\`ere  simple  $\tilde  \theta^{(1)} $ intervenant dans 
$\pi$. 
\end{itemize}
On ne peut l'appliquer qu'\`a  des paires de repr\'esentations poss\'edant  
  des strates et caract\`eres semi-simples sous-jacents  au-dessus desquels il existe 
  une strate semi-simple et un caract\`ere semi-simple 
  convenablement d\'ecompos\'es comme ci-dessus, ce qui peut imposer des conditions de compatibilit\'e entre les caract\`eres $\theta^{(0)}$ et $ \iota(\tilde \theta^{(1)})$ 
  (voir la d\'efinition des caract\`eres semi-simples   \cite[Definition 3.13]{S4} et la notion de ``common approximation'' de \cite[\S 8]{BK2}). 
  
  Pr\'ecisons ceci, avec les notations du paragraphe \ref{pc}. On choisit un type pour la repr\'esen\-ta\-tion supercuspidale autoduale $\pi$ avec strate simple sous-jacente $[\L^{(1)} ,n^{(1)} ,0,2\b^{(1)} ]$ et ca\-rac\-t\`ere   simple  sous-jacent  $\tilde  \theta^{(1)} $. On peut  sommer la strate $[\L^{(1)} ,n^{(1)} ,0, \b^{(1)} ]$  et la strate duale dans 
  $ W^{(-1)} $ de fa\c con \`a obtenir une strate simple gauche $[\L^1 ,n^1 ,0,\b^1 ]$ 
  dans  $V^1$ \`a laquelle  $V^1 = W^{(-1)}  \oplus W^{(1)} $ soit exactement subordonn\'ee 
  \cite{B} \cite{GKS}, avec $\b^1 = \b^{(-1)} \oplus \b^{(1)}$. On choisit par ailleurs 
  un type pour $\sigma$ avec strate semi-simple gauche sous-jacente  $[\L^{(0)} ,n^{(0)} ,0, \b^{(0)} ]$ et caract\`ere semi-simple gauche $ \theta^{(0)} $. On se ram\`ene au cas o\`u 
 $\L^1$ et $\L^{(0)}$ ont m\^eme p\'eriode. La proposition \ref{w0} s'applique alors dans les cas suivants : $n^1= n^{(0)} $ et les polyn\^omes ca\-rac\-t\'eristiques des strates $[\L^1 ,n^1 ,0,\b^1 ]$ et  $[\L^{(0)} ,n^{(0)} ,0, \b^{(0)} ]$  sont premiers entre eux ; 
 $n^1 = 0$ et  $\b^{(0)}$ est inversible ; $n^{(0)} = 0$ et  $\b^1$ est inversible. 
 Si les strates sont moins ``disjointes'', un examen plus approfondi est n\'ecessaire pour d\'eterminer si  $\tilde  \theta^{(1)} $ et $ \theta^{(0)} $ v\'erifient les conditions d'existence d'un caract\`ere $\theta$ comme ci-dessus. 
  
  En particulier ce r\'esultat s'applique toujours si la strate sous-jacente \`a $\pi$ 
  est nulle et   $\b^{(0)}$  inversible. En revanche le cas d'un caract\`ere autodual $\pi$ de $GL(1,F)$ et d'une repr\'esentation $\sigma$ attach\'ee \`a une strate semi-simple dont une composante simple est nulle 
  (ce qui est automatique dans un groupe sp\'ecial orthogonal impair) ne rel\`eve pas de la proposition \ref{w0}. 
  }
 \end{Remark}

 \subsection{Exemples}\label{exemples}  
 Commen\c cons par  illustrer  la signification de la proposition \ref{w0} \`a partir 
    d'exemples bas\'es sur les calculs de paires couvrantes et d'alg\`ebres de Hecke d\'ej\`a faits dans $\Sp(4,F)$ \cite{BB} ; le groupe   $G$ est ici   symplectique. 
    
Le cas des repr\'esentations supercuspidales autoduales $\pi$ de $\GL(2,F)$ vu comme sous-groupe de Levi de Siegel de $\Sp(4,F)$  \cite[Tableau (3.17)]{BB} est conforme \`a l'exemple \ref{exemple} ci-dessus. 
 Passons  au cas le plus int\'eressant, celui d'un caract\`ere quadratique ou tri\-vial $\tilde \chi$ de $\GL(1,F)$. Notons $\chi$ la restriction de $\tilde \chi$ \`a $\oFx$. La paire couvrante $(J_{P_1}(\L_1), \l_{P_1})$ dans $G_1 =\Sp(2,F)$  corres\-pondante est  tout simplement la paire $(I, \chi_0)$ o\`u 
 $I =\left(\begin{smallmatrix} \oFx & \oF \cr \pF & \oFx 
 \end{smallmatrix}\right) \cap G_1 $ est un sous-groupe d'Iwahori et 
 $\chi_0\left( \left(\begin{smallmatrix} a&b\cr c&d
 \end{smallmatrix}\right)\right)  = \chi(a),  \left(\begin{smallmatrix} a&b \cr c&d 
 \end{smallmatrix}\right)\in I$. Les deux sous-groupes parahoriques maxi\-maux contenant $I$ sont 
 $\mathcal P_0=\SL(2, \oF)$ et $\mathcal P_1= \left(\begin{smallmatrix} \oFx & \pF^{-1} \cr \pF & \oFx 
 \end{smallmatrix}\right) \cap G_1 $. La seule $\beta$-extension du caract\`ere trivial de 
$I^1 =  \left(\begin{smallmatrix} 1+ \pF & \oF \cr \pF & 1+ \pF
 \end{smallmatrix}\right) \cap G_1 $ \`a $I$ relativement \`a l'un ou l'autre de ces parahoriques est 
 le caract\`ere trivial. Notant encore $\chi_0$ le caract\`ere du sous-groupe parabolique image de $I$ dans le quotient r\'eductif de $P_0$ ou $P_1$, on voit par le corollaire \ref{dernier} 
 que les relations quadratiques d\'eterminant la r\'eductibilit\'e sont deux fois celle satisfaite par un g\'en\'erateur de $ \mathcal H(\SL(2, k_F), \chi_0 )$, soit 
 $(X+1)(X-q)=0 \ (A)$ ($r_0=r_1=1$) si $\chi$ est trivial sur $\oFx$, $(X+1)(X-1)=0 \ (B)$ ($r_0=r_1=0$) sinon, d'o\`u la structure bien connue pour $ \mathcal H(\SL(2,  F), \chi_0  )$ et les points de r\'eductibilit\'e 
 de parties r\'eelles $0$ et $1$ dans le premier cas, $0$ quatre fois dans le second, pour la repr\'esentation $ \  \Ind_{P_1 }^{\Sp(2,F)} \, \tilde\chi | \ |^s  $.  
 
 Occupons-nous maintenant des repr\'esentations induites 
 $ \  \Ind_{P }^{\Sp(4,F)} \, \tilde\chi | \ |^s \otimes \sigma $ o\`u $\sigma$ est une repr\'esentation supercuspidale de $\Sp(2,F)$. 
 L'examen des tableaux (2.17) p. 676 et 1 p. 677 de \cite{BB} appelle quelques remarques. 
\begin{enumerate}
	\item 
	Lorsque $\sigma$ est une repr\'esentation de la s\'erie non ramifi\'ee ``exceptionnelle'', i.e. induite \`a partir de l'inflation \`a $\SL(2,\oF)$ d'une repr\'esentation cuspidale de $\SL(2,k_F)$ qui ne se prolonge pas \`a $\GL(2,k_F)$, 
	et si $\chi$ n'est pas trivial, on trouve une relation dont le quotient des valeurs propres est $-q^2$. De fait, dans ce cas la strate sous-jacente \`a $\sigma$ est une strate nulle, la strate sous-jacente  \`a $\pi= \tilde \chi$ aussi, et nous ne sommes pas dans les conditions d'application de la proposition \ref{w0} (voir la remarque \ref{sigmapi}). 
	C'est pourquoi cette proposition ne permet de pr\'evoir ni ce cas, ni les autres cas dits ``de niveau $0$'' o\`u $\sigma$ est  induite   de l'inflation  d'une repr\'esentation cuspidale de $\SL(2,k_F)$ qui  se prolonge  \`a $\GL(2,k_F)$. 
	\item 
	Si $\chi$ est trivial, on trouve deux relations de type (A) si $\sigma$ est   de la s\'erie non ramifi\'ee et de niveau strictement positif, deux relations de type (B) si $\sigma$ est   de la s\'erie   ramifi\'ee~; si $\chi$ est non trivial, c'est exactement le contraire. Ces cas rel\`event de la proposition \ref{w0}, les 
  caract\`eres  $\mathbf w_U^0  $ et 
  $ \mathbf w_{U^-}^1$ sont  n\'ecessairement triviaux dans le cas non ramifi\'e et non triviaux dans le cas ramifi\'e.  
\end{enumerate}

 \medskip 
 
 Dans un groupe symplectique on a calcul\'e dans l'exemple \ref{exemplebeta} les caract\`eres 
 $\mathbf w_U^0$ et  $\mathbf w_{U^-}^1$
  dans le cas d'un 
   caract\`ere autodual $\tilde \chi$ de $\GL(1,F)$ et d'une repr\'esentation  $\sigma$  
   de $\Sp(2N, F)$  attach\'ee \`a un type simple maximal o\`u l'\'el\'ement $\b^{(0)}$ 
 engendre une extension totalement ramifi\'ee 
  et y est  de valuation congrue \`a $-1$ modulo 
  $\dim W^{(0)}$. Ces deux caract\`eres sont \'egaux au caract\`ere quadratique non trivial 
  de $\oFx$ (comme dans (ii) pour $N=1$). 
  La proposition \ref{w0} nous permet de comparer comme ci-dessus la r\'eductibilit\'e de 
  $ \  \Ind_{P  }^{\Sp(2N+2,F)} \,  \tilde \chi  | \ |^s  \otimes \sigma $ 
  \`a celle de  
  $ \  \Ind_{P_1 }^{\Sp(2,F)} \,  \tilde \chi  | \ |^s  $. On en d\'eduit 
     qu'il y a un unique   caract\`ere autodual $\tilde \chi$  de $F^\times$    tel 
    que  $ \  \Ind_{P  }^{\Sp(2N+2,F)} \,  \tilde \chi  | \ |^s  \otimes \sigma $ soit r\'eductible en $s=1$, et   que la restriction de ce caract\`ere \`a $\oFx$ est non triviale. Il s'agit donc, comme on s'y attend  par ailleurs, d'un caract\`ere quadratique ramifi\'e. 
 Cependant (voir remarque  \ref{piexacte}), la d\'etermination exacte de ce caract\`ere 
 parmi les deux caract\`eres quadratiques ramifi\'es n\'ecessite des calculs plus pouss\'es.

\begin{biblio}{}

\bibitem{BB}
L. Blasco et C. Blondel, {\it
Alg\`ebres de Hecke et s\'eries principales
g\'en\'era\-lis\'ees de $Sp_4(F)$}, \rm
\
Proc.   London Math. Soc.  \textbf{85(3)}, 2002, 659--685. \rm

\bibitem{B}
C. Blondel,  {\it  Sp(2N)-covers for self-contragredient supercuspidal representations of GL(N)}, \rm
Ann. scient. Ec. Norm. Sup. \textbf{37}, 2004, 533--558.

\bibitem{B2}
C. Blondel,  {\it Covers and propagation in symplectic groups}, \rm Functional analysis IX,  16--31, Various Publ. Ser. (Aarhus), 48, Univ. Aarhus, Aarhus, 2007.

\bibitem{BS} C. Blondel and S. Stevens, {\it Genericity of supercuspidal representations of $p$-adic $\Sp4$},
Compositio Math. \textbf{145}(1), 2009,  213-246.

\bibitem{BK} C.J.Bushnell and P.C.Kutzko, {\rm The admissible dual of
{${\rm GL}(N)$} via compact open subgroups}, Annals of Mathematics
Studies 129, Princeton, 1993.

\bibitem{BK1} C.J.Bushnell and P.C.Kutzko, {\it  Smooth representations of reductive $p$-adic   groups: structure theory via types},  Proc. London Math. Soc.     \textbf{77}, 1998,   582--634.

\bibitem{BK2} C.J.Bushnell and P.C.Kutzko, {\it Semisimple types in $\GL_n$},
Compositio Math. \textbf{119}, 1999, 53--97.

\bibitem{GT} W.T. Gan  and S.  Takeda, {\it The Local Langlands Conjecture for $Sp(4)$},
Int. Math. Res. Not., 2010.

\bibitem{G} P. G\'erardin,  {\it Weil representations associated to finite fields},
J. of Algebra  \textbf{46}, 1977, 54--101.

\bibitem{GKS}    D. Goldberg,  P. Kutzko   and   S. Stevens,
{\it Covers for self-dual supercuspidal representations of the Siegel Levi subgroup of
classical p-adic groups},    Int. Math. Res. Not., 2007.

\bibitem{HL} R. Howlett and G. Lehrer, {\it Induced cuspidal representations and generalised Hecke rings}, Invent. Math. \textbf{58}, 1980, 37--64.

\bibitem{J} C. Jantzen, {\it Discrete series for $p$-adic $    \SO(2n)   $
and restrictions of representations of $   {\mathrm O}(2n)$}, to appear.

\bibitem{KM} P. Kutzko and L. Morris, {\it Level zero Hecke algebras and parabolic induction : the Siegel case for split classical groups}, Int. Math. Res. Not., 2006.

\bibitem{Mg}
C. M\oe glin, {\it Normalisation des op\'erateurs d'entrelacement et r\'eductibilit\'e des induites de  cuspidales ; le cas des groupes classiques p-adiques}, Ann. of Math.  \textbf{151}, 2000, 817--847.

\bibitem{M}
L. Morris, {\it Tamely ramified intertwining algebras}, Invent. Math. \textbf{114}, 1993, 1-?54.

\bibitem{N} M. Neuhauser,  {\it An explicit construction of the metaplectic representation over a finite field}, J. of Lie Theory \textbf{12}, 2002, 15--30.

\bibitem{Sh} F. Shahidi, {\it   A proof of Langlands conjecture on Plancherel measure; complementary series for p-adic
groups},  Ann. of Math. \textbf{132}, 1990, 273--330.

\bibitem{S}
A. Silberger, {\it Special representations of reductive p-adic groups are not integrable}, Ann. of Math.
\textbf{111}, 1980, 571--587.

 \bibitem{S1} S. Stevens, {\it Double coset decomposition and intertwining},
manuscripta math.  \textbf{106}, 2001, 349--364.

\bibitem{S4} S.Stevens, {\it Semisimple characters for {$p$}-adic
classical groups}, Duke Math.\ J.\  \textbf{127(1)}, 2005, 123--173.

\bibitem{S5} S.Stevens, {\it The supercuspidal representations of {$p$}-adic classical groups}, Invent. Math.  \textbf{172}, 2008, 289--352.

\bibitem{Sz} F. Szechtman,  {\it Weil representations of the symplectic group},
J. of Algebra \textbf{208}, 1998, 662--686.

\bibitem{Z} Y. Zhang,  {\it Discrete series of classical groups},
Canad. J. Math.  \textbf{52}(5), 2000, 1101--1120.

\end{biblio}

\bigskip

\small

Corinne Blondel

 \indent
Groupes, repr\'esentations et g\'eom\'etrie - Case 7012

\indent
C.N.R.S. - Institut de Math{\'e}matiques de Jussieu - UMR 7586

\indent
Universit{\'e} Paris 7

\indent
F-75205 Paris Cedex 13

\indent
France

\smallskip
Corinne.Blondel@math.jussieu.fr

\end{document}